\documentclass[a4paper,reqno,11pt]{amsart}
\usepackage{amsmath, amsfonts, amssymb, amsthm, amscd}
\usepackage{perpage}
\usepackage{url}
\usepackage{color}
\usepackage[a4paper,scale={0.73,0.74},marginratio={1:1},footskip=7mm,headsep=10mm]{geometry}
\usepackage{dsfont}
\usepackage{hyperref}

\numberwithin{equation}{section}
\newtheorem{theorem}{Theorem}[section]
\newtheorem{lemma}[theorem]{Lemma}
\newtheorem{proposition}[theorem]{Proposition}

\newtheorem{remark}[theorem]{Remark}
\newtheorem{definition}[theorem]{Definition}

\newcommand{\cE}{{\ensuremath{\mathcal E}} }

\newcommand{\R}{\mathbb{R}}

\newcommand{\N}{\mathbb{N}}

\newcommand{\PEfont}{\mathrm}

\renewcommand\P{\ensuremath{\PEfont P}}

\DeclareMathOperator{\E}{\ensuremath{\PEfont E}}

\newcommand{\ind}{\mathds{1}}
\renewcommand{\epsilon}{\varepsilon}
\renewcommand{\theta}{\vartheta}
\renewcommand{\rho}{\varrho}
\renewcommand{\phi}{\varphi}

\newenvironment{myenumerate}{\renewcommand{\theenumi}{\arabic{enumi}}\renewcommand{\labelenumi}{{\rm(\theenumi)}}\begin{list}{\labelenumi}
	{	\setlength{\itemsep}{0.4em}	 \setlength{\topsep}{0.5em}\setlength{\partopsep}{0em}\setlength{\parsep}{0em}\setlength{\parskip}{0em}	 \setlength\leftmargin{2.45em}	\setlength\labelwidth{2.05em}	\setlength{\labelsep}{0.4em}	\usecounter{enumi}	}	 }{\end{list}
}
\renewenvironment{enumerate}{
\begin{myenumerate}}{\end{myenumerate}}
\newenvironment{myitemize}{\begin{list}{$\bullet$} 	{	\setlength{\itemsep}{0.4em}	 \setlength{\topsep}{0.5em}\setlength{\partopsep}{0em}\setlength{\parsep}{0em}\setlength{\parskip}{0em}	 \setlength\leftmargin{2.45em}	 \setlength\labelwidth{2.05em}	\setlength{\labelsep}{0.4em}\usecounter{enumi}	}	 }{\end{list}}
\renewenvironment{itemize}{
\begin{myitemize}}{\end{myitemize}}

\MakePerPage[2]{footnote}

\newcommand\dd{\mathrm{d}}
\input{tcilatex}

\begin{document}
\title[The strong renewal theorem]{Local large deviations\\
and the strong renewal theorem}
\author{Francesco Caravenna}
\address{Dipartimento di Matematica e Applicazioni, Universit\`a degli Studi
di Milano-Bicocca, via Cozzi 55, 20125 Milano, Italy}
\email{francesco.caravenna@unimib.it}

\author{Ron Doney}
\address{School of Mathematics, University of Manchester, Oxford Road, Manchester M13 9PL, UK}
\email{ron.doney@manchester.ac.uk}

\date{\today }

\begin{abstract}
We establish two different, but related results for
random walks in the domain of attraction of a stable
law of index $\alpha$. The first result is a local large deviation upper bound,
valid for $\alpha \in (0,1) \cup (1,2)$, which improves on the classical Gnedenko
and Stone
local limit theorems. The second result, valid for $\alpha \in (0,1)$,
is the derivation of necessary and sufficient conditions
for the random walk to satisfy the \emph{strong renewal theorem} (SRT). 
This solves a long-standing problem,
which dates back to the 1962 paper of Garsia and Lamperti
\cite{cf:GL} for renewal processes (i.e.\ random walks with non-negative
increments), and to the 1968 
paper of Williamson \cite{cf:Wil} for general random walks.
\end{abstract}

\maketitle

\section{Introduction and results}

This paper contains new results about asymptotically stable random walks.
We first present a local large deviation
estimate which improves the error term in the classical local limit theorems,
without making any further assumptions (see Theorem~\ref{th:ron}).
Then we exploit this bound to solve a long-standing
problem, namely we establish necessary and sufficient conditions
for the validity of the \emph{strong renewal theorem} (SRT),
both for renewal processes (Theorem~\ref{th:main}) and for
general random walks (Theorem~\ref{th:mainrw}). The corresponding
result for L\'evy processes is also presented (see Theorem~\ref{th:Levy}).

\smallskip

This paper supersedes the individual preprints \cite{cf:C} and \cite{cf:Don}.

\medskip
\noindent
\emph{Notation.}
We set $\mathbb{N}=\{1,2,3,\ldots \}$ and $\mathbb{N}_{0}=\mathbb{N}\cup \{0\}$. 
We denote by $RV(\gamma)$ the class of regularly
varying functions with index $\gamma $, namely $f\in RV(\gamma)$ if and only if
$f(x)=x^{\gamma }\ell (x)$ for some slowly varying function $\ell \in RV(0)$, see~\cite{cf:BinGolTeu}.
Given $f,g:[0,\infty )\rightarrow
(0,\infty )$ we write $f\sim g$ to mean $\lim_{s\to\infty}
f(s)/g(s)=1$, and $f \ll g$ to mean $\lim_{s\to\infty}
f(s)/g(s)=0$.

\subsection{Local large deviations}
\label{sec:LLD} 

Let $(X_i)_{i\in\mathbb{N}}$ be i.i.d.\ real-valued
random variables, with law $F$. Let $S_0 := 0$, $S_n := X_1 + \ldots + X_n$
be the associated random walk and
\begin{equation}  \label{eq:max}
M_n := \max\{X_1, X_2, \ldots, X_n\} \,.
\end{equation}

We assume that the law $F$ is in the domain of attraction of a strictly stable law with index
$\alpha \in (0,1) \cup (1,2)$, that is, with $\overline{F}(x) :=
F((x,\infty))$ and $F(x) := F((-\infty,x])$, 
\begin{equation}  \label{eq:tail2}
\begin{split}
&\overline{F}(x) \,\underset{x\to\infty}{\sim}\, \frac{p}{A(x)} 
\qquad \text{and} \qquad F(-x) \,\underset{x\to\infty}{\sim}\,
\frac{q}{A(x)} \qquad \text{for some } \ A \in RV(\alpha) \,.
\end{split}%
\end{equation}
More explicitly, if we write $A(x) = x^{\alpha} / L(x)$, with
$L(\cdot)$ slowly varying,
\begin{equation*}
\begin{split}
&\P(X > x) \,\underset{x\to\infty}{\sim}\, p \, \frac{L(x)}{x^{\alpha}} 
\qquad \text{and} \qquad \P(X \le -x) \,\underset{x\to\infty}{\sim}\,
q \, \frac{L(x)}{x^{\alpha}}  \,.
\end{split}%
\end{equation*}
We assume that $p > 0$ and $q \ge 0$
(when $q=0$, the second relation in \eqref{eq:tail2}
should be understood as $F(-x) = o(1/A(x))$).
For $\alpha > 1$, we further assume that $\E[X] = 0$.

\smallskip

Without loss of generality, we may assume that $A \in RV(\alpha)$ is
continuous and strictly increasing.
If we introduce the norming sequence $a_n \in RV(1/\alpha)$ defined by
\begin{equation}  \label{eq:an0}
a_n := A^{-1}(n) \,, \qquad n \in \mathbb{N} \,,
\end{equation}
then $S_n / a_n$ converges in law to a random variable $Y$ with a stable law of
index $\alpha$ and
positivity parameter $\rho = \P(Y > 0) = \frac{1}{2} + \frac{1}{%
\pi\alpha} \arctan (\frac{p-q}{p+q} \tan \frac{\pi \alpha}{2}) > 0$
(because $p > 0$).

\smallskip

Our first main result is a local large deviation estimate for $S_n$, constrained on $M_n$. 

\begin{theorem}[Local Large Deviations]
\label{th:ron} Let $F$ satisfy \eqref{eq:tail2} with $\alpha \in (0,1) \cup
(1,2)$ and $p > 0$, 
and $\E[X]=0$ if $\alpha > 1$. Fix a bounded measurable 
$J \subseteq \mathbb{R}$. Given $\gamma \in (0,\infty)$, there is $C_0 = C_0(\gamma,J) <
\infty$ such that, for all $n\in\mathbb{N}$ and $x \ge 0$, the following
relation holds:
\begin{gather}  \label{eq:ron1}
\P (S_n \in x+J, \ M_n \le \gamma x) \le C_0 \, \frac{1}{a_n} \left(\frac{n}{%
A(x)} \right)^{\lceil 1/\gamma \rceil}\,,
\end{gather}
where $\lceil x \rceil := \min\{n\in\N: \ n \ge x\}$ is the upper integer part of $x$.
More explicitly:
\begin{equation} \label{eq:ron1modk}
	\forall k \in \N, \ \ \forall \gamma \in [\tfrac{1}{k}, \tfrac{1}{k-1}): \qquad
	\P(S_n \in x + J, \ M_n \le \gamma x) \le C_0 \,
	\frac{1}{a_n} \left(\frac{n}{%
	A(x)} \right)^{k}\,.
\end{equation}%
Moreover, for some $C_0^{\prime }= C_0^{\prime }(J) < \infty$,
\begin{gather}  \label{eq:ron2}
\P (S_n \in x+J) \le C_0^{\prime }\, \frac{1}{a_n} \, \frac{n}{A(x)} \,.
\end{gather}
\end{theorem}

\smallskip

The non-local version of \eqref{eq:ron1}, where $S_n \in x+J$ is replaced
by $S_n \ge x$, is known as a \emph{Fuk-Nagaev inequality}~\cite{cf:N}.
This is the starting point of our proof of Theorem~\ref{th:ron}, 
see Section~\ref{sec:newproof}. We prove \eqref{eq:ron1} through direct
path estimates, combined with local limit theorems.
Relation \eqref{eq:ron2} is obtained as a simple corollary
of \eqref{eq:ron1} with $\gamma = 1$.

\smallskip

A heuristic explanation of \eqref{eq:ron2} goes as follows:
for large $x$, if $S_n \in x+J$,
it is likely that \emph{a single step $X_i$ takes a value $y$ comparable to $x$}.
Since $\P(X_i > c x) \approx 1/A(x)$ by \eqref{eq:tail2},
and since there are $n$ available steps, we get
the factor $n/A(x)$ in \eqref{eq:ron2}.
The extra factor $1/a_n$ comes from Gnedenko and Stone local limit theorems.

A similar argument sheds light on \eqref{eq:ron1}-\eqref{eq:ron1modk}.
Under the
constraint $M_n \le \gamma x$, with $\gamma \in [\frac{1}{k}, \frac{1}{k-1})$,
the most likely way to have $S_n \in x+J$ is that \emph{exactly $k$ steps
$X_{i_1}, \ldots, X_{i_k}$ take values comparable to $x/k$}, and this
yields the factor $(n/A(x))^k$ in \eqref{eq:ron1modk}.

\begin{remark}
The classical Gnedenko and Stone local limit theorems
only give the weak bound $\P (S_n \in x+J) = o( \frac{1}{a_n})$
as $x / a_n \to \infty$.
The inequality \eqref{eq:ron2} improves quantitatively on this bound,
with no further assumptions besides \eqref{eq:tail2}.
\end{remark}

The Cauchy case $\alpha = 1$ is left out from our analysis, because
of the extra care needed to handle the centering issues.
However, an analogue of Theorem~\ref{th:ron} holds also in this case,
as shown by Q. Berger in the recent paper \cite{B17}.

Finally, it is worth stressing that the estimate \eqref{eq:ron2} 
is essentially optimal, under the mere assumption \eqref{eq:tail2}. 
However, if one makes extra local requirements
on the step distribution, such as e.g.\ \eqref{eq:D97} below,
one can correspondingly sharpen \eqref{eq:ron2} along the same
line of proof, see \cite[Theorem~2.4]{B17} 
(which is valid for any $\alpha \in (0,2)$).

\subsection{The strong renewal theorem}

Henceforth we assume that $\alpha \in (0,1)$. We say that $F$ is \emph{%
arithmetic} if it is supported by $h\mathbb{Z}$ for some $h > 0$, in which
case the maximal value of $h>0$ with this property is called the \emph{%
arithmetic span} of $F$. It is convenient to set
\begin{equation}  \label{eq:I}
I = (-h,0] \qquad \text{where} \qquad h :=
\begin{cases}
\text{arithmetic span of $F$} & \text{(if $F$ is arithmetic)} \\
\text{any fixed number $> 0$} & \text{(if $F$ is non-arithmetic)} \,.%
\end{cases}%
\end{equation}

The renewal measure $U(\cdot)$ associated to $F$ is
the measure on $\R$ defined by
\begin{equation}  \label{eq:U}
U(\mathrm{d} x) := \sum_{n\ge 0} F^{*n}(\mathrm{d} x) = \sum_{n\ge 0} \P %
(S_n \in \mathrm{d} x) \,.
\end{equation}
It is well known 
(see \cite[Eq. (8.6.1)-(8.6.3)]{cf:BinGolTeu}  and \cite[Appendix]{cf:Chi0})
that \eqref{eq:tail2} implies 
\begin{equation}  \label{eq:SRTint}
U([0,x]) \,\underset{x\to\infty}{\sim}\, \frac{\mathsf{C}}{\alpha} \, A(x)\,, \qquad
\text{with} \qquad \mathsf{C} = \mathsf{C}(\alpha,\rho) = \alpha
\E[Y^{-\alpha} \, \ind_{\{Y
> 0\}}]
\end{equation}
(recall that $Y$ denotes a random variable with the limiting stable law).
In the special case
when $p=1$ and $q=0$ in \eqref{eq:tail2}
(so that $\rho = 1$) one has $\mathsf{C} = \frac{1}{\pi}
\sin(\pi\alpha)$.

It is natural to wonder whether the local version of \eqref{eq:SRTint}
holds, namely
\begin{equation}
U(x+I)=U((x-h,x]) \,\underset{x\to\infty}{\sim}\, \mathsf{C}\,h\,\frac{A(x)}{x} \,.  \tag{SRT}  \label{eq:SRT}
\end{equation}%
For a more usual formulation,
we can write $A(x) = x^{\alpha} / L(x)$ with
$L(\cdot)$ slowly varying: 
\begin{equation}
	U(x+I)=U((x-h,x]) \,\underset{x\to\infty}{\sim}\, \mathsf{C}\,h\,\frac{1}{L(x)
	\, x^{1-\alpha}} \,.
	\tag{SRT}
\end{equation}
This relation, called \emph{strong renewal theorem} \eqref{eq:SRT},
is known to follow from \eqref{eq:tail2} when $\alpha >\frac{1}{2}$, see \cite%
{cf:GL,cf:Wil,cf:Eri,cf:Eri2}. However,
when $\alpha \leq \frac{1}{2}$ there are examples of
$F$ satisfying \eqref{eq:tail2} but not \eqref{eq:SRT}.
The reason is that small values of $n$
in \eqref{eq:U} can give an anomalous contribution
to the renewal measure (see Subsection~\ref{sec:refor}
for more details).

\smallskip

In order for the SRT to hold, when $\alpha \le \frac{1}{2}$,
extra assumptions are needed.
Sufficient conditions have been derived along the years
\cite{cf:Wil,cf:D,cf:VT,cf:Chi0,cf:Chi}, but none of these is necessary.
\emph{In this paper we settle this problem, 
determining necessary and sufficient conditions for the SRT}:
see Theorem~\ref{th:main} for renewal processes and Theorem~\ref{th:mainrw}
for random walks. We also obtain very explicit and sharp
sufficient conditions,
which refine those in the literature, see Propositions~\ref{pr:main1} and~\ref{th:suffrw}.

Besides its intrinsic interest,
the SRT for heavy tailed renewal processes
has played and is still playing a key role in a variety of contexts.
Our results are already referred to in several papers, from classical renewal theory 
\cite{cf:Chi18,cf:K,cf:K2} to random walks and large deviations
\cite{B17,B18,cf:DSW,cf:DW,cf:U},
from dynamical systems \cite{cf:DN,cf:DN18,cf:MT}
to interacting particle systems \cite{cf:FMMV}.
The SRT has also played a key role in applications, 
e.g. in pinning and related models of statistical mechanics,
see \cite{cf:Gia,cf:dH,cf:Gia2}.
We also point out that, as a future direction of research, 
our results are likely to lead to an ultimate
version of the \emph{key renewal theorem},
in the context of random walks with infinite mean
(see \cite{cf:Eri,cf:AA} for partial versions).

\smallskip

Let us proceed with our results.
For $k \ge 0$ and $x \in \mathbb{R}$ we set
\begin{equation}  \label{eq:b}
b_k(x) := \frac{A(|x|)^{k}}{|x| \vee 1} \,.
\end{equation}
Note that 
$b_1(x) = A(x)/x = (L(x) x^{1-\alpha})^{-1}$ for $x \ge 1$
is precisely the rate in the right hand side of \eqref{eq:SRT}.
In the sequel, we will often need to require that some quantity
$J(\delta;x)$ 
is  \emph{much smaller than $b_1(x)$, when $x \to \infty$
followed by $\delta \to 0$}. This leads to the following

\begin{definition}
\label{def:an} 
Throughout the paper we write
``$J(\delta; x)$ is a.n.'' to mean that a function $J(\delta; x)$
 is \emph{asymptotically negligible}
with respect to $b_1$, in the following precise sense:
\begin{equation}  \label{eq:an}
\lim_{\delta \to 0} \, \limsup_{x\to+\infty} \frac{J(\delta; x)}{b_1(x)} =
\lim_{\delta \to 0} \, \limsup_{x\to+\infty} \frac{J(\delta; x)}{A(x)/x} =
0 \,.
\end{equation}
\end{definition}

We are ready to state our necessary and sufficient conditions for the SRT.
We start with the case of
renewal processes, which is simpler.

\subsection{The renewal process case}

Assume that $F$ is a law on $[0,\infty)$ such that
\begin{equation}  \label{eq:tail1}
\begin{split}
&\overline{F}(x) 
\,\underset{x\to\infty}{\sim}\, \frac{1}{A(x)} \qquad \text{for some } \ A \in RV(\alpha) \,,
\end{split}%
\end{equation}
which is a special case of \eqref{eq:tail2} with $p=1$, $q=0$.
For $\delta > 0$ and $x \ge 0$ we set
\begin{gather}  \label{eq:I1}
I_1^+(\delta; x) := \int_{1 \le z \le \delta x} F(x - \mathrm{d} z) \, b_2(z)
= \int_{1 \le z \le \delta x} F(x - \mathrm{d} z) \, \frac{A(z)^2}{z}
\,.
\end{gather}
The following is our main result for renewal processes.

\begin{theorem}[SRT for Renewal Processes]
\label{th:main} Let $F$ be a probability on $[0,\infty)$ satisfying %
\eqref{eq:tail1} with $\alpha \in (0, 1)$. Define $I = (-h,0]$ with $h > 0$
as in \eqref{eq:I}.

\begin{itemize}
\item If $\alpha > \frac{1}{2}$, the SRT holds with no extra assumption on $%
F $.

\item If $\alpha \le \frac{1}{2}$, the SRT holds if and only if $%
I_1^+(\delta;x)$ is a.n. (see Definition~\ref{def:an}).
\end{itemize}
\end{theorem}

\smallskip

Let us spell out the condition ``$I_1^+(\delta;x)$ is a.n.'' explicitly in terms of $F$,
by \eqref{eq:an}-\eqref{eq:I1}:
\begin{equation*}
	\text{``$I_1^+(\delta;x)$ is a.n.''} \ \quad \iff  \ \quad
	\lim_{\delta \to 0} \ \limsup_{x \to \infty} \
	x \, \overline{F}(x) \, \int_{1 \le z \le \delta x}
	F(x-\dd z) \, \frac{1}{z \, \overline{F}(z)^2} \,=\, 0 \,.
\end{equation*}
This can be checked in concrete examples, if one has enough
control on $F(\cdot)$. We will soon deduce more explicit sufficient conditions, see 
Proposition~\ref{pr:main1}, which are almost optimal.

\smallskip

Interestingly, in the ``boundary'' case $\alpha = \frac{1}{2}$, 
we can characterize the
class of $A(\cdot)$'s for which
the SRT holds with no extra assumption on $F$
besides \eqref{eq:tail1} (like for $\alpha > \frac{1}{2}$). 

\begin{theorem}[SRT for Renewal Processes with $\protect\alpha = \frac{1}{2}$]
\label{th:1/2} Let $F$ be a probability on $[0,\infty)$ satisfying %
\eqref{eq:tail1} with $\alpha = \frac{1}{2}$
(so that $A(x) / \sqrt{x}$ is
a slowly varying function). If
\begin{equation}  \label{eq:cond}
\sup_{1 \le s \le x} \frac{A(s)}{\sqrt{s}} \,
\underset{x\to\infty}{=} \, O\bigg( \frac{A(x)}{\sqrt{x}} \bigg) \,,
\end{equation}
then the SRT holds with no extra assumption on $F$.
(This includes the case $A(x) \sim c \sqrt{x}$.)
If condition \eqref{eq:cond} fails, there are examples of $F$ for
which the SRT fails.
\end{theorem}

\smallskip

The proof of Theorem~\ref{th:main}
is based on direct probabilistic arguments 
and is remarkably compact ($\simeq 6$ pages).
We start in Section~\ref{sec:usebou} recalling
a reformulation of the SRT, which can be
paraphrased as follows:
\emph{the contribution of ``small $n$'' 
to the renewal measure \eqref{eq:U} is asymptotically negligible} (see
Subsection~\ref{sec:refor}). In Section~\ref{sec:usebou}
we also derive two key bounds
on the contribution of ``big jumps'',
see Lemmas~\ref{th:kale3} and~\ref{th:nobig}. 
We complete the proof of Theorem~\ref{th:main}
in Subsection~\ref{sec:necren} (necessity) and in Section~\ref{sec:suffe}
(sufficiency).

\subsection{Sufficient conditions for renewal processes}
\label{sec:suffren}

For a probability $F$ on $[0,\infty)$ which satisfies
\eqref{eq:tail1}, a sufficient condition for the SRT is
that for some $x_0, C < \infty$ one has
\begin{equation} \label{eq:D97}
	F(x+I) \,\le\, \frac{C}{x\,A(x)}  \qquad \forall x \ge x_0 \,,
\end{equation}
as proved by Doney \cite{cf:D} in the arithmetic case
(extending previous results of Williamson \cite{cf:Wil}),
and by Vatutin and Topchii \cite{cf:VT} %
in the  non-arithmetic case.

\smallskip

Interestingly, if one only looks at the growth of the ``local'' probabilities $F(x+I)$,
\emph{no sharper condition than \eqref{eq:D97}
can ensure that the SRT holds}, as the following result shows.

\begin{proposition}\label{th:counter}
Fix $A \in RV(\alpha)$ with $\alpha \in (0,\frac{1}{2})$,
and let $\zeta: (0,\infty) \to (0,\infty)$ be an arbitrary non-decreasing function with
$\lim_{x\to\infty} \zeta(x) = \infty$.
Then there exists a probability $F$ on $[0,\infty)$ which satisfies \eqref{eq:tail1}, such
that $F(x+I) = O(\frac{\zeta(x)}{x A(x)})$, for which the SRT fails.
\end{proposition}

Intuitively, when condition \eqref{eq:D97} is \emph{not} satisfied, in order
for the SRT to hold, the points $x$ for which $F(x+I) \gg \frac{1}{xA(x)}$
must not be ``too cluttered''. We can
make this loose statement precise by
looking at the probabilities $F((x-y,x]) = F(x) - F(x-y)$.
The following result provides very explicit
conditions on $F(\cdot)$ for the SRT.

\begin{proposition}\label{pr:main1}
Let $F$ be a probability on $[0,\infty)$ satisfying \eqref{eq:tail1}
with $\alpha \in (0,\frac{1}{2}]$. 
\begin{itemize}
\item A sufficient condition for
the SRT is that 
for some $\gamma > 1-2\alpha$ and $x_0, C < \infty$ one has
\begin{equation} \label{eq:suff0}
	F((x-y, x])
	\,\le\,
	\frac{C}{A(x)} \, \Big(\frac{y}{x}\Big)^{\gamma} \qquad
	\forall x \ge x_0 \,, \ \ \forall y \in \big[ 1, \tfrac{1}{2} x \big] \,.
\end{equation}
\item A necessary condition for the SRT is that
for every $\gamma < 1-2\alpha$ there are $x_0, C < \infty$
such that \eqref{eq:suff0} holds.
\end{itemize}
\end{proposition}

\begin{remark}
The sufficient condition \eqref{eq:suff0} is a generalization of \eqref{eq:D97}.
Indeed, \eqref{eq:D97} yields
$F((x-y,x]) \le \sum_{j=0}^{\lceil y/h \rceil} F(x-hj+I)
\le (\frac{y}{h} + 2) \, \frac{C'}{xA(x)}$ for some $C'$, hence when $y \ge 1$
condition \eqref{eq:suff0} holds with $\gamma = 1$ and $C = (\frac{1}{h}+2)C'$
(recall that $h > 0$ is fixed, see \eqref{eq:I}).
\end{remark}

\begin{remark}
Other sufficient conditions for the SRT, which generalize and sharpen
\eqref{eq:D97}, were given by Chi in~\cite{cf:Chi0,cf:Chi}.
These can be deduced from Theorem~\ref{th:main}.
\end{remark}

\begin{remark}
Conditions similar to \eqref{eq:suff0}, in a different context, appear in \cite{cf:CSZ}.
\end{remark}

We point out that if $F$ satisfies \eqref{eq:tail1}, then
\eqref{eq:suff0} holds with $\gamma = 0$.
However, with no extra assumption,
one cannot hope to improve this estimate,
as Lemma~\ref{lem:uao} below shows.

\smallskip

To see how condition \eqref{eq:suff0} appears,
let us introduce the following variant of \eqref{eq:I1}:
\begin{equation} \label{eq:I1hat}
	\tilde I_1^+(\delta; x) := \int_1^{\delta x} \frac{F((x-z, x])}{z}
	\, b_2(z) \, \dd z = \int_1^{\delta x} \frac{F((x-z, x])}{z}
	\, \frac{A(z)^2}{z} \, \dd z \,.
\end{equation}
Our next result shows that one can look at $\tilde I_1^+(\delta; x)$
instead of $I_1^+(\delta; x)$.

\begin{proposition}\label{th:integr}
Let $F$ be a probability on $[0,\infty)$ satisfying \eqref{eq:tail1}
with $\alpha \in (0,\frac{1}{2}]$.
\begin{itemize}
\item If $\tilde I_1^+(\delta; x)$ is a.n., then also $I_1^+(\delta;x)$ is a.n.,
hence the SRT holds.

\item When $\alpha < \frac{1}{2}$, the converse is also true:
$\tilde I_1^+(\delta; x)$ is a.n.\ if and only if $I_1^+(\delta; x)$ is a.n..
\end{itemize}
\end{proposition}

\subsection{The general random walk case}

\label{sec:rw}

We now turn to the general random walk case, 
which is more challenging.
We assume that $F$ is a
probability on $\mathbb{R}$ which satisfies \eqref{eq:tail2} with $%
\alpha \in (0,1)$, $p>0$ and $q\geq 0$. 
Note that the associated random walk is \emph{transient},
because $a_n \in RV(1/\alpha)$ and then
$\sum_{n\in\N} \P(S_n \in (0,1]) \le \sum_{n\in\N}
\frac{C}{a_n}  < \infty$ (see \eqref{eq:sup} below).

Let us generalize \eqref{eq:I1} as
follows: for $\delta >0$ and $x\geq 0$ we set:
\begin{equation} \label{eq:I1rw}
I_{1}(\delta ;x):=\int_{|y|\leq \delta x}F(x+\mathrm{d}y)\,b_{2}(y)\,.
\end{equation}%
For $k\in\mathbb{N}$ with $k\ge 2$, we introduce a further parameter $\eta
\in (0,1)$ and we set
\begin{equation}  \label{eq:Ik}
\begin{split}
& I_k(\delta, \eta; x) := \int\limits_{|y_1| \le \delta x}
\!\! F(x + \mathrm{d} y_1)
\idotsint\limits_{|y_j| \le \eta |y_{j-1}| \; \text{for} \; 2 \le j \le k}
P_{y_1}(\mathrm{d} y_2, \ldots, \mathrm{d} y_k) \, b_{k+1}(y_k) \,, \\
& \rule{0pt}{1.2em} \text{where}
\quad P_{y_1}(\mathrm{d} y_2, \ldots, \mathrm{d} y_k) :=
F(-y_1 + \mathrm{d} y_2) F(- y_2 + \mathrm{d} y_3) \cdots F(- y_{k-1} +
\mathrm{d} y_k) \,.
\end{split}%
\end{equation}
Note that $P_{y_1}(\dd y_2, \ldots, \dd y_k)$ is the law of $(S_2, \ldots, S_k)$
conditionally on $S_1 = y_1$, hence 
\begin{equation}  \label{eq:Ikrepr}
I_k(\delta, \eta; x) = \E\big[ b_{k+1}(S_k) \, 
\mathds{1}_{|S_j| \le \eta
|S_{j-1}| \; \text{for} \; 2 \le j \le k\}} \, \mathds{1}_{|S_1| \le
\delta x\}} \,\big|\, S_0 = -x \big] \,.
\end{equation}
The same formula holds also for $k=1$ (where the first indicator function
equals $1$).

\smallskip

Let us define
\begin{equation}
\kappa _{\alpha }:=\bigg\lfloor\frac{1}{\alpha }\bigg\rfloor-1=%
\begin{cases}
0 & \text{if }\alpha \in (\frac{1}{2},1) \\
1 & \text{if }\alpha \in (\frac{1}{3},\frac{1}{2}] \\
2 & \text{if }\alpha \in (\frac{1}{4},\frac{1}{3}] \\
\,\vdots &  \\
m & \text{if }\alpha \in (\frac{1}{m+2},\frac{1}{m+1}]%
\end{cases}%
\,,  \label{eq:kappaalpha}
\end{equation}%
We are going to see that, when $1/\alpha \not\in \mathbb{N}$, 
necessary and sufficient conditions for
the SRT involve the a.n.\ of $I_{k}(\delta, \eta ;x)$ for $k=\kappa _{\alpha }$.
The case $1/\alpha \in \mathbb{N}$
is slightly more involved. 
We need to
introduce a suitable modification of \eqref{eq:b}, namely
\begin{equation}
\tilde{b}_{k}(z,x):=\tilde{b}_{k}(|z|,|x|)
:= \int_{|x|}^{|z|} \frac{b_k(t)}{t \vee 1} \, \dd t \,,
\label{eq:btilde}
\end{equation}%
where the integral vanishes if $|x|>|z|$.
We then define $\tilde{I}_{1}(\delta ;x)$ and $\tilde{I}_{k}(\delta ,\eta
;x) $ in analogy with \eqref{eq:I1rw} and \eqref{eq:Ik}, replacing $b_{2}(y)$
by $\tilde{b}_{2}(\delta x,y)$ and $b_{k+1}(y_{k})$ by $%
\tilde{b}_{k+1}(y_{k-1},y_{k})$:
\begin{align}
\tilde{I}_{1}(\delta ;x)& :=\int_{|y|\leq \delta x}F(x+\mathrm{d}y)\,\tilde{b%
}_{2}(\delta x,y)\,,
\label{eq:tildeI1}
\end{align}%
and for $k \ge 2$:
\begin{align}
\tilde{I}_{k}(\delta ,\eta ;x)& :=
\int\limits_{|y_{1}|\leq \delta x} \!\! F(x+\mathrm{d}%
y_{1})\!\!\!\!\idotsint\limits_{|y_{j}|\leq \eta |y_{j-1}|\;\text{for}%
\;2\leq j\leq k}\!\!\!P_{y_{1}}(\mathrm{d}y_{2},\ldots ,\mathrm{d}y_{k})\,%
\tilde{b}_{k+1}(y_{k-1},y_{k})\,.
\label{eq:tildeIk}
\end{align}%
Note that, by Fubini's theorem, we can equivalently rewrite
\eqref{eq:tildeI1} as follows:
\begin{equation}\label{eq:tildeI10}
	\tilde I_{1}(\delta; x) = 
	\int_{0}^{\delta x} \frac{F((x-t, x+t])}{t \vee 1}
	\, b_2(t) \, \dd t \,,
\end{equation}
which is a natural random walk generalization
of \eqref{eq:I1hat}.

\smallskip

We can now state our main result for random walks.

\begin{theorem}[SRT for Random Walks]
\label{th:mainrw} Let $F$ be a probability on $\mathbb{R}$ satisfying %
\eqref{eq:tail1} with $\alpha \in (0, 1)$ and with $p, q > 0$. Define $I =
(-h,0]$ with $h > 0$ as in \eqref{eq:I}.

\begin{itemize}
\item If $\alpha > \frac{1}{2}$, the SRT holds with no extra assumption on $%
F $.

\item If $\alpha \le \frac{1}{2}$ and $\frac{1}{\alpha} \not\in\mathbb{N}$,
we distinguish two cases:

\begin{itemize}
\item[-] if $\alpha \in (\frac{1}{3},\frac{1}{2})$, i.e.\ $\kappa_{\alpha} = 1$,
the SRT holds if and only if $I_1(\delta;x)$ is a.n..

\item[-] if $\alpha \in (\frac{1}{k+2}, \frac{1}{k+1})$
for some $k = \kappa_{\alpha} \ge 2$,
the SRT holds if and
only if $I_{\kappa_{\alpha}}(\delta, \eta;x)$ is a.n., for every fixed $\eta \in (0,1)$.
\end{itemize}

\item If $\alpha \le \frac{1}{2}$ and $\frac{1}{\alpha} \in\mathbb{N}$, the
same statement holds if we replace $I_k$ by $\tilde I_k$, namely:

\begin{itemize}
\item[-] if $\alpha = \frac{1}{2}$, i.e.\ $\kappa_{\alpha} = 1$, the SRT
holds if and only if $\tilde I_1(\delta;x)$ is a.n..

\item[-] if $\alpha = \frac{1}{k + 1}$, 
for some $k = \kappa_{\alpha} \ge 2$, the SRT holds if and only if $\tilde
I_{\kappa_{\alpha}}(\delta, \eta;x)$ is a.n.,
for every fixed $\eta \in (0,1)$.
\end{itemize}
\end{itemize}
\end{theorem}

\emph{We stress that the conditions that $I_k$ and $\tilde I_k$ are a.n.
can be spelled out in terms of $F$}. Indeed,
in the definitions \eqref{eq:I1rw}-\eqref{eq:tildeIk} of $I_k$ 
and $\tilde I_k$, we can
replace $b_k(y)$ by the equivalent
expression $1/\{(|y| \vee 1) \, \overline{F}(y)^k\}$,
which depends only on $F$. Moreover, the condition that a quantity
$J(\delta;x)$ is a.n.\ can be rephrased using only $F$
(see \eqref{eq:an}, \eqref{eq:tail2}):
\begin{equation*}
	\text{``$J(\delta;x)$ is a.n.''} \qquad \iff \qquad
	\lim_{\delta \to 0} \ \limsup_{x \to \infty} \
	x \, \overline{F}(x) \, J(\delta;x) = 0 \,.
\end{equation*}

\smallskip

In Appendix~\ref{sec:app} we show some relations
between the quantities $I_k$ and $\tilde I_k$.
These lead to the following clarifying remarks.

\begin{remark}\label{rem:refo0}
The condition ``$\tilde I_{\kappa_{\alpha}}$ is a.n.'' is stronger
than ``$I_{\kappa_{\alpha}}$ is a.n.'',
but for $\frac{1}{\alpha} \not\in \N$ they are equivalent
(see Lemma~\ref{th:corequiv2}).
As a consequence, we can rephrase Theorem~\ref{th:mainrw}
in a more compact way as follows:
\begin{equation}
\text{The SRT holds: }
\begin{cases}
\text{with no extra assumption} & \text{for } \alpha > \frac{1}{2}\\
\rule{0pt}{1.2em}\text{iff $\tilde I_1(\delta; x)$ is a.n.}  & \text{for } \frac{1}{3} < \alpha \le \frac{1}{2}\\
\rule{0pt}{1.2em}\text{iff $\tilde I_{\kappa_{\alpha}}(\delta, \eta; x)$ is a.n.\
for every $\eta \in (0,1)$} & \text{for } \alpha \le \frac{1}{3}\\
\end{cases}
\end{equation}

When $\alpha \le \frac{1}{3}$, our proof actually shows that
if $\tilde I_{\kappa_{\alpha}}(\delta,\eta; x)$ is a.n. 
for \emph{some} $\eta > 1-\frac{\alpha}{1-\alpha}$,
then (the SRT holds and consequently) it is a.n. 
for \emph{every} $\eta \in (0,1)$. It is not clear whether the a.n.\
of $\tilde I_{\kappa_{\alpha}}(\delta,\eta; x)$ for some 
$\eta \le 1-\frac{\alpha}{1-\alpha}$ also implies its a.n.\ for any $\eta \in (0,1)$.
\end{remark}

\begin{remark}\label{rem:refo}
If $\frac{1}{\alpha} \not\in \N$,
the condition
``$I_{\kappa _{\alpha }}(\delta,\eta ;x)$ is a.n.''
is equivalent to the seemingly stronger one 
``$I_{k}(\delta,\eta ;x)$ is a.n. for all $k\in \mathbb{N}$''
(see  Lemma~\ref{th:corcascade}).
Similarly, the condition 
``$\tilde I_{\kappa _{\alpha }}(\delta,\eta ;x)$ is a.n.''
is equivalent to
``$\tilde I_{k}(\delta,\eta ;x)$ is a.n. for all $k\in \mathbb{N}$''
(see Lemma~\ref{th:corequiv1}).
\end{remark}

\begin{remark}
In Theorem~\ref{th:mainrw} we require $q > 0$ (that
is the positivity index $\rho$ is strictly less than one),
but a large part of it actually extends to $q=0$. More precisely,
when $q = 0$, our
proof shows that if $\alpha > \frac{1}{2}$ the SRT holds
with no extra assumption on $F$, while if $\alpha \le \frac{1}{2}$
the a.n.\ of $I_{\kappa_{\alpha}}$ (if $\frac{1}{\alpha}\not\in \N$)
or $\tilde I_{\kappa_{\alpha}}$ (if $\frac{1}{\alpha}\in \N$) are
sufficient conditions for the SRT. However, when $q=0$, we do not expect the a.n.\ of
$I_{\kappa_{\alpha}}$
or $\tilde I_{\kappa_{\alpha}}$ to be necessary, in general.
\end{remark}

\subsection{Sufficient conditions for random walks}

Necessary and sufficient conditions for the SRT in the random walk case
involve the a.n.\ of $\tilde I_k$ for a suitable $k = \kappa_{\alpha} \in \N$.
Unlike the renewal process case,
this cannot be reduced to the a.n.\ of just $\tilde I_1$.

\begin{proposition}\label{th:counterex}
For any $\alpha \in (0, \frac{1}{3})$, there is
a probability $F$ on $\R$ which satisfies \eqref{eq:tail1},
such that $\tilde I_1(\delta; x)$ is a.n.\ but $\tilde I_2(\delta, \eta; x)$ is not a.n.,
for any $\eta \in (0,1)$
(hence the SRT fails).
\end{proposition}

Let us now give simpler sufficient conditions which ensure the
a.n.\ of $\tilde I_{k}$.
Note that the condition that 
$\tilde I_{1}(\delta; x)$ is a.n.\ only involves the \emph{right tail} of $F$
(see Definition~\ref{def:an}).
To express conditions on the \emph{left tail} of $F$, we define
\begin{equation} \label{eq:I1rw*}
	\tilde I_{1}^{*}(\delta ;x):=
	\int_{0}^{\delta x} \frac{F((-x-t, -x+t])}{t \vee 1}
	\, b_2(t) \, \dd t \,,
\end{equation}%
which is nothing but $\tilde I_1(\delta; x)$ in \eqref{eq:tildeI10} applied to the \emph{reflected
probability $F^{*}(A) := F(-A)$}.

\begin{proposition}\label{th:suffrw}
Let $F$ be a probability on $\mathbb{R}$ satisfying %
\eqref{eq:tail1} with $\alpha \in (0, \frac{1}{2}]$ and $p > 0$, $q \ge 0$. If
both $\tilde I_1(\delta; x)$
and $\tilde I_1^{*}(\delta; x)$ are a.n., then the SRT holds.

In particular, a sufficient condition for
the SRT
is that there exists $\gamma > 1-2\alpha$ such that
relation \eqref{eq:suff0} holds both for $F$ and for $F^*$
(i.e., both as $x \to +\infty$ and as $x \to -\infty$).

In particular,
the SRT holds when the classical condition \eqref{eq:D97} holds
both for $F$ and~$F^*$.
\end{proposition}

\subsection{L\'{e}vy processes}
Let $X=(X_{t})_{t\geq 0}$ be a L\'evy process with 
L\'evy measure $\Pi$, Brownian coefficient $\sigma ^{2}$ and linear term $\mu$
in its L\'{e}vy-Khintchine representation, that is
\begin{equation} \label{eq:LK}
	\log \E[e^{i\theta X_1}] = i \mu \theta - \frac{\sigma^2}{2} \theta^2
	+ \int_{\R \setminus \{0\}} (e^{i\theta x} - 1 - i \theta x \ind_{\{|x|\le 1\}})
	\, \Pi(\dd x) \,.
\end{equation}
Whenever $X$ is transient, we can define its
potential or renewal measure by
\begin{equation*}
	G(\dd x) := \int_0^\infty \P(X_t \in \dd x) \, \dd t \,.
\end{equation*}

We assume that $X$ is asymptotically stable:
more precisely, there is a norming function $a(t)$ such that
$X_t / a(t)$ converges in law as $t\to\infty$
to a random variable $Y$ with
a stable law of index $\alpha \in (0,1)$ and positivity
parameter $\rho > 0$. In this case
\begin{equation*}
	A(x) := \frac{1}{\overline{\Pi}(x)} = \frac{1}{\Pi((x,\infty))}
	\in RV(\alpha) \qquad \text{as } x \to +\infty \,,
\end{equation*}
and we can take $a(\cdot) = A^{-1}(\cdot)$.
Under these assumptions, the renewal theorem \eqref{eq:SRTint} holds, just
replacing $U([0,x])$ by $G([0,x])$. It
is natural to wonder whether the corresponding local version \eqref{eq:SRT} holds as well,
in which case we say that $X$ satisfies the SRT.

Our next result shows that this question can be reduced to the validity of the SRT
for a random walk whose step distribution $F$ only depends on the L\'evy measure $\Pi$, namely:
\begin{equation}\label{eq:F0}
	F(\dd x) :=
	\begin{cases}
	\displaystyle\frac{\Pi (\dd x)}{\Pi (\R \setminus (-1,1))} 
	& \text{for } |x|\geq 1 \\
	\rule{0pt}{1.3em}0 & \text{for } |x| < 1
	\end{cases} \,.
\end{equation}

\begin{theorem}[SRT for L\'evy Processes]\label{th:Levy}
Let $X$ be any L\'{e}vy process that is in the domain of attraction of a
stable law of index $\alpha \in (0,1)$ and positivity parameter $\rho >0$ as
$t\rightarrow \infty$. Suppose also that its L\'{e}vy measure is
non-arithmetic. Then $X$ satisfies the SRT, i.e.\
\begin{equation}
	\lim_{x\rightarrow \infty } \, x \, \overline{\Pi }(x) \, G((x-h, x])
	= h \, \alpha \,  \E[Y^{-\alpha } \, \ind_{\{Y>0\}}] \,,
	\qquad \forall h > 0 \,, \label{x1}
\end{equation}%
if and only if the random walk with step distribution $F$ defined
in \eqref{eq:F0} satisfies the SRT.
\end{theorem}

As a consequence, the necessary and sufficient conditions for the SRT in
Theorems~\ref{th:main} and~\ref{th:mainrw} can be applied  to
the L\'evy process $X$. We recall
that these conditions can be spelled out 
in terms of the probability $F$ alone (see the comments
 after Theorems~\ref{th:main} and~\ref{th:mainrw}). Then, for a L\'evy process~$X$,
\emph{we have necessary and sufficient conditions for the SRT that
can be spelled out explicitly in terms of the L\'evy measure $\Pi$},
through $F$ defined in \eqref{eq:F0}.

\smallskip

The proof of Theorem~\ref{th:Levy}, 
given in Section~\ref{sec:soft}, is obtained comparing the L\'evy process
$X$ with a compound Poisson process with step distribution $F$.

\begin{remark}
It is known, see \cite[Proof of Theorem 21 on page 38]{B96}, that 
the potential measure $G(\dd x)$ of any L\'evy process $X$ coincides for $x \ne 0$
with the renewal measure of a random walk $(S_n)_{n\ge 0}$ with step distribution
$\P(S_1 \in \dd x) := \int_{0}^{\infty }e^{-t} \, \P(X_{t}\in \dd x) \, \dd t$.
It is also easy to see that $X$ is in the domain of attraction of a stable
law of index $\alpha \in (0,1)$ and positivity parameter $\rho >0$, with
norming function $a(t)$, if and only if the random walk $S$ 
is in the domain of attraction of a
the same stable law with norming function $a(n)$.

So, if we write down
necessary and sufficient conditions
for $S$ to verify the SRT, these will be necessary and sufficient conditions for $X$ to verify
the SRT. However this approach is unsatisfactory, because one would like
conditions expressed in terms of the characteristics of $X$, i.e. the quantities
$\Pi$, $\sigma ^{2}$, $\mu$ appearing in the
L\'{e}vy-Khintchine representation \eqref{eq:LK}, and the technical problem of
expressing our necessary and sufficient conditions for 
$S$ to satisfy the SRT in terms of these characteristics seems quite challenging.
\end{remark}

\subsection{Structure of the paper}

The paper is organized as follows.
\begin{itemize}
\item In Section~\ref{sec:prepa} we recall some standard results.

\item In Section~\ref{sec:newproof} we prove Theorem~\ref{th:ron}.

\item Sections~\ref{sec:usebou}--\ref{sec:casek>1} are devoted to the proofs of Theorems~\ref{th:main}
and~\ref{th:mainrw}.
\begin{itemize}
\item[-] In Section~\ref{sec:usebou} we reformulate the SRT and we give two key bounds.

\item[-] In Section~\ref{sec:necce} we prove the necessity part for both Theorems~\ref{th:main}
and~\ref{th:mainrw}.

\item[-] In Section~\ref{sec:suffe} we prove the sufficiency part of Theorem~\ref{th:main}.

\item[-] The sufficiency part of Theorem~\ref{th:mainrw} 
is proved in Section~\ref{sec:casek=1}
for the case $\alpha > \frac{1}{3}$. The case $\alpha \le \frac{1}{3}$ is treated in
Section~\ref{sec:casek>1} and is much more technical.
\end{itemize}

\item In Section~\ref{sec:soft} we prove ``soft'' results, such as
Theorem~\ref{th:1/2},
Propositions~\ref{pr:main1}, \ref{th:integr},
\ref{th:suffrw}, and Theorem~\ref{th:Levy},
which are corollaries of our main results.

\item In Section~\ref{sec:examples} we prove
Propositions~\ref{th:counter} and~\ref{th:counterex},
which provide counter-examples.

\item In Appendix~\ref{sec:app} we prove
some technical results.
\end{itemize}

\smallskip

\section{Setup}

\label{sec:prepa}

\subsection{Notation}

We recall that $f(s) \lesssim g(s)$ or $f \lesssim g$ means $f(s) = O(g(s))$,
i.e.\ for a suitable constant $C < \infty$ one has $f(s) \le C \, g(s)$ for
all $s$ in the range under consideration. The constant $C$ may depend on the
probability $F$ (in particular, on $\alpha$) and on $h$. When some extra
parameter $\epsilon$ enters the constant $C = C_{\epsilon}$, we write $f(s)
\lesssim_{\epsilon} g(s)$. If both $f \lesssim g$ and $g \lesssim f$, we
write $f \approx g$. We recall that $f(s) \sim g(s)$ means $%
\lim_{s\to\infty} f(s)/g(s) = 1$.

\subsection{Regular variation}

\label{eq:regvar}

Without loss of generality \cite[\S 1.3.2]{cf:BinGolTeu}, we can assume that
$A:[0,\infty) \to (0,\infty)$ is differentiable, strictly increasing and
such that
\begin{equation}  \label{eq:deriv}
A^{\prime }(s) \sim \alpha \frac{A(s)}{s} \,, \qquad \text{as } s \to \infty
\,.
\end{equation}
We fix $A(0) := \frac{1}{2}$ and $A(1) := 1$, so that both $A$ and $%
A^{-1}$ map $[1,\infty)$ onto itself. We also write $a_u = A^{-1}(u)$ for
all $u \in [\frac{1}{2},\infty)$, in agreement with \eqref{eq:an0}.

We observe that, by Potter's bounds, for every $\epsilon > 0$ one has
\begin{equation}  \label{eq:Potter}
\rho^{\alpha+\epsilon} \lesssim_{\epsilon} \frac{A(\rho s)}{A(s)}
\lesssim_{\epsilon} \rho^{\alpha-\epsilon} \,, \qquad \forall \rho \in
(0,1], \ s \in [1,\infty) \ \text{ such that } \ \rho s \ge 1 \,.
\end{equation}
More precisely, part (i) of \cite[Theorem 1.5.6]{cf:BinGolTeu} shows that
relation \eqref{eq:Potter} holds for $\rho s \ge \bar x_{\epsilon}$, for a
suitable $\bar x_{\epsilon} < \infty$; the extension to $1 \le \rho s \le
\bar x_{\epsilon}$ follows as in part (ii) of the same theorem, because $%
A(y) $ is bounded away from zero and infinity for $y \in [1, \bar
x_{\epsilon}]$.

We also recall Karamata's Theorem \cite[Propositions~1.5.8 and~1.5.10]%
{cf:BinGolTeu}:
\begin{align}  \label{eq:Kar1}
\text{if } f \in RV(\zeta) \text{ with } \zeta > -1: \qquad
\int_{s \le t} f(s) \, \dd s 
\,\underset{t \to \infty}{\sim}\, \sum_{n \le t} f(n) 
\,\underset{t \to \infty}{\sim}\, \frac{1}{\zeta+1} \, t \,
f(t) \,, \\
\label{eq:Kar2}
\text{if } f \in RV(\zeta) \text{ with } \zeta < -1: \qquad
\int_{s > t} f(s) \, \dd s 
\,\underset{t \to \infty}{\sim}\, \sum_{n > t} f(n) 
\,\underset{t \to \infty}{\sim}\, \frac{-1}{\zeta+1} \, t \,
f(t) \,.
\end{align}

\subsection{Local limit theorems}

\label{sec:llt}

We call a probability $F$ on $\mathbb{R}$ \emph{lattice} if it is supported
by $v\mathbb{Z}+a$ for some $v > 0$ and $0 \le a < v$, and the maximal value
of $v>0$ with this property is called the \emph{lattice span} of $F$. If $F$
is arithmetic (i.e.\ supported by $h\mathbb{Z}$),
then it is also lattice, but the spans might differ (for instance, $%
F(\{-1\}) = F(\{+1\}) = \frac{1}{2}$ has arithmetic span $h=1$ and lattice
span $v=2$). A lattice distribution is not necessarily arithmetic.\footnote{%
If $F$ is lattice, say supported by $v\mathbb{Z}+a$ where $v$ is the lattice
span and $a \in [0,v)$, then $F$ is arithmetic if and only if $a/v \in
\mathbb{Q}$, in which case its arithmetic span equals $h = v / m$ for some $%
m\in\mathbb{N}$.}

Recall that, under \eqref{eq:tail2}, $S_{n}/a_{n}$ converges in distribution
as $n\rightarrow \infty $ toward a stable law, whose density we denote by $%
\phi $ (the norming sequence $a_{n}$ is defined in \eqref{eq:an0}). If we set
\begin{equation}
J=(-v,0]\qquad \text{with}\qquad v=%
\begin{cases}
\text{lattice span of $F$} & \text{(if $F$ is lattice)} \\
\text{any fixed number $>0$} & \text{(if $F$ is non-lattice)}%
\end{cases}%
\,,  \label{eq:J}
\end{equation}%
Gnedenko's and Stone's local limit theorems \cite[Theorems~8.4.1 and~8.4.2]%
{cf:BinGolTeu} yield
\begin{equation}
\lim_{n\rightarrow \infty }\,\sup_{x\in \mathbb{R}}\left\vert a_{n}\,\P %
(S_{n}\in x+J)-v\,\phi \left( \frac{x}{a_{n}}\right) \right\vert =0\,.
\label{eq:llt}
\end{equation}%
Since $\sup_{z\in \mathbb{R}}\phi (z)<\infty $, we obtain the useful
estimate
\begin{equation}
\sup_{x\in \mathbb{R}}\P (S_{n}\in (x-w,x])\lesssim _{w}\frac{1}{a_{n}}\,,
\label{eq:sup}
\end{equation}%
which, plainly, holds for \emph{any} fixed $w>0$ (not necessarily the
lattice span of $F$).

\smallskip

\section{Proof of Theorem~\protect\ref{th:ron}}

\label{sec:newproof}

We prove \eqref{eq:ron1}, equivalently \eqref{eq:ron1modk},
by steps. Without loss of generality, we assume that
$J \subseteq [0,\infty)$ (it suffices to redefine $x \mapsto x' := x + \min J$ and 
$J \mapsto J' := J - \min J$).

\subsection*{Step 1}
Our starting point is an integrated version of \eqref{eq:ron1}:
\begin{equation}\label{eq:step1}
	\forall \gamma \in (0,\infty): \qquad 
	\P (S_n \ge x, \, M_n \le \gamma x) \lesssim_{\gamma} \bigg(\frac{n}{A(x)}\bigg)^{1/\gamma} \,.
\end{equation}
This is a Fuk-Nagaev inequality, which follows from \cite[Theorems 1.1 and 1.2]{cf:N} 
(see \cite[Theorem 5.1]{B17}
for a more transparent statement). Let us be more precise.
\begin{itemize}
\item \emph{Case $\alpha \in (0,1)$.} We apply equation (1.1) from \cite[Theorem 1.1]{cf:N}
(neglecting the first term in the right hand side, which is the contribution of 
$M_n > y$): for every $y \in (0, x]$ and $t \in (0,1]$, if we define
$A(t; 0, y) := n \int_{0}^{y} u^t \, F(\dd u)$, we have
\begin{equation*}
	\P(S_n \ge x, \, M_n \le y) \le P_1 = \bigg( \frac{e}{1+ \frac{x y^{t-1}}{A(t;0,y)}}
	\bigg)^{\frac{x}{y}} \le
	\bigg( e \, \frac{A(t;0,y)}{x y^{t-1}}
	\bigg)^{\frac{x}{y}} \,.
\end{equation*}
We fix $t \in (\alpha, 1]$, so that
$A(t; 0, y) \le n \int_{0}^{y} t z^{t-1} \, \overline{F}(z) \, \dd z \lesssim
n y^t / A(y)$, thanks to \eqref{eq:tail2} and \eqref{eq:Kar1}.
Taking $y = \gamma x$, since $A(y) \gtrsim_{\gamma} A(x)$, we obtain
\eqref{eq:step1}.
\item \emph{Case $\alpha \in (1,2)$.} We apply equation (1.3) from \cite[Theorem 1.2]{cf:N}: 
for $y \in (0, x]$ and $t \in [1,2]$, setting
$A(t; -y, y) := n \int_{-y}^{y} |u|^t \, F(\dd u)$ and
$\mu(-y,y) := n \int_{-y}^{y} u \, F(\dd u)$,
\begin{equation*}
	\P(S_n \ge x, \, M_n \le y) \le P_3 = \frac{e^{\frac{x}{y}}}{\left(1+ \frac{x y^{t-1}}{A(t;-y,y)}
	\right)^{\frac{x - \mu(-y,y)}{y} + \frac{A(t,-y-y)}{y^t}}} \,.
\end{equation*}
We drop the term $A(t,-y-y) / y^t \ge 0$ from the exponent and get an upper bound.
Next we fix $t \in (\alpha,2]$, so that $A(t; -y, y) \lesssim
n y^t / A(y)$ as before, hence
\begin{equation*}
	\P(S_n \ge x, \, M_n \le y) \le
	\frac{e^{\frac{x}{y}}}{\left(1+ \frac{x}{y} \frac{A(y)}{n}
	\right)^{\frac{x - \mu(-y,y)}{y}}}
	\le \bigg( \frac{e n}{A(y)} \bigg)^{\frac{x}{y}}
	\, \left(1+ \frac{x}{y} \frac{A(y)}{n}
	\right)^{\frac{\mu(-y,y)}{y}} \,.
\end{equation*}
If we fix $y = \gamma x$, the first term in the right hand side matches with
\eqref{eq:step1}. It remains to show that the second term is bounded.
Since we assume that $F$ has zero mean,
we can write $|\mu(-y,y)| = |-n \int_{|u| \ge y} u \, F(\dd u)|
\lesssim n y / A(y)$, by \eqref{eq:tail2} and \eqref{eq:Kar2},
therefore for $y= \gamma x$ the second term is
$\lesssim (1 + \frac{A(y)}{\gamma n})^{n/A(y)}
\le \exp(\frac{1}{\gamma})$. This proves \eqref{eq:step1}.
\end{itemize}

\subsection*{Step 2}
Next we deduce from \eqref{eq:step1} the following relation
\begin{equation}\label{eq:step2}
	\P (S_n \in x+J, \, M_n \le \tfrac{1}{2} \gamma x) \lesssim
	\frac{1}{a_n} \, \bigg( \frac{n}{A(x)} \bigg)^{1/\gamma} \,,
\end{equation}
which is rougher than \eqref{eq:ron1}, due to the factor
$\frac{1}{2}$ and to the exponent $1/\gamma$ instead of $\lfloor 1/\gamma \rfloor$.

Define $\hat X_i := X_{n+1-i}$,
for $1 \le i \le n$, and let
$(\hat S_k := \hat X_1 + \ldots + \hat X_k = S_n - S_{n-k})_{1 \le k \le n}$ be
the corresponding random walk, which has the same law as
$(S_k)_{1 \le k \le n}$. Then
\begin{equation*}
\begin{split}
	\P (S_n \in x+J,\, S_{\lfloor n/2 \rfloor} < \tfrac{x}{2}, \, M_n \le \tfrac{1}{2} \gamma x)
	& = \P (\hat S_n \in x+J,\, \hat S_{\lfloor n/2 \rfloor} 
	< \tfrac{x}{2}, \, \hat M_n \le \tfrac{1}{2} 
	\gamma x) \\
	& = \P (S_n \in x+J,\, S_n - S_{\lceil n/2 \rceil} < \tfrac{x}{2}, 
	\, M_n \le \tfrac{1}{2} \gamma x) \\
	& \le \P (S_n \in x+J,\, S_{\lceil n/2 \rceil} 
	> \tfrac{x}{2}, \, M_n \le \tfrac{1}{2} \gamma x) \,,
\end{split}
\end{equation*}
where the second equality holds because $\hat S_n = S_n$ and $\hat M_n = M_n$,
while for the inequality note that $S_n \ge x$ (by $J \subseteq [0,\infty)$).
To lighten notation, henceforth we assume that $n$ is even
(the odd case is analogous). It follows from the previous inequality that
\begin{equation*}
\begin{split}
	\P (S_n \in x+J, \, M_n \le \tfrac{1}{2} \gamma x) & \le
	2 \, \P (S_n \in x+J,\, S_{n/2} \ge \tfrac{x}{2}, \, M_n \le \tfrac{1}{2} \gamma x)  \\
	& \le 2 \int_{z \ge \frac{x}{2}} \P(S_{n/2} \in \dd z, \, M_{n/2} \le \tfrac{1}{2} \gamma x)
	\, \P(S_{n/2} \in x-z+J) \\
	& \lesssim \frac{1}{a_{n/2}} \, \P(S_{n/2} \ge \tfrac{1}{2}x, \, M_{n/2} \le \tfrac{1}{2} 
	\gamma x)
	\lesssim \frac{1}{a_n} \, \bigg(\frac{n}{A(x)}\bigg)^{1/\gamma}\,,
\end{split}
\end{equation*}
where we have used \eqref{eq:sup} and \eqref{eq:step1}.

\subsection*{Step 3}
Next we prove relation \eqref{eq:ron2}, i.e.\ we show that
\begin{equation}\label{eq:step3}
	\P (S_n \in x+J) \lesssim
	\frac{1}{a_n} \, \frac{n}{A(x)} \,.
\end{equation}
This is easy: if we fix $\epsilon = \frac{1}{2}$,
by \eqref{eq:sup} we can write
\begin{equation} \label{eq:analog}
\begin{split}
	\P (S_n \in x+J, \, & M_n > \epsilon \, x) 
	\le n \, \P (S_n \in x+J, \, X_1 > \epsilon \, x) \\
	& = n \int_{y > \epsilon \, x} F(\dd y) \, \P(S_{n-1} \in x-y+J)
	\lesssim \frac{n}{a_{n-1}} \, \overline F(\epsilon \, x)
	\lesssim_{\epsilon} \frac{1}{a_n} \, \frac{n}{A(x)} \,.
\end{split}
\end{equation}
Applying \eqref{eq:step2} with $\gamma = 1$, we see that \eqref{eq:step3} holds.

\subsection*{Step 4}
Finally we prove \eqref{eq:ron1modk}.
The case $k = 1$, that is $\gamma \in [1,\infty)$, follows by \eqref{eq:step3}.
Inductively, we fix $k\in\N$ and we prove that \eqref{eq:ron1modk} holds for
$\gamma \in [\frac{1}{k+1}, \frac{1}{k})$,
assuming that it holds for $\gamma \in [\frac{1}{k}, \frac{1}{k-1})$.
Let us fix $\epsilon := \frac{1}{2(k+1)}$.
By \eqref{eq:step2} (where we choose $\gamma = 2 \epsilon$) we get
\begin{equation*}
\begin{split}
	\P (S_n \in x+J, \ M_n \le \epsilon x) 
	\lesssim \frac{1}{a_n} \, \bigg( \frac{n}{A(x)} \bigg)^{1/(2\epsilon)}
	= \frac{1}{a_n} \, \bigg( \frac{n}{A(x)} \bigg)^{k+1} \,.
\end{split}
\end{equation*}
It remains to consider
\begin{equation*}
\begin{split}
	\P (S_n \in x+J, \ \epsilon x < M_n \le \gamma x) 
	& \le n \, \P (S_n \in x+J, \, X_1 > \epsilon x\,, M_n \le \gamma x)  \\
	& \le
	n \int_{y \in (\epsilon \, x, \gamma x]} F(\dd y) \, \P(S_{n-1} \in x-y+J,\,
	M_{n-1} \le \gamma x) \\
	& \le n \, \overline{F}(\epsilon x) \, \sup_{z \ge (1-\gamma)x} 
	\P(S_{n-1} \in z+J,\,
	M_{n-1} \le \gamma x) \,.
\end{split}
\end{equation*}
Observe that, for $z \ge (1-\gamma) x$, we can bound
\begin{equation*}
	\P(S_{n-1} \in z+J,\,
	M_{n-1} \le \gamma x) \le
	\P(S_{n-1} \in z+J,\,
	M_{n-1} \le \gamma' z) \,, \qquad \text{with} \quad
	\gamma' := \frac{\gamma}{1-\gamma} \,.
\end{equation*}
The key observation is that $\gamma' \in [\frac{1}{k}, \frac{1}{k-1})$,
since $\gamma \in [\frac{1}{k+1}, \frac{1}{k})$. By our inductive assumption, relation
\eqref{eq:ron1modk} holds for $\gamma'$, so
$\P(S_{n-1} \in z+J,\, M_{n-1} \le \gamma' z) \lesssim_{\gamma} \frac{1}{a_n}
(\frac{n}{A(z)})^k$ and we get
\begin{equation*}
	\P (S_n \in x+J, \ \epsilon x < M_n \le \gamma x) 
	\lesssim_{\gamma} n \, \overline{F}(\epsilon x) \, \frac{1}{a_n}
	\bigg( \frac{n}{A(x)} \bigg)^k
	\lesssim_{\epsilon} \frac{1}{a_n}
	\bigg( \frac{n}{A(x)} \bigg)^{k+1} \,,
\end{equation*}
which completes the proof.\qed

\smallskip

\section{Strategy and key bounds for Theorems~\ref{th:main} and~\ref{th:mainrw}}

\label{sec:usebou}

\subsection{Reformulation of the SRT}
\label{sec:refor}

It turns out that proving the SRT amounts to showing that \emph{small values of $n$ give a
negligible contribution to the renewal measure}. More precisely, if $F$ is a
probability on $\mathbb{R}$ satisfying \eqref{eq:tail2}, it is known that %
\eqref{eq:SRT} holds if and only if
\begin{equation}  \label{eq:SRTeq}
T(\delta; x) := \sum_{1 \le n \le A(\delta x)} \P (S_n \in x+I) \qquad \text{%
is a.n.} \,,
\end{equation}
see \cite[Appendix]{cf:Chi0} or Remark~\ref{rem:SRTeq} below. 

Applying Theorem~\ref{th:ron}, it is easy to
show that \eqref{eq:SRTeq} always holds for $\alpha > \frac{1}{2}$.
Since $%
n/a_n$ is regularly varying with index $1-1/\alpha > -1$, by \eqref{eq:ron2}
and \eqref{eq:Kar1}
\begin{equation*}
\sum_{1 \le n \le A(\delta x)} \P (S_n \in x+I) \lesssim \frac{1}{A(x)}
\sum_{1 \le n \le A(\delta x)} \frac{n}{a_n} \lesssim \frac{1}{A(x)} \frac{%
A(\delta x)^2}{\delta x} \underset{x\to\infty}{\sim} \delta^{2\alpha - 1}
\frac{A(x)}{x} \,,
\end{equation*}
from which \eqref{eq:SRTeq} follows, since $2\alpha - 1 > 0$.
\emph{We have just proved Theorems~\ref{th:main} and~%
\ref{th:mainrw} for $\alpha > \frac{1}{2}$.}
In the next sections, we will focus on the case $\alpha \le \frac{1}{2}$.

\begin{remark}
\label{rem:SRTeq}\textrm{It is easy to see how \eqref{eq:SRTeq} arises. For
fixed $\delta > 0$, by \eqref{eq:U} we can write
\begin{equation}  \label{eq:Uheur}
U(x+I) \ge \sum_{A(\delta x) < n \le A(\frac{1}{\delta}x)} \P (S_n \in x+I)
\,.
\end{equation}
Since $\P (S_n \in x+I) \sim \frac{h}{a_n} \, \phi(\frac{x}{a_n})$ by %
\eqref{eq:llt} (where we take $h=v$ for simplicity), a Riemann sum
approximation yields (see \cite[Lemma~3.4]{cf:Chi0})
\begin{equation*}
\sum_{A(\delta x) < n \le A(\frac{1}{\delta}x)} \P (S_n \in x+I) \sim h \,
\frac{A(x)}{x} \, \mathsf{C}(\delta) \,, \qquad \text{with} \qquad \mathsf{C}%
(\delta) = \alpha \int_{\delta}^{\frac{1}{\delta}} z^{\alpha-2} \phi(\tfrac{1%
}{z}) \, \mathrm{d} z \,.
\end{equation*}
Since $\lim_{\delta \to 0} \mathsf{C}(\delta) = \mathsf{C}$, proving %
\eqref{eq:SRT} amounts to controlling the ranges excluded from %
\eqref{eq:Uheur}, i.e. $\{n \le A(\delta x)\}$ and $\{n > A(\frac{1}{\delta}%
x)\}$. The latter gives a negligible contribution by $\P (S_n \in x+I) \le
C/a_n$ (recall \eqref{eq:sup}), while the former is controlled precisely by %
\eqref{eq:SRTeq}. }
\end{remark}

\subsection{Key bounds}

The next two lemmas estimate the contribution
of the maximum $M_n$, see \eqref{eq:max},
to the probability $\P(S_n \in x+I)$. Recall that $\kappa_{\alpha}$ is defined
in \eqref{eq:kappaalpha}.

\smallskip

We first consider the case when there is a ``big jump'', i.e.\
$M_n > \gamma x$ for some $\gamma > 0$.

\begin{lemma}[Big jumps]
\label{th:kale3} Let $F$ satisfy \eqref{eq:tail2} for some $A \in RV(\alpha)$, 
with $\alpha \in (0, 1)$. There is $\eta = \eta_{\alpha} > 0$
such that for all $\delta \in (0, 1]$, $\gamma \in (0,1)$ and $x \in
[0,\infty)$ the following holds:
\begin{equation}  \label{eq:k}
\begin{split}
& \forall \ell \ge \kappa_{\alpha}: \qquad \sum_{1 \le n \le A(\delta x)}
n^\ell \, \left\{ \sup_{z\in\mathbb{R}} \,\P \left(S_n \in z+I, \, M_n >
\gamma x \right) \right\} \lesssim_{\gamma,\ell} \delta^{\eta} \, b_{\ell +
1}(x) \,.
\end{split}%
\end{equation}
\end{lemma}

\begin{proof}
For $\delta x < 1$ the left hand side of \eqref{eq:k}
vanishes, because $A(\delta x) < A(1) = 1$. Then
we can assume that $\delta x \ge 1$, hence $x \ge 1$.
Recalling \eqref{eq:sup}, we can write
\begin{equation}  \label{eq:plu}
\begin{split}
\P \big(S_n \in z+I, \ M_n > \gamma x \big) & \le n \, \P \left( S_n \in
z+I, \ X_1 > \gamma x \right) \\
& = n \, \int_{w > \gamma x} \P \left(X \in \mathrm{d} w \right) \, \P %
(S_{n-1} \in z-w+I) \\
& \le n \, \P \left(X > \gamma x \right) \, \left\{ \sup_{y\in\mathbb{R}} \P %
(S_{n-1} \in y+I) \right\} \\
& \lesssim \frac{n}{A(\gamma x)} \, \frac{1}{a_n} \lesssim_{\gamma} \frac{n}{%
A(x)} \, \frac{1}{a_n} \,,
\end{split}%
\end{equation}
therefore
\begin{equation}  \label{eq:use}
\begin{split}
\sum_{1 \le n \le A(\delta x)} n^\ell \, \left\{ \sup_{z\in\mathbb{R}} \P %
\left(S_n \in z+I, \, M_n > \gamma x \right) \right\} & \lesssim_{\gamma}
\frac{1}{A(x)} \sum_{1 \le n \le A(\delta x)} \frac{n^{\ell+1}}{a_n} \\
& \lesssim_{\ell} \frac{1} {A(x)} \, \frac{A(\delta x)^{\ell+2}} {\delta x}
\,,
\end{split}%
\end{equation}
by \eqref{eq:Kar1}, because $n^{\ell + 1}/a_n$ is regularly varying with
index $(\ell + 1) - \frac{1}{\alpha} \ge (\kappa_{\alpha} + 1) - \frac{1}{%
\alpha} = \left\lfloor \frac{1}{\alpha} \right\rfloor - \frac{1}{\alpha} >
-1$.
Let us introduce a parameter $b = b_{\alpha} \in (0,1)$, depending only on $%
\alpha$, that will be fixed in a moment. Since we assume that $\delta x \ge 1$,
we can apply the upper bound in %
\eqref{eq:Potter} with $\epsilon = (1-b)\alpha$ and $\rho = \delta$, that is
$A(\delta x) \lesssim \delta^{b\alpha} A(x)$, which shows that \eqref{eq:use}
is
\begin{equation*}
\begin{split}
\lesssim \delta^{b \alpha(\ell + 2)-1} \, \frac{A(x)^{\ell + 1}}{x} \lesssim
\delta^{b \alpha(\kappa_{\alpha} + 2)-1} \, \frac{A(x)^{\ell + 1}}{x} \,,
\end{split}%
\end{equation*}
because $\delta \le 1$ and $\ell \ge \kappa_{\alpha}$ by assumption. Since $%
\alpha(\kappa_{\alpha} +2) > 1$ (because $\kappa_{\alpha} + 2 = \lfloor
\frac{1}{\alpha}\rfloor + 1 > \frac{1}{\alpha}$), we can choose $b =
b_{\alpha}<1$ so that the exponent of $\delta$ is strictly positive (e.g.\ $%
b_{\alpha} = \{\alpha(\kappa_{\alpha}+2)\}^{-1/2}$). This completes the
proof.
\end{proof}

We next consider the case of ``no big jump'', i.e.\ $M_n < \gamma x$.
The proof exploits in an essential way the large deviation estimate
provided by Theorem~\ref{th:ron}.

\begin{lemma}[No big jump]
\label{th:nobig} Let $F$ satisfy \eqref{eq:tail2} with $\alpha \in (0, 1)$.
For any $\gamma \in (0, \frac{\alpha}{1-\alpha})$ there is $\theta =
\theta_{\alpha,\gamma} > 0$ such that for all $\delta \in (0, 1]$ and $x \in
[0,\infty)$ the following holds:
\begin{equation}  \label{eq:0}
\forall \ell \ge 0: \qquad \sum_{1 \le n \le A(\delta x)} n^\ell \, \P %
\left(S_n \in x+I, \, M_n \le \gamma x \right) \lesssim_{\gamma,\ell} \,
\delta^{\theta} \, b_{\ell + 1}(x) \,.
\end{equation}
\end{lemma}

\begin{proof}
As in the proof of Lemma~\ref{th:kale3}, we can assume that
$x \ge 1$ and $\delta x \ge 1$ (since otherwise
the left hand side of \eqref{eq:0} vanishes). By \eqref{eq:ron1}
\begin{equation*}
\sum_{1 \le n \le A(\delta x)} n^\ell \, \P \left(S_n \in x+I, \, M_n \le
\gamma x \right) \lesssim \frac{1}{A(x)^{\frac{1}{\gamma}}} \sum_{1 \le n
\le A(\delta x)} \frac{n^{\ell + \frac{1}{\gamma}}}{a_n} \lesssim \frac{1}{%
A(x)^{\frac{1}{\gamma}}} \frac{A(\delta x)^{\ell + \frac{1}{\gamma} + 1}}{%
\delta x} \,,
\end{equation*}
where we applied \eqref{eq:Kar1}, because the sequence $n^{\ell + \frac{1}{%
\gamma}}/a_n$ is regularly varying with index $\ell + \frac{1}{\gamma} -
\frac{1}{\alpha} > \frac{1-\alpha}{\alpha} - \frac{1}{\alpha} = -1$. 
By the upper bound in \eqref{eq:Potter}, since $\ell \ge 0$ and $\delta
\le 1$ we get
\begin{equation*}
A(\delta x)^{\ell + \frac{1}{\gamma} + 1} \lesssim_{\epsilon}
\delta^{(\alpha-\epsilon)(\ell + \frac{1}{\gamma} + 1)} A(x)^{\ell + \frac{1%
}{\gamma} + 1} \lesssim \delta^{(\alpha-\epsilon)(\frac{1}{\gamma} + 1)}
A(x)^{\ell + \frac{1}{\gamma} + 1} \,.
\end{equation*}
Since $\gamma < \frac{\alpha}{1-\alpha}$, we can choose $\epsilon =
\epsilon_{\alpha,\gamma} > 0$ small enough so that $(\alpha-\epsilon)(\frac{1%
}{\gamma} + 1) > 1$.
\end{proof}

\smallskip

\section{Proof of
Theorems~\protect\ref{th:main} and~\protect\ref{th:mainrw}: necessity}

\label{sec:necce}

In this section we assume \eqref{eq:SRTeq}, which is equivalent to the strong
renewal theorem \eqref{eq:SRT}, and we deduce the necessary conditions 
in Theorems~\protect\ref{th:main} 
and~\protect\ref{th:mainrw}.
We can actually assume \eqref{eq:SRTeq} with $I = (-h,0]$
replaced by \emph{any fixed bounded interval $J$}. Indeed,
if $J = (-v,0]$ (for simplicity), we can bound
$\P (S_n \in x+J) \le \sum_{\ell = 0}^{\lfloor v/h \rfloor} \P (S_n \in
x_{\ell} + I)$, with $x_{\ell} := x-\ell h$.

Note that, since we assume \eqref{eq:SRTeq}, the following holds:
\begin{equation} \label{eq:Fo}
	\text{for any fixed $k\in\N$:} \qquad
	\P(S_k \in x+J) \underset{x\to\infty}{=} o(b_1(x)) \,.
\end{equation}

\subsection{Necessity for Theorem~\ref{th:main}}
\label{sec:necren}
Let us fix a probability $F$ on $[0,\infty)$ satisfying %
\eqref{eq:tail1} with $\alpha \in (0, 1)$.
We assume \eqref{eq:SRTeq}
and we deduce that $I_1^+(\delta;x)$ is a.n.
(recall \eqref{eq:I1}). 

\smallskip

We need some preparation.
Let us define the compact interval
\begin{equation}  \label{eq:K}
K := [\tfrac{1}{2}, 1] \,.
\end{equation}
By \eqref{eq:llt}, 
since $\inf_{z \in K} \phi(z) > 0$,
there are $n_{1} \in \N$ and $c_1, c_2 \in (0,\infty)$ such that%
\begin{gather}
\label{eq:infz}
\forall n\ge n_1 : 
\qquad \inf_{z \in \mathbb{R}: \ z/a_n \in K}
\, \P(S_n \in z+J) \ge \frac{c_1}{a_n} \,, \\
\label{eq:supz}
\forall n \in \N: 
\qquad \sup_{z \in \mathbb{R}} \, \P(S_n \in z+J) \le \frac{c_2}{a_n} \,.
\end{gather}%
Then, since $F((-\infty,-x] \cup [x,\infty)) \lesssim 1/A(x)$,
we can fix $C \in (0,\infty)$ such that
\begin{equation} \label{eq:choiceC}
	\forall n \in \N: \qquad 
	F((-\infty,-C a_n] \cup [C a_n,\infty))  \leq \frac{c_{1}}{2 c_{2}} \, \frac{1}{n} \,.
\end{equation}%
(Of course, we could just take $F([C a_n,\infty))$, since $F((-\infty,0)) = 0$,
but this estimate will be useful later for random walks.)
We also claim that
\begin{equation}  \label{eq:Chillt}
	\forall n \ge n_1: \qquad
	\inf_{z \in \mathbb{R}: \ z/a_n \in K} \P \left(S_n \in z+J, \, 
	\max\{|X_1|,\ldots, |X_n|\} < C
	a_n \right) \ge \frac{c_1}{2} \, \frac{1}{a_n} \,.
\end{equation}
This follows because $\P(S_n \in z+J) \ge c_1/a_n$, by \eqref{eq:infz}, and
applying \eqref{eq:supz}, \eqref{eq:choiceC} we get
\begin{align*}
&\P(S_{n}\in z+J\,, \
\exists 1\leq j\leq n \text{ with } |X_{j}|\geq Ca_{n}) \\
&\leq n\int\limits_{|y|\geq Ca_{n}} F(\dd y) \, \P(S_{n-1}\in z-y+J)\leq \frac{%
n \, F((-\infty,-C a_n] \cup [C a_n,\infty)) \, c_{2}}{a_{n}}\leq \frac{c_{1}}{2 \, a_{n}}.
\end{align*}%

\smallskip

We can now start the proof.
The events $B_i := \{X_i \ge C a_n, \, \max_{j \in \{1,\ldots,n+1\}
\setminus \{i\}} X_j < C a_n\}$ are disjoint for $i=1,\ldots, n$,
hence for $n \ge n_1$ we can write
\begin{equation} \label{eq:intnec}
\begin{split}
\P(& S_{n+1} \in x+J) \\
& \ge (n+1) \, \P \left(S_{n+1} \in x+J, \ 
\max\{X_1,\ldots, X_n\} < C a_n%
, \ X_{n+1} \ge C a_n \right) \\
& \ge n \, \int_{\{z \le x - C a_n\}} \P \left(X_{n+1} \in x - \mathrm{d} z \right) \P %
\left(S_{n} \in z+J, \ \max\{X_1,\ldots, X_n\} < C a_n \right) \\
& \ge \int_{\{z \le x - C a_n\}} F\left(x - \mathrm{d} z \right) \, \frac{c_1}{2} \, 
\frac{n}{a_{n}} \, \mathds{1}_{\{z/a_{n} \in K\}} \,,
\end{split}%
\end{equation}
where the last inequality holds by \eqref{eq:Chillt}.
We are going to choose $n \le A(\delta x)$, in particular $x - C a_n \ge x - C \delta x
\ge \frac{\delta}{2} x$ for $\delta > 0$ small enough. Restricting the integral, we get
\begin{equation*}
	\sum_{n_1 \le n \le A(\delta x)} \P(S_{n+1} \in x+J)
	\gtrsim \int_{\{z \le \frac{\delta}{2} x\}} 
	F\left(x - \mathrm{d} z \right) \, 
	\Bigg( \sum_{n_1 \le n \le A(\delta x)} \frac{n}{a_{n}} \, \mathds{1}_{\{z/a_{n} \in K\}}
	\Bigg) \,.
\end{equation*}
Note that $z/a_{n} \in K$ means $\frac{1}{2} a_{n}
\le z \le a_{n}$, that is $A(z) \le n \le A(2z)$,
so in the range of integration we have $A(2z) \le A(\delta x)$.
If we further restrict the integration on $z \ge a_{n_1}$,
we also have $A(z) \ge n_1$. This leads to the following lower bound:
\begin{equation*}
	\sum_{n_1 \le n \le A(\delta x)} \frac{n}{a_{n}} \mathds{1}_{\{z/a_{n}
	\in K\}} \ge \sum_{A(z) \le n \le A(2z)} \frac{n}{a_{n}} \ge \frac{A(z)}
	{2z} \big(A(2z)-A(z)-1\big) \gtrsim \frac{A(z)^2}{z} = b_2(z) \,,
\end{equation*}
where the last inequality holds for $z \ge a_{n_1}$
large (just take $n_1$ large enough).
Then
\begin{equation*}
\begin{split}
\sum_{n_1 \le n \le A(\delta x)} \P (S_{n+1} \in x+J) & \gtrsim 
\int_{\{a_{n_1} \le z \le \frac{\delta}{2} x\}} F\left(x - \mathrm{d} z \right) \, b_2(z) \\
& \ge I_1^+(\tfrac{\delta}{2}; x) -
\hat C \, F([x-a_{n_1},x-1]) \,,
\end{split}%
\end{equation*}
where $\hat C := \sup_{|z| \le a_{n_1}} b_2(z) < \infty$.
The left hand side is a.n.\ by \eqref{eq:SRTeq},
hence the right hand side is a.n.\ too. Since $F([x-a_{n_1}, x])$ is a.n.\
by \eqref{eq:Fo}, it follows that $I_1^+(\delta; x)$ is a.n..\qed

\smallskip

\subsection{Necessity for Theorem~\ref{th:mainrw}}
Let $F$ be a probability on $\R$ satisfying %
\eqref{eq:tail2} with $\alpha \in (0, 1)$ and $p, q > 0$.
We assume \eqref{eq:SRTeq}, which is equivalent to the \eqref{eq:SRT},
and we deduce that
$\tilde I_{1}(\delta;x)$ is a.n.\ and, for any $k\ge 2$, that
$\tilde I_{k}(\delta, \eta;x)$ is also a.n., for every fixed $\eta \in (0,1)$.
This completes the proof of the necessity part in Theorem~\ref{th:mainrw}
(see Remarks~\ref{rem:refo0}-\ref{rem:refo}).

\smallskip

\begin{remark}
For $|x| \ge 1$ and $|z| \ge |x|$ we can rewrite \eqref{eq:btilde} as
\begin{equation*}
	\tilde{b}_{k}(z,x) = \int_{|x| }^{|z| } \frac{b_k(t)}{t} \, \dd t 
	= \int_{A(|x|) }^{A(|z|)} \frac{b_k(A^{-1}(s))}{A^{-1}(s)} \,
	\frac{1}{A'(A^{-1}(s))} \, \dd s \approx
	\int_{A(|x|) }^{A(|z|)} \frac{s^{k-1}}{A^{-1}(s)} \, \dd s \,,
\end{equation*}
by \eqref{eq:deriv} and \eqref{eq:b}
(we recall that $\approx$ means both $\lesssim$ and $\gtrsim$).
Recalling also \eqref{eq:an0}, we obtain
\begin{equation}\label{eq:btildealt}
	\tilde{b}_{k}(z,x)
	\approx \sum_{A(|x|) \le n \le A(|z|)}\frac{n^{k-1}}{a_{n}} \,.
\end{equation}
\end{remark}

Since we assume that $p, q > 0$ in \eqref{eq:tail2}, the density $\phi(\cdot)$
of the limiting L\'evy process is strictly positive on the whole real line.
In particular, instead of \eqref{eq:K}, we can define
\begin{equation}\label{eq:K2}
	K := [-1,1] \,,
\end{equation}
and relations \eqref{eq:infz}, \eqref{eq:supz}, \eqref{eq:choiceC}, \eqref{eq:Chillt} 
still hold, where $n_1 \in \N$ is fixed (it depends on $F$).

Let us show that $\tilde I_{1}(\delta;x)$ is a.n.. This is similar to the case
of renewal processes in Subsection~\ref{sec:necren}.
In fact, relation \eqref{eq:intnec} with $X_i$ replaced by $|X_i|$
and $z$ replaced by $-y$ gives
\begin{equation}\label{eq:intnec2}
	\P(S_{n+1} \in x+J) \gtrsim \int_{|x + y| \ge C a_n} F(x+\dd y)
	\, \frac{n}{a_n} \, \ind_{\{|y| \le a_n\}} \,,
\end{equation}
because $K=[-1,1]$.
Note that for $n \le A(\delta x)$ we have $a_n \le \delta x \le x - C a_n$
for $\delta > 0$ small, hence \emph{we can ignore the restriction $|x + y| \ge C a_n$}.
Next we write, by \eqref{eq:btildealt},
\begin{equation*}
	\sum_{n=n_1}^{A(\delta x)} \frac{n}{a_n} \, \ind_{\{|y| \le a_n\}}
	= \ind_{\{|y| \le \delta x\}} \, \sum_{n=A(y) \vee n_1}^{A(\delta x)} \frac{n}{a_n}
	\gtrsim \ind_{\{|y| \le \delta x\}} \, \tilde b_2(\delta x, |y| \vee a_{n_1}) \,.
\end{equation*}
For $|y| < a_{n_1}$, $\tilde b_2(\delta x, |y| \vee a_{n_1})
= \tilde b_2(\delta x, a_{n_1})$
differs from $\tilde b_2(\delta x, |y|)$ at most by the constant
$C := \sum_{n \le n_1} \frac{n}{a_n}$, so
$\tilde b_2(\delta x, |y| \vee a_{n_1}) \ge \tilde b_2(\delta x, |y|) - C
\, \ind_{\{|y| \le K\}}$, with $K := a_{n_1}$. This yields
\begin{equation*}
	\sum_{1 \le n \le A(\delta x)}\P(S_{n+1} \in x+J) \gtrsim
	\tilde I_1(\delta; x) - C \, F([x-K, x+K]) \,.
\end{equation*}
Since we assume that \eqref{eq:SRTeq} holds, 
and we have $F([x-K, x+K]) = o(b_1(x))$ as $x \to \infty$,
as we already observed in \eqref{eq:Fo},
it follows that $\tilde I_1(\delta; x)$ is a.n..

\smallskip

Next we fix $k\ge 2$ and $\eta \in (0,1)$ and we generalize the
previous arguments in order to show that
$\tilde I_k(\delta,\eta; x)$ is a.n., see \eqref{eq:tildeIk}.
\emph{Inductively, we assume that we already know that $\tilde I_1(\delta; x)$,
$\tilde I_2(\delta,\eta; x)$, \ldots, 
$\tilde I_{k-1}(\delta,\eta; x)$
are a.n.}.
Suppose that $z_{1},\cdots z_{k} \in \R$ satisfy 
\begin{equation*}
	\min_{1\leq j\leq k}|z_{j}| \ge Ca_{n}	\,, \qquad
	|(z_1 + \ldots + z_k) - x| \le  a_{n} \,,
\end{equation*}
and set $y_k := x-(z_1 + \ldots + z_k)$. Then, for $n \ge n_1$, we can write
\begin{align*}
P&(\exists 1\leq j_{1}<j_{2}<\cdots <j_{k}\leq n\text{ with }X_{j_{1}}\in
\dd z_{1}\,, \ldots, \, X_{j_k} \in \dd z_k\,, \text{ and }S_{n+k}\in x+I) \\
&\ge \binom{n+k}{k} \, \P(X_{r}\in \dd z_{r} \ \forall 1\leq r\leq k\,,
\ X_{j}\notin \{\dd z_{1},\cdots \dd z_{k}\} \ \forall k<j\leq n+k\,, \ S_{n+k} \in x+I) \\
&\gtrsim n^{k} \, \P(X_{r}\in \dd z_{r}\,, \ \forall 1\leq r\leq k ) \;
\P( |X_{j}| \le Ca_{n}\,, \ \forall 1\leq j\leq n \,, \ S_{n}\in y_{k}+I) \\
&\gtrsim \frac{n^{k}}{a_{n}} \, \P(X_{r}\in \dd z_{r}\,,  \ \forall 1\leq r\leq k) \,,
\end{align*}%
having used \eqref{eq:Chillt} in the last inequality.
It follows that for $n \ge n_1$ we have the the bound
\begin{equation*}
\begin{split}
	\P(S_{n+k} \in x+I) & \gtrsim
	\frac{n^k}{a_n} \, \P\left(\min_{1 \le r \le k}|X_r| \ge C a_n\,, \ |(X_1 + \ldots
	+ X_k)-x| \le  \, a_n\right) \\
	& = \frac{n^k}{a_n} \,
	\P_{-x} \left(\min_{1 \le r \le k} |S_r - S_{r-1}| \ge C a_n \,, 
	\ |S_k| \le  \, a_n \right) \,,
\end{split}
\end{equation*}
where $\P_{-x}$ denotes the law of the random walk $S_r := -x + (X_1 + \ldots + X_r)$,
$r \ge 1$, which starts from $S_0 := -x$.

If we fix $\eta \in (0,1)$, and define $\overline{\eta} := 1-\eta$, we can write
\begin{equation*}
	\left\{ \min_{1 \le r \le k} |S_r - S_{r-1}| \ge C a_n \right\}
	\supseteq \left\{  |S_r - S_{r-1}| \ge \overline{\eta} |S_{r-1}|
	\text{ and } |S_{r-1}| \ge \tfrac{C}{\overline{\eta}} \, a_n \,, \forall 1 \le r \le k
	\right\} \,.
\end{equation*}
For $r=1$,
$|S_{r-1}| \ge \frac{C}{\overline{\eta}} \, a_n$ reduces to
$x \ge \frac{C}{\overline{\eta}} \, a_n$, 
which holds automatically, since we take $n \le A(\delta x)$ with $\delta > 0$ small,
while $|S_r - S_{r-1}| \ge \overline{\eta} |S_{r-1}|$ becomes $|S_1 +x| \ge \overline{\eta} x$,
which is implied by $|S_{1}| \le \frac{C}{\overline{\eta}} \delta x$, for $\delta > 0$ small.
For $r \ge 2$, $|S_r - S_{r-1}| \ge \overline{\eta} |S_{r-1}|$
is implied by $|S_r| \le \eta |S_{r-1}|$, since $\overline{\eta} = 1-\eta$. Thus
\begin{equation*}
	\left\{ \min_{1 \le r \le k} |S_r - S_{r-1}| \ge C a_n \right\}
	\supseteq \left\{  |S_{1}| \le \tfrac{C}{\overline{\eta}} \delta x \,,
	\ \; |S_r| \le \eta |S_{r-1}| \,, 
	\, \forall 2 \le r \le k
	\,, \ \; |S_{k-1}| \ge \tfrac{C}{\overline{\eta}} a_n \right\} ,
\end{equation*}
where the last term is justified because
$|S_{k-1}| = \min_{2 \le r \le k-1} |S_{r-1}|$ on the event.
Thus
\begin{equation*}
\begin{split}
	\P(S_{n+k} \in x+I)
	\gtrsim 
	\E_{-x} \Bigg[& \ind_{\{|S_{1}| \le \frac{C}{\overline{\eta}} \delta x, \ 
	|S_r| \le \eta |S_{r-1}| \,, 
	\, \forall 2 \le r \le k \}} \, \bigg( \frac{n^k}{a_n} \, \ind_{\{A(|S_k|) \le n \le
	A(\frac{\overline{\eta}}{C} |S_{k-1}|)\}} \bigg) \Bigg] \,.
\end{split}
\end{equation*}

Let us now sum over $n_1 \le n \le A(\delta x)$. 
Note that 
$A(\frac{\overline{\eta}}{C} |S_{k-1}|) \le A(\frac{\overline{\eta}}{C} |S_{1}|) \le A(\delta x)$,
hence
\begin{equation} \label{eq:ababd}
\begin{split}
	\sum_{n_1 \le n \le A(\delta x)} & \P(S_{n+k} \in x+I) \\
	& \gtrsim 
	\E_{-x} \Big[ \ind_{\{|S_{1}| \le \frac{C}{\overline{\eta}} \delta x, \ 
	|S_r| \le \eta |S_{r-1}| \,, 
	\, \forall 2 \le r \le k
	\}} 
	\, \tilde b_{k+1}\big(\tfrac{\overline{\eta}}{C} |S_{k-1}|,
	\, |S_k| \vee a_{n_1} \big) \Big] \,,
\end{split}
\end{equation}
where we recall that $\tilde b_{k+1}$ is given by \eqref{eq:btildealt}.
The right hand side can be rewritten as
\begin{equation} \label{eq:quasitildeI}
\begin{split}
	\int\limits_{\substack{|y_{1}| \le \delta' x \\
	|y_r| \le \eta |y_{r-1}| \text{ for all }
	2 \le r \le k}} 
	& F(x+\dd y_1) \, P_{y_1}(\dd y_2, \ldots,
	\dd y_k) \, \tilde b_{k+1}\big(\epsilon |y_{k-1}|, 
	\, |y_k| \vee c \big) \,, \\
	& \text{where} \quad \delta' = \tfrac{C}{\overline{\eta}} \delta \,, \quad
	\epsilon = \tfrac{\overline{\eta}}{C}\,, \quad c = a_{n_1} \,.
\end{split}
\end{equation}
This is like $\tilde I_k(\delta', \eta; x)$, see \eqref{eq:tildeIk},
except that
$\tilde b_{k+1} (y_{k-1}, y_k ) =
\tilde b_{k+1} (|y_{k-1}|, |y_k| )$ in \eqref{eq:tildeIk} is replaced by 
$\tilde b_{k+1} (\epsilon |y_{k-1}|, |y_k| \vee c \big)$.
We now show that this is immaterial.
More precisely, by \eqref{eq:SRTeq}
and \eqref{eq:ababd}, we know that \eqref{eq:quasitildeI} is a.n.. 
We now deduce that $\tilde I_{k}(\delta,\eta; x)$ is a.n..

\smallskip

Since $\tilde b_{k+1} (|y_{k-1}|, |y_k| )$
differs from $\tilde b_{k+1} (|y_{k-1}|, |y_k| \vee c )$
at most by $C := \sum_{n=1}^{A(c)} \frac{n^{k+1}}{a_n}$,
see \eqref{eq:btildealt}, we can bound
$\tilde b_{k+1} (|y_{k-1}|, |y_k| ) \le \tilde b_{k+1} (|y_{k-1}|, |y_k| \vee c )
+ C \, \ind_{\{|y_k| \le c\}}$. Plugging this into
\eqref{eq:tildeIk}, we see that the contribution
of $\ind_{\{|y_k| \le c\}}$ is a.n., because it is at most
\begin{equation*}
	\int_{\R} F(x+\dd y_1) 
	\int_{\R^{k-1}}
	P_{y_{1}}(\mathrm{d}y_{2},\ldots ,\mathrm{d}y_{k})
	\, \ind_{\{|y_k| \le c\}}
	= \P(S_k \in [x-c, x+c])
	\underset{x\to\infty}{=} o(b_1(x)) \,,
\end{equation*}
by \eqref{eq:Fo}. Then in \eqref{eq:tildeIk}
we can safely replace $\tilde b_{k+1} (y_{k-1}, y_k )$
by $\tilde b_{k+1} (|y_{k-1}|, |y_k| \vee c )$.

Finally, we write 
$\tilde b_{k+1}\big(|y_{k-1}|, \, |y_k| \vee c \big)
= \tilde b_{k+1}\big(|y_{k-1}|, \, \epsilon |y_{k-1}| \big)
+ \tilde b_{k+1}\big(\epsilon |y_{k-1}|, \, |y_k| \vee c \big)$.
Note that the contribution of the second term to 
$\tilde I_{k}(\delta,\eta; x)$ in \eqref{eq:tildeIk} is a.n., because
we already know that \eqref{eq:quasitildeI} is a.n.. 
For the first term, observe that by \eqref{eq:btildealt}
\begin{equation*}
	\tilde b_{k+1}\big(|y_{k-1}|, \, \epsilon |y_{k-1}| \big)
	\lesssim \sum_{n=A(\epsilon|y_{k-1}|)}^{A(|y_{k-1}|)} \frac{n^k}{a_n}
	\le \frac{A(|y_{k-1}|)^{k+1}}{\epsilon |y_{k-1}|} \lesssim_{\epsilon} b_{k+1}(y_{k-1}) \,,
\end{equation*}
hence
\begin{equation*}
\begin{split}
	\int\limits_{|y_k| \le \eta |y_{k-1}|}
	\!\!\!\!\!\!\!\! F(-y_{k-1} + \dd y_k)
	\, \tilde b_{k+1}(|y_{k-1}|, \epsilon |y_{k-1}|) & \lesssim_{\epsilon} 
	b_{k+1}(y_{k-1}) \,
	F(-(1-\eta)|y_{k-1}|) 
	\lesssim_{\eta} b_{k}(y_{k-1}) \,,
\end{split}
\end{equation*}
so the contribution to $\tilde I_{k}(\delta,\eta; x)$ in \eqref{eq:tildeIk}
is $\lesssim_{\epsilon,\eta} I_{k-1}(\delta,\eta; x)$. 
We know that $\tilde I_{k-1}$ is a.n., by 
our inductive
assumption, and this implies that $I_{k-1}$ is a.n. too,
by the inequalities \eqref{eq:IItildele1} and \eqref{eq:IItildele}
in the Appendix (see \eqref{eq:ssttaa}-\eqref{eq:inan} for their proof).
We are done.\qed

\smallskip

\section{Proof of Theorem~\protect\ref{th:main}: sufficiency}
\label{sec:suffe}

In this section we prove the sufficiency part of Theorem~\ref{th:main}:
we assume that $I_1^+(\delta;x)$ is a.n.\ and we deduce
\eqref{eq:SRTeq}, which is equivalent to the SRT.
Let us set
\begin{equation}  \label{eq:Tell}
T_{\ell}(\delta; x) := \sum_{1 \le n \le A(\delta x)} n^\ell \, \P (S_n \in
x+I) \,.
\end{equation}
We actually prove the following result.

\begin{theorem}
\label{th:main2} Let $F$ be a probability on $[0,\infty)$ satisfying %
\eqref{eq:tail1} with $\alpha \in (0, 1)$. Assume that $I_1^+(\delta;x)$ is
a.n.. Then for every $\ell \in \mathbb{N}_0$:
\begin{equation}  \label{eq:ggoal}
\lim_{\delta \to 0} \, \limsup_{x\to\infty} \, \frac{T_{\ell}(\delta; x)}{%
b_{\ell + 1}(x)} = 0 \,.
\end{equation}
In particular, setting $\ell = 0$, relation \eqref{eq:SRTeq} holds.
\end{theorem}

The proof exploits the general bounds provided by Lemmas~\ref{th:kale3} and~%
\ref{th:nobig}, together with the next Lemma, which is specialized to
renewal processes.

\begin{lemma}
\label{th:basic} If $F$ is a probability on $[0,\infty)$ which satisfies %
\eqref{eq:tail1} with $\alpha \in (0,1)$, there are $C,c \in (0,\infty)$
such that for all $n\in\mathbb{N}_0$ and $z \in [0,\infty)$
\begin{equation}  \label{eq:basic}
\P (S_n \in z+I) \le \frac{C}{a_n} \, e^{-c \frac{n}{A(z)}} \,.
\end{equation}
\end{lemma}

\begin{proof}
Assume that $n$ is even (the odd case is analogous). By \eqref{eq:sup}, we
get
\begin{equation*}
\begin{split}
\P \left(S_{n} \in z+I \right) & = \int_{y \in [0,z]} \P (S_{\frac{n}{2}}
\in \mathrm{d} y ) \, \, \P (S_{\frac{n}{2}} \in z-y+I) \lesssim \frac{1}{a_{%
\frac{n}{2}}} \, \P (S_{\frac{n}{2}} \le z) \\
& \lesssim \frac{1}{a_n} \, \P \left(\max_{1 \le i \le \frac{n}{2}} X_i \le
z \right) = \frac{\left(1 - \P (X>z) \right)^{\frac{n}{2}}}{a_n} \le \frac{%
e^{-\frac{n}{2} \P (X>z)}}{a_n} \le \frac{e^{-c \frac{n}{A(z)}}}{a_n} \,,
\end{split}%
\end{equation*}
provided $c > 0$ is chosen such that $\P (X > z) \ge 2c/A(z)$ for all $z \ge
0$. This is possible by \eqref{eq:tail1} and because $z \mapsto A(z)$ is
increasing and continuous, with $A(0) > 0$ (see \S \ref{eq:regvar}).
\end{proof}

\begin{remark}
\textrm{Since $A(\cdot)$ is increasing,
it follows by \eqref{eq:basic}
that for any $\bar x > 0$
and $\ell\in\N$ 
\begin{equation}  \label{eq:unibo}
	\sup_{z \in [0, \bar{x}]} 
	\bigg\{ \sum_{n\in\N} n^\ell \, \P (S_n \in z+I) \bigg\}
	 \lesssim C_{\bar{x},\ell}
	 \,, \qquad \text{with} \qquad
	C_{\bar{x},\ell}
	:= \sum_{n\in\N} \frac{n^\ell}{a_n}
	\, e^{-c \frac{n}{A(\bar{x})}} \, < \, \infty \,.
\end{equation}%
}
\end{remark}

\smallskip

Before proving Theorem~\protect\ref{th:main2},
we state some easy consequences of ``$I_1^+(\delta; x)$ is a.n.''.
\begin{itemize}
\item First we show that, for any bounded interval $J \subseteq \mathbb{R}$,
\begin{equation}  \label{eq:nec0}
	I_1^+(\delta; x) \ \text{ is a.n.} \qquad \Longrightarrow \qquad F( x + J) 
	\underset{x \to \infty}{=} o\big(b_1(x)\big)  \,.
\end{equation}
It is convenient to write $J = [-1 - b, -1 - a]$, for $a,b\in\R$ with $a < b$.
Restricting \eqref{eq:I1} to
$z \in -a-J = [1, 1 + (b-a)]$ we get
$I_1^+(\delta; x) \ge F(x + a + J) 
\cdot \inf_{z\in -a-J} b_2(z) \gtrsim_J F(x+ a +J)$,
so if $I_1^+(\delta; x)$ is a.n.\ then
$F(x+ a +J) = o(b_1(x))$, hence \eqref{eq:nec0} follows.

\smallskip

\item Next we improve \eqref{eq:nec0} as follows:
\begin{equation}  \label{eq:nec1}
	I_1^+(\delta; x) \ \text{ is a.n.} \quad \ \ \Longrightarrow \quad \ \ 
\text{for every fixed $\ell \in \N$}: \quad
\P(S_\ell \in x+J)
\underset{x \to \infty}{=} o\big( b_1(x) \big) \,.
\end{equation}
To see this, we write $J = [a,b]$
and we note that on the event $\{S_\ell \in x+J\}$
we must have $M_\ell \ge (x-a)/\ell$, hence
$S_\ell - M_\ell \le x+b - \frac{x-a}{\ell} \le x - (\frac{x}{\ell} - 2b)$. Then
\begin{equation} \label{eq:elluse}
\begin{split}
	\P(S_{\ell} \in x+J) 
	\le \ell \int_{z > \frac{x}{\ell} - 2b} 
	\P( S_{\ell-1} \in x-\dd z) \, \P(X_1 \in z + J)
	\le \ell \, \sup_{z > \frac{x}{\ell}-2b} F(z+J) \,,
\end{split}
\end{equation}
and, by \eqref{eq:nec0}, as $x\to\infty$ the right hand side is
$o(b_1(\frac{x}{\ell} - 2b)) = o(b_1(x))$, for fixed $\ell$.

\smallskip

\item Finally, we observe that
for any fixed $\gamma > 0$
\begin{equation}\label{eq:I1+O}
	I_1^+(\delta; x) \ \text{ is a.n.} \quad \ \ \Longrightarrow \quad \  \
	\text{for every fixed $\gamma \in (0,1)$}: \quad
	I_1^+(1-\gamma; x)
	\underset{x \to \infty}{=} O\big(b_1(x)\big)  \,.
\end{equation}
First we fix $\bar{\delta} > 0$ small enough
so that $I_1^+\big( \bar{\delta}; x \big) = O(b_1(x))$
(recall Definition~\ref{def:an}).
Then we consider the contribution
to $I_1^+(1-\gamma,x)$ from $z \ge \bar{\delta}$, see \eqref{eq:I1}, which is
\begin{equation*}
	\int_{\bar{\delta} x \le z \le (1-\gamma) x} F(x - \mathrm{d} z) \, \frac{A(z)^2}{z}
	\le \frac{A(x)^2}{\bar{\delta} \, x} \, \overline{F}(\gamma x)
	\lesssim_{\gamma, \bar{\delta}} \frac{A(x)}{x} = b_1(x) \,.
\end{equation*}
\end{itemize}

\begin{proof}[Proof of Theorem~\protect\ref{th:main2}]
We fix, once and for all, $\gamma \in (0, \frac{\alpha}{1-\alpha})$, and we
decompose
\begin{equation*}
\begin{split}
T_{\ell}(\delta; x) & = \sum_{1 \le n \le A(\delta x)} n^\ell \, \P (S_n \in
x+I, \, M_n > \gamma x) \\
& \qquad \qquad + \sum_{1 \le n \le A(\delta x)} n^\ell \, \P (S_n \in x+I,
\, M_n \le \gamma x) \,.
\end{split}%
\end{equation*}
Then it follows by Lemma~\ref{th:kale3} and Lemma~\ref{th:nobig} that %
\eqref{eq:ggoal} holds for every $\ell \ge \kappa_{\alpha}$.

It remains to prove that \eqref{eq:ggoal} holds for $\ell < \kappa_{\alpha}$%
. We proceed by backward induction: we fix $\ell \in \{0, 1, \ldots,
\kappa_{\alpha} - 1\}$ and, \emph{assuming} that
\begin{equation}  \label{eq:assu}
\lim_{\delta \to 0} \, \limsup_{x\to\infty} \, \frac{T_{\ell+1}(\delta; x)}{%
b_{\ell + 2}(x)} = 0 \,,
\end{equation}
we deduce \eqref{eq:ggoal}. We need to estimate $T_{\ell}(\delta; x)$ and we
split it in some pieces.

We start by writing
\begin{equation*}
\P \left(S_n \in x+I \right) = \P \left(S_n \in x+I, \, M_n \le \gamma
x\right) + \P \left(S_n \in x+I, \, M_n > \gamma x\right) \,,
\end{equation*}
and note that the contribution of the first term in the right hand side is
negligible for \eqref{eq:ggoal}, by Lemma~\ref{th:nobig}. Next we bound
\begin{equation}  \label{eq:splito}
\begin{split}
\P \left(S_n \in x+I, \, M_n > \gamma x\right) & \le n \, \P \left( S_n \in
x+I, \, X_1 > \gamma x \right) \\
& = n \, \int_{0 \le z < (1-\gamma)x} F( x - \mathrm{d} z) \, \P %
\left(S_{n-1} \in z +I \right) \,.
\end{split}%
\end{equation}
Looking back at \eqref{eq:Tell}, we may restrict the sum to $n \ge 2$,
because the contribution of the term $n=1$ is negligible for \eqref{eq:ggoal}, 
since $F(x+I) = o(b_1(x)) = o(b_{\ell + 1}(x))$ by \eqref{eq:nec0}. As a
consequence, it remains to prove that \eqref{eq:ggoal} holds with $%
T_{\ell}(\delta;x)$ replaced by
\begin{equation*}
\begin{split}
\widetilde T_{\ell}(\delta; x) & := \sum_{2 \le n \le A(\delta x)} n^{\ell +
1}\, \int_{0 \le z < (1-\gamma)x} F( x - \mathrm{d} z ) \, \P \left(S_{n-1}
\in z +I \right) \,.
\end{split}%
\end{equation*}
We can bound $n^{\ell + 1} \lesssim (n-1)^{\ell + 1}$, since $n \ge 2$, and
rename $n-1$ as $n$, to get
\begin{equation}  \label{eq:cheint300}
\widetilde T_{\ell}(\delta; x) \le \int_{1 \le z < (1-\gamma)x} F( x -
\mathrm{d} z ) \Bigg\{\sum_{1 \le n \le A(\delta x)-1} n^{\ell + 1}\, \P %
\left(S_{n} \in z +I \right) \Bigg\} \, + \, o(b_{\ell + 1}(x)) \,,
\end{equation}
where we have restricted the integral to $z \ge 1$, because the
contribution of $z \in [0,1)$
can be estimated as $o(b_1(x)) = o(b_{\ell + 1}(x))$, thanks to %
\eqref{eq:unibo} and \eqref{eq:nec0}.

\smallskip

Let us fix $\epsilon \in (0,1)$ and consider the contribution to the sum in %
\eqref{eq:cheint300} given by $n > A(\epsilon z)$. Applying Lemma~\ref%
{th:basic}, since $a_n \ge \epsilon z \gtrsim_{\epsilon} z$, we get
\begin{equation}  \label{eq:neg1}
\begin{split}
\sum_{A(\epsilon z) < n \le A(\delta x)} n^{\ell + 1}\, & \P \left(S_{n} \in
z +I \right) \lesssim \sum_{A(\epsilon z) < n \le A(\delta x)} \frac{n^{\ell
+ 1}}{a_n} \, e^{-c \frac{n}{A(z)}} \\
& \lesssim_{\epsilon} \frac{A(z)^{\ell + 2}}{z} \, \mathds{1}_{\{z < \frac{%
\delta}{\epsilon} x\}} \, \left\{ \frac{1}{A(z)} \sum_{n \in\mathbb{N}} \left(%
\frac{n}{A(z)}\right)^{\ell + 1} \, e^{-c \frac{n}{A(z)}} \right\} \,.
\end{split}%
\end{equation}
The bracket is a Riemann sum which converges to $\int_0^\infty t^{\ell + 1}
\, e^{-ct} \, \mathrm{d} t < \infty$ as $z \to \infty$, hence it is
uniformly bounded for $z \in [0,\infty)$. The contribution of $n > A(\epsilon z)$
to \eqref{eq:cheint300} is then
\begin{equation*}
\lesssim_{\epsilon} \int_{1 \le z < \frac{\delta}{\epsilon} x} F( x -
\mathrm{d} z ) \, \frac{A(z)^{\ell + 2}}{z} \le A(\tfrac{\delta}{\epsilon}%
x)^\ell \, I_1^+(\tfrac{\delta}{\epsilon}; x) \lesssim_{\epsilon} A(x)^\ell \,
I_1^+(\tfrac{\delta}{\epsilon}; x) \,.
\end{equation*}
This is negligible for \eqref{eq:ggoal}, for any fixed $\epsilon > 0$, by
the assumption that $I_1^+$ is a.n..

\smallskip

Finally, the contribution of $n \le A(\epsilon z)$
to the integral in \eqref{eq:cheint300} is, by \eqref{eq:Tell},
\begin{equation}  \label{eq:qula}
\int_{1 \le z < (1-\gamma)x} F( x - \mathrm{d} z ) \, T_{\ell+ 1}(\epsilon;
z) \,.
\end{equation}
By the inductive assumption \eqref{eq:assu}, for every $\eta > 0$ we can
choose $\epsilon > 0$ and $\bar x_{\epsilon} < \infty$ 
so that $T_{\ell+1}(\epsilon; z) \le \eta
\, b_{\ell + 2}(z)$ for $z \ge \bar x_{\epsilon}$. 
Then the integral in
\eqref{eq:qula} 
restricted to $z \ge \bar x_{\epsilon}$ is
\begin{equation*}
	\le \eta \, \int_{\bar{x}_{\epsilon} \le z < (1-\gamma)x} F( x - \mathrm{d} z ) \, 
	b_{\ell + 2}(z)
	\le \eta \, A(x)^{\ell} \, I_1^+(1-\gamma, x)
	\lesssim \eta \, A(x)^{\ell} \, b_1(x)
	= \eta \, b_{\ell+1}(x) \,,
\end{equation*}
where we have applied \eqref{eq:I1+O}.
If we let $x \to \infty$ and then $\eta \to 0$, this is negligible for \eqref{eq:ggoal}.
Finally,
by \eqref{eq:unibo} and \eqref{eq:nec0}, 
the integral in
\eqref{eq:qula} 
restricted to $z \le \bar x_{\epsilon}$ is, as $x \to\infty$
\begin{equation*}
\le C_{\bar x_{\epsilon},\ell+1} \, F((x-\bar x_{\epsilon}, x-1]) = o(b_{1}(x)) =
o(b_{\ell+1}(x)) \,.
\end{equation*}
This completes the proof.
\end{proof}

\smallskip

\section{Proof of Theorem~\protect\ref{th:mainrw}: 
sufficiency in case $\alpha \in (\frac{1}{3}, \frac{1}{2}]$}

\label{sec:casek=1}

Let $F$ be a probability on $\mathbb{R}$ that satisfies \eqref{eq:tail2}
with $p, q \ge 0$ and $\alpha \in (%
\frac{1}{3},\frac{1}{2}]$, i.e.\ $\kappa_{\alpha}=1$. \emph{We 
assume that $\tilde I_1(\delta; x)$ is a.n.}
(hence also $I_1(\delta; x)$ is a.n., recall  Remark~\ref{rem:refo0}),
and we deduce
\eqref{eq:SRTeq}, which is equivalent to the SRT. 
This proves the sufficiency in Theorem~\ref{th:mainrw}.

\smallskip

Let us set
\begin{equation}
Z_{1}:=M_{n}=\max \{X_{1},\ldots ,X_{n}\}\,.  \label{eq:Z1}
\end{equation}%
We fix $\gamma
\in (0,\frac{\alpha}{1-\alpha})$
and define the events
\begin{equation}
\begin{split}
& E_{1}^{(1)}:=\{Z_{1}\leq \gamma x\}\,,\qquad E_{1}^{(2)}:=\{|Z_{1}-x|\leq
a_{n}\}\,, \\
& E_{1}^{(3)}:=\{Z_{1}>\gamma x,\,|Z_{1}-x|>a_{n}\}
\end{split}
\label{eq:E123}
\end{equation}%
By Lemma~\ref{th:nobig} with $\ell =0$, we
already know that (with no extra assumptions on $F$)
\begin{equation}
\sum_{1\leq n\leq A(\delta x)}\P (S_{n}\in x+I,E_{1}^{(1)})\quad \text{is
always a.n.}\,.  \label{eq:E1an}
\end{equation}

Next we look at $E_{1}^{(2)}$. Note that by \eqref{eq:sup}
\begin{equation}
\begin{split}
& \P (S_{n}\in x+I,\,E_{1}^{(2)})\leq \sum_{i=1}^{n}\P (S_{n}\in
x+I,\,|X_{i}-x|\leq a_{n}) \\
& \qquad =n\int_{|y|\leq a_{n}}F(x+\mathrm{d}y)\,\P (S_{n-1}\in
I-y)\lesssim \frac{n}{a_{n}}\int_{|y|\leq a_{n}}F(x+\mathrm{d}y)\,.
\end{split}
\label{eq:thesa}
\end{equation}%
Using the fact that $n/a_{n}$ is regularly varying and recalling %
\eqref{eq:btildealt}, we obtain
\begin{equation} \label{eq:estE12}
\begin{split}
\sum_{1\leq n\leq A(\delta x)}\P (S_{n}\in x+I,\,E_{1}^{(2)})& \lesssim
\int_{|y|\leq \delta x}F(x+\mathrm{d}y)\sum_{A(|y|)\leq n\leq A(\delta x)}%
\frac{n}{a_{n}} \\
& \lesssim \int_{|y|\leq \delta x}F(x+\mathrm{d}y)\,\tilde{b}_{2}(\delta x,y)
= \tilde I_1(\delta; x) \,.
\end{split}%
\end{equation}%
Recalling %
\eqref{eq:tildeI1}, we have shown that
\begin{equation}
\tilde{I}_{1}(\delta ;x)\text{ is a.n.}
\qquad \Longrightarrow \qquad 
\sum_{1\leq n\leq A(\delta x)}\P (S_{n}\in x+I,E_{1}^{(2)})\quad \text{is
a.n.} \,.
\label{eq:E2an}
\end{equation}
(The reverse implication also holds, as shown in Section~\ref{sec:necce}.)

We finally turn to $E_{1}^{(3)}$. Arguing as in \eqref{eq:thesa} and setting
$\overline{\gamma }:=1-\gamma $, we have by \eqref{eq:ron2}
\begin{equation*}
\begin{split}
\P (S_{n}\in x+I,E_{1}^{(3)})& \lesssim n\int_{|y|>a_{n},\,y>-\overline{%
\gamma }x}F(x+\mathrm{d}y)\,\P (S_{n-1}\in I-y) \\
& \lesssim \frac{n^{2}}{a_{n}}\int_{|y|>a_{n},\,y>-\overline{\gamma }x}F(x+%
\mathrm{d}y)\,\frac{1}{A(y)}\,,
\end{split}%
\end{equation*}%
hence, recalling \eqref{eq:b}, we get
\begin{equation*}
\begin{split}
\sum_{1\leq n\leq A(\delta x)}\P (S_{n}\in x+I,\,E_{1}^{(3)})& \lesssim
\int_{y>-\overline{\gamma }x}F(x+\mathrm{d}y)\,\frac{1}{A(y)}\sum_{1\leq
n\leq A(\delta x\wedge |y|)}\frac{n^{2}}{a_{n}} \\
& \lesssim \int_{y>-\overline{\gamma }x}F(x+\mathrm{d}y)\,\frac{1}{A(y)}%
\,b_{3}(\delta x\wedge |y|)\,,
\end{split}%
\end{equation*}%
\emph{where the last inequality holds for $\alpha >\frac{1}{3}$,} thanks to %
\eqref{eq:Kar1}, because $n^{2}/a_{n}$ is regularly varying with index $%
2-1/\alpha >-1$. For fixed $\delta_0 > 0$, the right hand side can be estimated by
\begin{equation} \label{eq:secondin}
\begin{split}
& \int_{|y|\leq \delta_0 x}F(x+\mathrm{d}y)\,b_{2}(y)\,+\,b_{3}(\delta
x)\int_{y>-\overline{\gamma }x,\,|y|> \delta_0 x}\frac{F(x+\mathrm{d}y)}{A(y)} 
\\
& \qquad \lesssim I_{1}(\delta_0 ;x)\,+\,\frac{b_{3}(\delta x)}%
{A(\delta_0 x)}%
\,F((\gamma x,\infty ))\lesssim I_{1}(\delta_0 ;x)\,+\,\frac{b_{3}(\delta x)}%
{A(\delta_0 x) \, A(x)}  \,.
\end{split}%
\end{equation}%
By the a.n.\ of $I_1$,
given $\epsilon > 0$, we can fix $\delta_0 > 0$ small so that 
$I_{1}(\delta_0 ;x) \le \epsilon \, b_1(x)$ for large $x$.
Then we can fix $\delta > 0$ small (depending on $\delta_0$) so that
the second term in the right hand side of 
\eqref{eq:secondin} is also $\le \epsilon \, b_1(x)$ for large $x$, because
$b_3(\delta x) \sim \delta^{3\alpha-1}
\, b_3(x)$ and $\alpha > \frac{1}{3}$.
Thus
\begin{equation}
\sum_{1\leq n\leq A(\delta x)}\P (S_{n}\in x+I,E_{1}^{(3)})\quad \text{is
a.n.}\quad \text{if }\alpha >\frac{1}{3}\text{ and }I_{1}(\delta ;x)\text{
is a.n.}\,.  \label{eq:E3an}
\end{equation}
Relations \eqref{eq:E1an}, \eqref{eq:E2an}, \eqref{eq:E3an} prove the sufficiency
part in Theorem~%
\ref{th:main2}, when $\kappa_{\alpha} = 1$.

\smallskip

\section{Proof of Theorem~\protect\ref{th:mainrw}: sufficiency in case $\alpha \le \frac{1}{3}$}
\label{sec:casek>1}

Let $F$ be a probability on $\mathbb{R}$ that satisfies \eqref{eq:tail2}
with $p, q \ge 0$ and $\alpha \in (0,\frac{1}{3}]$. In this section, 
\emph{we assume that $\tilde I_{\kappa_{\alpha}}(\delta, \eta; x)$ is a.n.}
and we deduce \eqref{eq:SRTeq}, which is equivalent to the SRT.
By Remark~\ref{rem:refo0}, this proves the sufficiency part
in Theorem~\ref{th:mainrw}, in case $\kappa_{\alpha} \ge 2$.

We stress that our assumption that
$\tilde I_{\kappa_{\alpha}}(\delta, \eta; x)$ is a.n.\
ensures that \emph{$\tilde I_{r}(\delta, \eta; x)$ and $I_{r}(\delta, \eta; x)$ are
a.n.\ for every $r\in\N$}, by  Remark~\ref{rem:refo} (see
Lemmas~\ref{th:corequiv1}-\ref{th:corcascade}-\ref{th:corequiv2}).

Throughout this section we fix $\gamma \in (0,\frac{\alpha}{1-\alpha})$, we
choose $\eta = \overline\gamma := 1 - \gamma$ and we drop it from notations.
In particular, we write $I_{k}(\delta ;x)$ instead of $I_{k}(\delta, \eta ;x)$.

\subsection{Preparation}

We will prove that
$T(\delta;x) := \sum_{1 \le n \le A(\delta x)} \P (S_n \in x+I)$ is a.n.\
by direct path arguments, see Subsection~\ref{sec:suff}. This will lead us
to consider explicit quantities $J_{k}(\delta ;x)$, $\tilde J_{k}(\delta ;x)$
that generalize $I_{k}(\delta ;x)$, $\tilde I_{k}(\delta ;x)$.
For clarity, in this subsection we define such quantities and show that they are a.n..

\smallskip

We recall that $I_{k}(\delta ;x)$ is defined in \eqref{eq:I1rw}, %
\eqref{eq:Ik}. 
Let us rewrite it as follows:
\begin{align}
I_{k}(\delta ;x)& :=\int_{|y_{1}|\leq \delta x}F(x+\mathrm{d}%
y_{1})\,g_{k}(y_{1})\,,  \label{eq:Ikmod} \\
\rule{0pt}{2.8em}\text{where we set}\quad \ g_{k}(y_{1})& :=%
\begin{cases}
b_{2}(y_{1}) & \text{ if }k=1 \\
\rule{0pt}{1.9em}\displaystyle\int_{\Omega _{k}(y_{1})}P_{y_{1}}(\mathrm{d}%
y_{2},\ldots ,\mathrm{d}y_{k})\,b_{k+1}(y_{k}) & \text{ if }k\geq 2%
\end{cases}%
\,,  \label{eq:gak} \\
\rule{0pt}{1.3em} \Omega _{k}(y_{1})& :=\big\{(y_2, \ldots, y_k) \in \R^{k-1} 
: \ |y_{j}|\leq \overline{\gamma }%
|y_{j-1}|\;\text{for}\;2\leq j\leq k\big\} \,,
\label{eq:Omegak}
\end{align}%
and we recall that $P_{y_{1}}(\mathrm{d}y_{2},\ldots ,\mathrm{d}%
y_{k}):=F(-y_{1}+\mathrm{d}y_{2})F(-y_{2}+\mathrm{d}y_{3})\cdots F(-y_{k-1}+%
\mathrm{d}y_{k})$.

\smallskip

We define $J_k(\delta; x)$ by extending the integral in \eqref{eq:gak}
to a larger subset $\Theta _{k}(y_{1}) \supseteq \Omega _{k}(y_{1})$.
We introduce the shortcut
\begin{equation}
	\theta (y)  :=
\begin{cases}
(-\infty, \overline{\gamma} y] & \text{if } y \ge 0 \\
\rule{0pt}{1.1em} [\overline{\gamma} y, +\infty) & \text{if } y < 0 
\end{cases} \,,
\end{equation}
and note the important fact that (since $0 <\overline{\gamma} < 1$)
\begin{equation}
\int_{\theta (y)}F(-y+ \dd z)
= O\bigg(\frac{1}{A(|y|)}\bigg) \qquad
\text{ as }|y| \to \infty \,.
\label{theta}
\end{equation}%
Then, recalling that $\overline{\gamma} = 1-\gamma$,
we set for $k\ge 2$ and $y_1 \in \R$
\begin{equation}  \label{y}
\begin{split}
\Theta_k(y_1) 
& := \big\{ (y_2, \ldots, y_k) \in \R^{k-1} : \
y_j \in \theta(y_{j-1}) \; \text{for} \; 2 \le j \le k \big\} \,.
\end{split}%
\end{equation}
We then define 
\begin{align}  \label{K}
J_k(\delta; x) & := \int_{|y| \le \delta x} F(x + \mathrm{d} y) \, h_k(y) \,.
\end{align}
where $h_k(y_1)$ is nothing but $g_k(y_1)$, see \eqref{eq:gak},
with $\Omega_k(y_1)$ replaced by $\Theta_k(y_1)$:
\begin{align}  \label{L}
h_k(y_1) := 
\begin{cases}
b_{2}(y_{1}) & \text{ if }k=1 \\
\rule{0pt}{1.9em}\displaystyle\int_{\Theta_k(y_1)}P_{y_{1}}(\mathrm{d}%
y_{2},\ldots ,\mathrm{d}y_{k})\,b_{k+1}(y_{k}) & \text{ if }k\geq 2%
\end{cases} \,.
\end{align}

It will be useful to consider a slight generalization of $h_k(y_1)$:
for any non-negative, even function $f:\mathbb{R}\rightarrow [0,\infty)$ we define
\begin{align}  \label{Lf}
h_k(y_1,f) := 
\begin{cases}
b_{1}(y_{1})\,f(y_{1}) & \text{ if }k=1 \\
\rule{0pt}{1.9em}\displaystyle\int_{\Theta_{k}(y_{1})}P_{y_{1}}(\mathrm{d}%
y_{2},\ldots ,\mathrm{d}y_{k})\,b_{k}(y_{k})\,f(y_{k}) & \text{ if }k\geq 2%
\end{cases} \,,
\end{align}
and note that $h_k(y_1,A) = h_k(y_1)$.

\smallskip

The next proposition shows that $J_k(\delta;x)$ is a.n.\ 
and provides a useful auxiliary estimate.
Its proof is quite tedious and is deferred to
Subsection~\ref{sec:technical}.

\begin{proposition}
\label{G} 
Fix $\ell \in\N$ and $\alpha \in (0,\frac{1}{2})$.
If $\ell \ge 2$, assume that $I_{j}(\delta ;x)$ is a.n.\ 
for $j=1,\ldots, \ell-1$.
\begin{enumerate}
\item\label{it:1} Fix any $f \in RV(\beta)$
with $0 < \beta < 1- \ell\alpha$. Then for all $0 < \delta_0 < \kappa <1$
\begin{equation}
\int_{|y|>\delta _{0}x,\, y>-\kappa x} F(x+\mathrm{d} y) \,
h_{\ell}(y,f)\lesssim_{\delta_0, \kappa,\gamma,\ell} \frac{f(x)}{x\vee 1} 
\qquad (\forall x \ge 0)\,.
\label{m}
\end{equation}
\item\label{it:2} Assume that $I_\ell(\delta;x)$ is a.n.\ too and, moreover,
$\alpha < \frac{1}{\ell+1}$. Then
\begin{equation}
J_{\ell}(\delta ;x)\text{ is a.n.}  \label{n}
\end{equation}%
\end{enumerate}
\end{proposition}

\smallskip

We finally define $\tilde{J}_{2}(\delta ;x) := \tilde{I}_{2}(\delta ;x)$ and,
for $k \ge 3$,
\begin{equation} \label{eq:tildeJ}
	\tilde{J}_{k}(\delta ;x)=\int_{|y_1|\leq \delta x}F(x+\mathrm{d}y_1)\,
	\tilde{h}_{k}(y_1)\,,
\end{equation}%
where we set (recall that $\tilde{b}(x,z)$ is defined in \eqref{eq:btilde},
or equivalently \eqref{eq:btildealt}):
\begin{equation}
\tilde{h}_{k}(y_1)=\int_{\Theta _{k-1}(y_1)\cap \{|y_{k}|\leq \overline{\gamma }%
|y_{k-1}|\}}P_{y}(\mathrm{d}y_{2},\cdots ,\mathrm{d}y_{k})\,\tilde{b}_{k+1}
(y_{k-1},y_{k}) \,.  \label{h}
\end{equation}%

\smallskip

The next result shows that $\tilde J_k(\delta;x)$ is a.n..
Its proof is also deferred to
Subsection~\ref{sec:technical}.

\begin{proposition}
\label{H} 
Fix $\ell \in\N$ with $\ell \ge 2$ and $0 < \alpha \le \frac{1}{\ell + 1}$.
Assume that $I_j(\delta;x)$ and $\tilde I_j(\delta; x)$ are a.n.\ for $j=1,\ldots, \ell$.
Then also $\tilde J_{\ell}(\delta ;x)$ is a.n..
\end{proposition}

\subsection{Proof of Sufficiency for Theorem~\protect\ref{th:mainrw}}
\label{sec:suff}

Throughout the proof we fix $\alpha \in (0,\frac{1}{3}]$
and $k = \kappa_{\alpha} = \lfloor
1/\alpha \rfloor - 1$, see \eqref{eq:kappaalpha}.
We stress that $k \ge 2$ and $\frac{1}{k+2} < \alpha \le \frac{1}{k+1}$.
Our goal is to prove \eqref{eq:SRTeq}.

\smallskip

We generalize \eqref{eq:Z1}, defining two sequences $Z_{1},Z_{2},\ldots Z_{k}$
and $Y_{1},Y_{2},\ldots, Y_{k}$ as follows:
\begin{equation*}
Z_{1}:=\max \{X_{1},\ldots ,X_{n}\}\,,\qquad Y_{1}:=Z_{1}-x\,,
\end{equation*}%
and for $r \in \{2, \ldots, k\}$
\begin{align}
Z_{r}& :=%
\begin{cases}
\max \big\{\{X_{j},1\leq j\leq n\}\setminus \{Z_{j},1\leq j\leq r-1\}\big\}
& \text{if }Y_{r-1}\leq 0, \\
\rule{0pt}{1.4em}\min \big\{\{X_{j},1\leq j\leq n\}\setminus \{Z_{j},1\leq
j\leq r-1\}\big\} & \text{if }Y_{r-1}>0.%
\end{cases} 
\label{eq:Zer}
\\
Y_{r}& :=\sum_{i=1}^{r}Z_{i}-x
\label{eq:Yer}
\end{align}%
Intuitively, \emph{$Z_{r}$ is the largest available step towards $%
x$ from $Z_{1}+\ldots +Z_{r-1}$}.

In fact, we may assume that the following holds:
\begin{equation}\label{eq:standard}
	Z_r > 0 \quad \text{if } Y_{r-1} \le 0 \,, \qquad
	\text{while} \qquad
	Z_r < 0 \quad \text{if } Y_{r-1} > 0 \,,
\end{equation}
because, as we now show, the event that \eqref{eq:standard} fails to be true
is negligible. This event occurs if, for some $%
r\leq k$, either $Y_{r-1}\leq 0$ and $\{X_{j},1\leq j\leq n\}\setminus
\{Z_{j},1\leq j\leq r-1\}$ contains no positive terms
or $Y_{r-1}>0$ and
this set contains no negative terms. 
Call $\cE_{n,r}$ such event and 
recall that $I = (-h,0]$.
We first observe that
$\P(\cE_{n,r}, \, S_{n}\in x+I , \, Y_{r-1}\notin I)=0$ for all $n\geq r$
(if $Y_{r-1} > 0$ then $S_{n} - x \ge Y_{r-1} > 0$
on the event $\cE_{n,r}$, and similarly if $Y_{r-1} \le -h$ 
then $S_{n} - x \le Y_{r-1} \le -h$). Next we observe that
\begin{equation*}
\begin{split}
\P(\cE_{n,r}, \, Y_{r-1}\in I)
	& \le \binom{n}{r-1}\P(S_{r-1}\in x+I, \ X_r, X_{r+1}, \ldots, X_n \le 0 ) \\
	& \le n^{r-1} \, c^{n-(r-1)} \, \P(S_{r-1} \in x+I) \,,
\end{split}	
\end{equation*}
with $c = \P(X_1 \le 0) < 1$.
If we set $K_r := \sum_{n=r}^\infty n^{r-1} \, c^{n-(r-1)} < \infty$, we 
can write, by \eqref{eq:nec1},
\begin{equation*}
	\sum_{n=r}^{\infty }
	\P(\cE_{n,r}, \, S_{n}\in x+I) 
	\le \sum_{n=r}^{\infty} \P(\cE_{n,r}, \, Y_{r-1}\in I)
	\le K_r \, \P(S_{r-1}\in x+I)
	\underset{x\to\infty}{=}
	o\big(b_1(x) \big) \,,
\end{equation*}
so the contribution of $\cE_{n,r}$ to \eqref{eq:SRTeq} is negligible.
Henceforth we will assume that \eqref{eq:standard} holds.

We cover the probability space 
$\Omega \subseteq E^{(1)}_1 \cup E^{(2)}_1 \cup E^{(3)}_1$, where
we recall from \eqref{eq:E123} that
\begin{equation*}
E_{1}^{(1)}:=\{Z_{1}\leq \gamma x\}\,,\qquad E_{1}^{(2)}:=\{|Y_1|\leq
a_{n}\}\,, \qquad
E_{1}^{(3)}=\{Z_{1}>\gamma x, \ |Y_{1}|>a_{n}\} \,.
\end{equation*}%
The argument to show that $\sum_{1 \le n \le A(\delta x)}
\P (S_{n} \in x+I,E_{1}^{(1)}\cup E_{1}^{(2)})$
is a.n.\ presented in Section~\ref{sec:casek=1} is still valid,
see \eqref{eq:E1an} and \eqref{eq:E2an}, so it
remains to focus on $E_{1}^{(3)}$.

\smallskip

We introduce the constants $C_{r}=(\overline{\gamma })^{r-1}$
and then the events $E_{r}^{(1)}, E_{r}^{(2)}, E_{r}^{(3)}$ for $r \ge 2$
by%
\begin{gather}  \label{eq:Er12}
E_{r}^{(1)} =E_{r-1}^{(3)}\cap \{|Z_{r}|\leq \gamma |Y_{r-1}|\} \,, \qquad
E_{r}^{(2)}=E_{r-1}^{(3)}\cap \{|Y_{r}|\leq  C_{r}a_{n}\} \,, \\
E_{r}^{(3)} = E_{r-1}^{(3)}\cap \{|Z_{r}|>\gamma |Y_{r-1}|,\
|Y_{r}|> C_{r}a_{n}\} \,.
\notag
\end{gather}%
Note that we can decompose (recall that $k = \kappa_{\alpha}$ is fixed)
\begin{equation*}
E^{(3)}_1 \subseteq \bigcup_{r=1}^k E_r^{(1)} \ \cup \ \bigcup_{r=1}^k E_r^{(2)} \ \cup
\ E^{(3)}_k \,.
\end{equation*}
We will
show that $\sum_{1 \le n \le A(\delta x)}\P (S_{n} \in x+I)$ is a.n.\
by estimating the contributions of $E_{r}^{(1)}$ and 
$E_{r}^{(2)}$, for $2\leq r\leq k$, and finally the contribution of
$E_{k}^{(3)}$. Let us note that
\begin{equation}\label{eq:westre}
	\alpha \le \tfrac{1}{k+1} < \tfrac{1}{r} \,, \qquad
	\text{for } r=2,\ldots, k \,,
\end{equation}
which allows us to apply equation~\eqref{n} from Proposition~\ref{G} for $\ell = r-1$
(but not for $\ell = r$, unlike Proposition~\ref{H}).

\begin{remark}
We can rewrite $E_{\ell}^{(3)}$ more explicitly as follows:
\begin{gather*}
E_{\ell}^{(3)} = \{Z_1 > \gamma x \,, \ |Z_i| > \gamma |Y_{i-1}| \text{ for } 2
\le i \le \ell \} \cap \{ |Y_i| > C_{i} a_n \,, \text{ for } 1 \le i \le \ell\} \,.
\end{gather*}
Recalling the definition \eqref{y} of $\Theta_k(\cdot)$, 
we claim that $E_{\ell}^{(3)}$ can also be rewritten as
\begin{gather}  \label{eq:Er3}
E_{\ell}^{(3)} = \{ 
Y_1 > -\overline{\gamma} x, \ (Y_2, \ldots, Y_{\ell}) \in \Theta_{\ell} (Y_1)
\} 
\cap \{ |Y_i| > C_{i} a_n \,, \text{ for } 1 \le i \le \ell\} \,.
\end{gather}
To prove the claim, we show that $|Z_i| > \gamma |Y_{i-1}|$ is
equivalent to $Y_i \in \theta(Y_{i-1})$, for $2 \le i \le k$, with $\theta(\cdot)$ 
defined in \eqref{y}. 
We recall that $Z_i = Y_i - Y_{i-1}$, see \eqref{eq:Yer}. If $Y_{i-1} > 0$, then $Z_i
\le 0$ by \eqref{eq:standard}, hence $|Z_i| > \gamma |Y_{i-1}|$ becomes $%
Y_{i-1} - Y_i > \gamma Y_{i-1}$, which is precisely $Y_i \in (-\infty, -\overline{\gamma} Y_{i-1})
= \theta(Y_{i-1})$. Similar arguments apply if $Y_{i-1}
\le 0$, in which case $Z_i \ge 0$.
\end{remark}

\smallskip

\subsubsection{Estimate of $E_{r}^{(1)}$}

We fix $r\in \{2,\ldots ,k\}$.
By exchangeability,
\begin{equation}
\P (S_{n}\in x+I,\,E_{r}^{(1)})\leq n^{r-1}\,\P \big((Z_{1},\ldots
Z_{r-1})=(X_{1},\ldots ,X_{r-1}),\,S_{n}\in x+I,\,E_{r}^{(1)}\big)\,.
\label{eq:thist}
\end{equation}%
Conditionally on $(X_{1},\ldots ,X_{r-1})=(z_{1},\ldots ,z_{r-1})$,
we have $S_{n}=(z_{1}+\ldots +z_{r-1})+S_{n-(r-1)}^{\prime }$, where we set $%
S_{k}^{\prime }:=X_{1}^{\prime }+\ldots +X_{k}^{\prime }$ with $%
X_{i}^{\prime }:=X_{(r-1)+i}$. Motivated by \eqref{eq:Yer}, if we set
\begin{equation*}
	y_{i}:=(z_{1}+\ldots +z_{i})-x \qquad \text{for } i=1,\ldots, r-1 \,,
\end{equation*}%
then we can write $\{S_{n}\in x+I\}=\{S_{n-(r-1)}^{\prime }\in -y_{r-1}+I\}$. Assume
first that $y_{r-1}\leq 0$, so $Z_{r}=M_{n-(r-1)}^{\prime } := \max
\{X_{i}^{\prime },1\leq i\leq n-(r-1)\}$. By \eqref{eq:Er12} and %
\eqref{eq:Zer}, we need to evaluate
\begin{equation}
\P \Big(S_{n-(r-1)}^{\prime }\in -y_{r-1}+I,\ |M_{n-(r-1)}^{\prime }|\leq
\gamma \,|y_{r-1}|\Big)\,.  \label{eq:als}
\end{equation}%
Since this probability is increasing in $\gamma $, applying \eqref{eq:ron1},
we get the bound
\begin{equation}
\lesssim \frac{1}{a_{n}}\,\bigg(\frac{n}{A(|y_{r-1}|)}\bigg)^{d}\,,\qquad
\text{for all}\quad d\leq \frac{1}{\gamma }\,.  \label{eq:tebo}
\end{equation}%
In case $y_{r-1}>0$, relation \eqref{eq:als} holds with $M_{k}^{\prime }$
replaced by $(M_{k}^{\prime })^{\ast }:=\min_{1\leq i\leq k}X_{i}^{\prime }$. 
Applying \eqref{eq:ron1} to the reflected walk $(S^{\prime })^{\ast}=-S^{\prime }$, 
we see that the bound \eqref{eq:tebo} still holds. 
Then, by \eqref{eq:thist}
and $E_{r}^{(1)} = E_{r-1}^{(3)}\cap \{|Z_{r}|\leq \gamma |Y_{r-1}|\}$,
using \eqref{eq:Er3} for $E_{r-1}^{(3)}$,
we have the key bound
\begin{equation}
\P (S_{n}\in x+I,\,E_{r}^{(1)})\lesssim \!\!\!\!\int\limits_{\substack{ %
y_{1}>-\overline{\gamma }x,  \\ \,(y_{2},\ldots ,y_{r-1})\in \Theta_{r-1}(y_{1}),  
\\ \tilde{y}_{r-1}\geq a_{n}}}
\!\!\!\!\!\! F(x+\mathrm{%
d}y_{1})\,P_{y_{1}}(\mathrm{d}y_{2},\cdots ,\mathrm{d}y_{r-1})\,\frac{%
n^{r+d-1}}{a_{n}\,A(|y_{r-1}|)^{d}}\,,
\label{eq:ussbo}
\end{equation}
where we set
\begin{equation}\label{eq:ytild}
	\tilde{y}_{j}:=\min \{C_{i}^{-1} \, |y_{i}|,1\leq i\leq j\}\,,
	\qquad \text{for } j\geq 1 \,.
\end{equation}%
(For $r=2$ the integral in \eqref{eq:ussbo} is only over $y_1$, so
the restriction $(y_{2},\ldots ,y_{r-1})\in \Theta_{r-1}(y_{1})$
and the term $P_{y_{1}}(\mathrm{d}y_{2},\cdots ,\mathrm{d}y_{r-1})$ should
be ignored.)

Henceforth we fix $d\in (\frac{1}{\alpha }-r,\frac{1}{\alpha }- r + 1)$. Since $%
\gamma <\frac{\alpha }{1-\alpha }$ by assumption, and $r\geq 2$, 
we have $\frac{1}{\gamma }>\frac{1}{\alpha }-1 \geq \frac{1}{\alpha }- r + 1
>d$, hence
the constraint $d\leq \frac{1}{\gamma }$ is satisfied. 
The sequence $n^{r+d-1}/a_{n}$ is regularly varying with exponent $(d+r-1)-%
\frac{1}{\alpha }>-1$, hence by \eqref{eq:Kar1}
\begin{equation*}
\sum_{1\leq n\leq A(w)}\frac{n^{r+d-1}}{a_{n}}\lesssim \frac{A(w)^{r+d}}{%
w\wedge 1}=b_{r+d}(w)\,,\qquad \forall w\geq 0\,,
\end{equation*}%
where we recall that $b_{k}(\cdot )$ was defined in \eqref{eq:b}. 
Since the integral in \eqref{eq:ussbo} is restricted to 
$n\leq A(\tilde{y}_{r-1})$, we see that
\begin{equation}
\begin{split}
& \sum_{1\leq n\leq A(\delta x)}\P (S_{n}\in x+I,\,E_{r}^{(1)}) \\
& \qquad \qquad \lesssim \!\int\limits_{\substack{ y_{1}>-\overline{\gamma }%
x  \\ \,(y_{2},\ldots ,y_{r-1})\in \Theta _{r-1}(y_{1})}}F(x+\mathrm{d}%
y_{1})\,P_{y_{1}}(\mathrm{d}y_{2},\cdots ,\mathrm{d}y_{r-1})\,\frac{b_{r+d}(%
\tilde{y}_{r-1}\wedge \delta x)}{A(|y_{r-1}|)^{d}}\,.
\end{split}
\label{eq:int1r}
\end{equation}%
We split the
integral in two terms, corresponding to $|y_{1}|\leq \delta_0 x$ and $|y_{1}|>\delta_0 x$.
Given $\epsilon > 0$,
we first show that for $\delta_0 > 0$ small enough
the first term is $\le \epsilon \, b_1(x)$, for large $x$.
We then show that
for $\delta > 0$ small enough (depending on $\delta_0$)
the second term is also
$\le \epsilon \, b_1(x)$, for 
large $x$.
Altogether, this proves that \eqref{eq:int1r} is a.n. and
completes the estimate of $E_{r}^{(1)}$.

\begin{itemize}
\item \emph{First term.}
Since $\tilde{y}_{r-1}\leq |y_{r-1}| /C_{r-1}$, see \eqref{eq:ytild},
and since $b_{r+d}(\cdot )$ is asymptotically
increasing (because $r+d>\frac{1}{\alpha }$), we have
\begin{equation*}
\frac{b_{r+d}(\tilde{y}_{r-1}\wedge \delta x)}{A(|y_{r-1}|)^{d}}
\lesssim_{r} \frac{%
b_{r+d}(|y_{r-1}|)}{A(|y_{r-1}|)^{d}}=b_{r}(|y_{r-1}|)\,,
\end{equation*}%
hence the contribution of $|y_{1}|\leq \delta_0 x$ to the integral in %
\eqref{eq:int1r} is bounded by
\begin{equation*}
\int\limits_{\substack{ |y_{1}|\leq \delta_0 x  \\ \,(y_{2},\ldots
,y_{r-1})\in \Theta _{r-1}(y_{1})}}F(x+\mathrm{d}y_{1})\,P_{y_{1}}(\mathrm{d}%
y_{2},\cdots ,\mathrm{d}y_{r-1})\,b_{r}(|y_{r-1}|)=:J_{r-1}(\delta_0 ; x)\,,
\end{equation*}%
see \eqref{K} for the definition of $J_k$. By Proposition~\ref{G}
with $\ell = r-1$ (recall \eqref{eq:westre}),
we can fix $\delta_0 > 0$ small enough
so that for large $x$ we have $J_{r-1}(\delta_0 ,x) \le \epsilon \, b_1(x)$.

\item \emph{Second term.} If we define $f_{1}(y):=1/b_{r+d-1}(y)$, 
then by \eqref{Lf} we
can write
\begin{equation*}
\int\limits_{(y_{2},\ldots ,y_{r-1})\in \Theta _{r-1}(y_{1})}P_{y_{1}}(%
\mathrm{d}y_{2},\cdots ,\mathrm{d}y_{r-1})\,\frac{1}{A(|y_{r-1}|)^{d}}%
=h_{r-1}(y_{1},f_{1})\,, \qquad \forall y_1 \in \R \,.
\end{equation*}%
As a consequence, the contribution of $|y_{1}|> \delta_0 x$ to the integral
in \eqref{eq:int1r} is at most
\begin{equation} \label{eq:wesee}
	b_{r+d}(\delta x)\int\limits_{\substack{ |y_{1}|> \delta_0 x,\, y_{1}>-%
	\overline{\gamma }x}}F(x+\mathrm{d}y_{1})\,h_{r-1}(y_{1},f_{1})\,.
\end{equation}%
Note that $f_{1}(\cdot )\in RV(\beta )$ with $\beta =1-(r+d-1)\alpha $. Our
choice of $d$ implies that $\frac{1}{\alpha }<r+d<\frac{1}{\alpha }+1$,
hence $0<\beta <\alpha $. 
Since $\alpha < \frac{1}{r}$, see \eqref{eq:westre},
we also have $\alpha < 1 - (r-1)\alpha$, which yields $0<\beta < 1 - (r-1)\alpha$.
By Proposition~\ref{G} with $f=f_{1}$ and $\ell = r-1$,
the expression in \eqref{eq:wesee} is
$\lesssim_{\delta_0, \gamma, r} b_{r+d}(\delta x) \, \frac{f_1(x)}{x}
\lesssim \frac{A(\delta x)}{x}$. Then
we can fix $\delta > 0$ small (depending on $\delta_0$) so that
it is $\le \epsilon \, b_1(x)$ for large $x$.
\end{itemize}

\subsubsection{Estimate of $E_{r}^{(2)}$}

Always for $r \in \{2,
\ldots, k\}$, in analogy with \eqref{eq:thist}, we have
\begin{equation}  \label{eq:thist2}
\begin{split}
& \P (S_{n}\in x+I, \, E_{r}^{(2)}) \le n^{r} \, \P \big( (Z_1, \ldots
Z_{r}) = (X_1, \ldots, X_{r}), \, S_{n}\in x+I, \, E_{r}^{(2)} \big) \\
& \qquad\qquad \le n^{r} \, \P \big( (Z_1, \ldots Z_{r}) = (X_1, \ldots,
X_{r}), \, E_{r}^{(2)} \big) \, \bigg( \sup_{z \in \mathbb{R}} \P %
(S_{n-r}\in z+I) \bigg) \\
& \qquad\qquad \lesssim \frac{n^{r}}{a_n} \, \P \big( (Z_1, \ldots Z_{r}) =
(X_1, \ldots, X_{r}), \, E_{r}^{(2)} \big) \,,
\end{split}%
\end{equation}
where we have applied \eqref{eq:sup}. 
Since $E_{r}^{(2)}=E_{r-1}^{(3)}\cap \{|Y_{r}|\leq  C_{r}a_{n}\}$, by
\eqref{eq:Er3} and \eqref{eq:ytild} we obtain
\begin{equation}  \label{eq:ussbo2}
\begin{split}
\P (S_{n}\in x+I, \, E_{r}^{(2)}) & \lesssim \!\!\!\!\! \int\limits
_{\substack{ y_1 > -\overline{\gamma} x, \ \, (y_2, \ldots, y_{r-1}) \in
\Theta _{r-1}(y_1),  \\ \tilde{y}_{r-1}\geq a_{n},\ |y_{r}|<C_{r}a_{n}}} \!\!\!\!\!\!\!\!\!\! F(x + \mathrm{d} y_{1}) \, P_{y_{1}}(%
\mathrm{d} y_{2},\cdots , \mathrm{d} y_{r}) \, \frac{n^{r}}{a_{n}} \,.
\end{split}%
\end{equation}%
Recalling \eqref{eq:ytild} and \eqref{eq:btildealt}, we can write
\begin{align}
& \sum_{1\le n \le A(\delta x)} \P (S_{n}\in x+I, \, E_{r}^{(2)}) \lesssim
\!\!\!\!\!\!\!\!\! \int\limits_{\substack{ y_1 > -\overline{\gamma} x,  \\ %
\, (y_2, \ldots, y_{r-1}) \in \Theta _{r-1}(y_1) \\
|y_r| \le \overline{\gamma} |y_{r-1}|}} \!\!\!\!\!\!\!\!\!\! F(x
+ \mathrm{d} y_{1}) \, P_{y_{1}}(\mathrm{d} y_{2},\cdots , \mathrm{d} y_{r})
\, \sum_{n=A(C_{r}^{-1}|y_r|)}^{A(\delta x \wedge \tilde y_{r-1})}\frac{n^{r}}{a_{n}}
\nonumber
\\
& \quad \qquad = \int\limits_{\substack{ y_1 > -%
\overline{\gamma} x,  \\ \, (y_2, \ldots, y_{r-1}) \in \Theta _{r-1}(y_1) \\
|y_r| \le \overline{\gamma} |y_{r-1}|}}
\!\!\!\!\!\!\!\!\!\! F(x + \mathrm{d} y_{1}) \, P_{y_{1}}(\mathrm{d}
y_{2},\cdots , \mathrm{d} y_{r}) \, \tilde{b}_{r+1}(\delta x\wedge \tilde
y_{r-1} , C_{r}^{-1}|y_{r}|)  \,,
\label{eq:deint}
\end{align}
where we made explicit the restriction $|y_r| \le \overline{\gamma} |y_{r-1}|$
in the domain of integration, because for $|y_r| > \overline{\gamma} |y_{r-1}|$
the integrand vanishes (since $\tilde y_{r-1} < C_{r-1}^{-1} |y_{r-1}|$
and $C_r = \overline{\gamma}^{r-1}$).

We split the integral \eqref{eq:deint} 
in two terms, i.e.\ $|y_{1}|\leq \delta_0 x$ and $|y_{1}|>\delta_0 x$.
First we show that, given any $\epsilon > 0$,
the first term is $\le \epsilon \, b_1(x)$ for $\delta_0 > 0$ small
and $x$ large.
Then we show that the second term is $\le \epsilon \, b_1(x)$ for $\delta > 0$ small
(depending on $\delta_0$) and $x$ large.

\begin{itemize}
\item \emph{First term.} Recalling \eqref{eq:ytild}, \eqref{eq:btildealt} and
the definition $C_i = \overline{\gamma}^{i-1}$, we can bound
\begin{equation*}
\begin{split}
	\tilde{b}_{r+1}(\delta x\wedge \tilde{y}_{r-1}, C_{r}^{-1} |y_r|) 
	&\leq \tilde{b}_{r+1}(C_{r-1}^{-1}|y_{r-1}|,C_{r-1}^{-1} \overline{\gamma}^{-1} |y_{r}|)
	\lesssim \tilde{b}_{r+1}(|y_{r-1}|,\overline{\gamma}^{-1}|y_{r}|) \\
	& \le \tilde{b}_{r+1}(|y_{r-1}|,|y_{r}|) \,,
\end{split}
\end{equation*}
where the last inequality holds because $\overline{\gamma}^{-1} > 1$.
For $|y_{1}|\leq \delta _{0}x$, when we plug this into \eqref{eq:deint}
we obtain $\tilde{J}_{r}(\delta_{0}; x)$, see \eqref{eq:tildeJ} and \eqref{h}. 
By Proposition~\ref{H}
with $\ell = r$, we can fix $\delta _{0}>0$ small enough so
that $\tilde{J}_{r}(\delta_{0}; x) \leq \epsilon
\,b_{1}(x)$ for large $x$.

\item \emph{Second term.} Next we deal with $|y_{1}| > \delta_0 x$.
Note that $\alpha (r+1) \le \alpha(k+1) \le 1$, see \eqref{eq:westre}.
We fix any
$\psi \in (0,1)$, so that $\alpha (r+1-\psi )<1$. By \eqref{eq:btildealt} we can bound
\begin{align*}
	\tilde{b}_{r+1}(\delta x\wedge \tilde{y}_{r-1}, C_{r}^{-1}|y_{r}|) 
	& \lesssim A(\delta x)^{\psi } \,
	\sum_{n \ge A(C_r^{-1} |y_r|)} \frac{n^{r-\psi}}{a_n}
	\lesssim_{r} A(\delta x)^{\psi } \, b_{r+1-\psi }(|y_{r}|) \,,
\end{align*}%
where the last inequality holds by \eqref{eq:Kar2} (note that
$n^{r-\psi}/a_n$ is regularly varying with index $r-\psi-\frac{1}{\alpha} < -1$).
If we set $f_{2}(y):=A(y)^{1-\psi}$,
the contribution of $|y_{1}| > \delta_0 x$ is
\begin{equation*}
	\lesssim A(\delta x)^{\psi }\int_{y_{1}>-\overline{\gamma }%
	x,|y_{1}|>\delta_0 x}F(x+\mathrm{d}y_{1}) \, h_{r}(y_{1},f_{2}).
\end{equation*}%
Note that $f_{2}(y):=A(y)^{1-\psi} \in RV(\beta)$, with $\beta = \alpha(1-\psi)$,
hence $0 < \beta < 1 - r \alpha$ by our choice
of $\psi$. We can apply point~\eqref{it:1}
in Proposition~\ref{G}
with $\ell = r$, to get
\begin{equation*}
	\lesssim A(\delta x)^{\psi } \, \frac{f_2(x)}{x \vee 1}
	\underset{x\to\infty}{\sim} \delta^{\alpha \psi} 
	\, \frac{A(x)}{x \vee 1} = \delta^{\alpha \psi} \, b_1(x) \,,
\end{equation*}
which is a.n..
\end{itemize}

\subsubsection{Estimate of $E_{k}^{(3)}$}

Finally, recalling \eqref{eq:Er3},
\eqref{eq:ytild} and applying \eqref{eq:ron2}, we can write
\begin{equation*}
\begin{split}
	\P(S_{n} \in x+I, \ E_{k}^{(3)})
	& \lesssim 
	n^{k} \, \P \big( (Z_1, \ldots
	Z_{k}) = (X_1, \ldots, X_{k}), \, S_{n}\in x+I, \, E_{k}^{(3)} \big) \\
	& \lesssim
	\int\limits_{\substack{ y_{1}>-\overline{\gamma }x,\\ \,(y_{2},\ldots ,y_{k})
	\in \Theta _{k}(y_{1}),  \\ 
	\tilde{y}_{k}\geq a_{n} }}\!F(x+\mathrm{d}y_{1})\,
	P_{y_{1}}(\mathrm{d}y_{2},\cdots ,\mathrm{d}y_{k})\,\frac{n^{k+1}}{a_{n}A(y_{k})}.
\end{split}
\end{equation*}
Note that $n^{k+1}/a_n \in RV(\zeta)$ with $\zeta = k+1-\frac{1}{\alpha} > -1$,
by $k=\kappa_{\alpha}$.
Therefore, by \eqref{eq:Kar1},
\begin{align} \label{eq:tirdte}
&\sum_{1 \le n \le A(\delta x)}
\!\!\!\!
\P(S_{n}\in x+I,E_{k}^{(3)})\lesssim
\!\!\!\!\!\!\!\!\!\!
\int\limits_{\substack{y_{1}>-\overline{\gamma }x,\\ \,(y_{2},\ldots ,y_{k})\in \Theta
_{k}(y_{1})}}
\!\!\!\!\!\!\!\!\!\!
F(x+\mathrm{d}y_{1})\,P_{y_{1}}(\mathrm{d}y_{2},\cdots ,\mathrm{%
d}y_{k})
\frac{b_{k+2}(\delta x\wedge \tilde{y}_{k})}{A(y_{k})} \,.
\end{align}%
We split the
integral in two terms $|y_{1}|\leq \delta_0 x$ and $|y_{1}|>\delta_0 x$.
We recall that $\frac{1}{k+2} < \alpha \le \frac{1}{k+1}$.

\begin{itemize}
\item \emph{First term.} 
We focus on $|y_{1}|\leq \delta_0 x$
and distinguish two cases.
First we consider $|y_k| \le \overline{\gamma}
|y_{k-1}|$. By $(k+1) - \frac{1}{\alpha} > -1$
and $\tilde y_k \lesssim_k \min\{|y_{k-1}|, |y_{k}|\}$, see \eqref{eq:ytild}, we get
\begin{equation*}
\begin{split}
	b_{k+2}(\delta x\wedge \tilde{y}_{k})
	& \lesssim \sum_{n=1}^{A(\delta x\wedge \tilde{y}_{k})}
	\frac{n^{k+1}}{a_n} 
	\lesssim A(y_k) \, \sum_{n=1}^{A(|y_{k-1}|)}
	\frac{n^{k}}{a_n} 
	\lesssim_{\gamma} A(y_k) \, \sum_{n=A(|y_{k}|)}^{A(|y_{k-1}|)}
	\frac{n^{k}}{a_n} \\
	& \lesssim A(y_k) \, \tilde b_{k+1}(y_{k-1}, y_k) \,,
\end{split}
\end{equation*}
where the third inequality holds for $|y_k| \le \overline{\gamma} |y_{k-1}|$,
and for the last inequality we recall \eqref{eq:btildealt}.
When we plug this bound into \eqref{eq:tirdte}, with the integral restricted
to $|y_{1}|\leq \delta_0 x$ and $|y_k| \le \overline{\gamma} |y_{k-1}|$,
we obtain $\tilde{J}_{r}(\delta_{0}; x)$, see \eqref{eq:tildeJ}-\eqref{h}. 
By Proposition~\ref{H} with $\ell = r$, we can fix $\delta _{0}>0$ small enough so
that $\tilde{J}_{r}(\delta_{0}; x) \leq \epsilon \,b_{1}(x)$.

Next we consider $|y_k| > \overline{\gamma} |y_{k-1}|$,
hence we can bound $A(y_k) \gtrsim_{\gamma} A(y_{k-1})$.
Since $b_{k+2}$ is asymptotically increasing (it is regularly varying with index
$(k+2)\alpha-1 > 0$), we can
also bound $b_{k+2}(\delta x\wedge \tilde{y}_{k}) / A(y_k)
\lesssim b_{k+2}(y_{k-1}) / A(y_{k-1}) = b_{k+1}(y_{k-1})$.
When we plug this into \eqref{eq:tirdte}, the integrand does not
depend on $y_k$ anymore, so we can integrate
over $y_k$ to get $\int_{y_k \in \theta(y_{k-1})} F(-y_{k-1} + \dd y_k)
\lesssim_{\gamma} 1 / A(y_{k-1})$, which multiplied by 
$b_{k+1}(y_{k-1})$ gives $b_{k}(y_{k-1})$. Then
the contribution of $|y_{1}|\leq \delta _{0}x$ and $|y_k| > \overline{\gamma} |y_{k-1}|$
to \eqref{eq:tirdte} is bounded by $%
J_{k-1}(\delta _{0};x)$, see \eqref{K}, which is a.n.\ by Proposition~\ref{G}.

\item \emph{Second term.} To deal with $\{|y_{1}| > \delta_0 x\}$,
we fix $\nu \in (0,1)$ sufficiently close to $1$,
so that $(k+1+\nu)\alpha - 1 > 0$ (we recall that $\alpha > \frac{1}{k+2}$),
which ensures that $b_{k+1+\nu}(\cdot)$ is asymptotically increasing. Then we can bound
\begin{align*}
\frac{b_{k+2}(\delta x\wedge \tilde{y}_{k})}{A(y_{k})} &\lesssim \frac{%
b_{k+1+\nu }(\delta x)A(y_{k})^{1-\nu }}{A(y_{k})} 
= \frac{b_{k+1+\nu }(\delta x)}{A(y_{k})^{\nu }}=b_{k+1+\nu }(\delta
x) \, b_{k}(y_{k}) \, f_{3}(y_{k}) \,,
\end{align*}%
where we set $f_{3}(y) := 1/b_{k+\nu }(y)$.
Note that $\alpha(k+\nu) < 1$ (by $\alpha \le \frac{1}{k+1}$),
hence $f_{3} \in RV(\beta)$
with $\beta = 1-\alpha(k+\nu)$ satisfies $0 < \beta < 1 - \alpha k$, i.e.\
the assumption of point~\eqref{it:1}
in Proposition~\ref{G} with $\ell =k$. 
The contribution of $\{|y_{1}| > \delta_0 x\}$ is then
\begin{equation*}
\begin{split}
	& \lesssim b_{k+1+\nu }(\delta x)\int_{y_{1}>-\overline{\gamma }%
	x,|y_{1}|> \delta_0 x}F(x+\mathrm{d}y_{1}) \, h_{k}(y_{1},f_{3}) 
	\lesssim b_{k+1+\nu }(\delta x) \, \frac{f_3(x)}{x \vee 1} \\
	& = \frac{b_{k+1+\nu }(\delta x)}{A(x)^{k+\nu}}
	\underset{x\to\infty}{\sim} \delta^{(k+1+\nu)\alpha-1} \, b_1(x)\,,
\end{split}
\end{equation*}%
which is a.n.\ and completes the proof.
\end{itemize}

\smallskip

\subsection{Technical proofs}
\label{sec:technical}

In this subsection we are going to prove Propositions~\ref{G} and~\ref{H}.
We first need two preliminary results,
stated in the next Propositions~\ref{th:boundcru} and~\ref{F}.

\smallskip

First an elementary
observation. Recall that $g_r(y)$ is defined in \eqref{eq:gak}.
We claim that
\begin{equation}\label{eq:firstobserve}
	\forall r \ge 2 \,, \ \forall y \in \R : \qquad
	g_r(y) \le A(y) \, \int_{|z| \le \overline{\gamma} |y|}
	F(-y + \dd z) \, g_{r-1}(z) \,.
\end{equation}
The case $r=2$ follows immediately from \eqref{eq:gak}-\eqref{eq:Omegak}
and \eqref{eq:b} (recall that $A$ is increasing).
Similarly, for $r \ge 3$, we simply observe that
$|y_r| \le |y|$ for $(y_2,\ldots, y_r) \in \Omega_r(y)$, hence
\begin{equation} \label{Z16a}
\begin{split}
g_{r}(y) &= \int_{|y_{2}|\leq \overline{\gamma }|y|} F(-y+\mathrm{d} y_{2})
\int_{\Omega_{r-1} (y_{2})}P_{y_{2}}(\mathrm{d} y_{3},\cdots , \mathrm{d}
y_{r}) \, b_{r+1}(y_{r})  \\
&\leq A(|y|) \int_{|y_{2}|\leq \overline{\gamma }|y|} F(-y+\mathrm{d} y_{2}) 
\int_{\Omega_{r-1} (y_{2})}P_{y_{2}}(\mathrm{d} y_{3},\cdots , \mathrm{d}
y_{r})\, b_{r}(y_{r}) \\
&= A(|y|) \int_{|y_{2}|\leq \overline{\gamma }|y|} F(-y+\mathrm{d} y_{2}) \, 
g_{r-1}(y_{2}) \,.
\end{split}%
\end{equation}%

We are ready for our first preliminary result.
If $I_{r}(\delta ;x)$ is a.n., then for $\delta > 0$ small we
have $I_{r}(\delta ;x) \lesssim b_1(x)$ for all $x\ge 0$
(recall Definition~\ref{def:an}).
We now show that the same bound holds when the
integral in \eqref{eq:Ikmod} is enlarged to $\{y_1 > -\kappa x\}$, for any
fixed $\kappa < 1$.

\begin{proposition}
\label{th:boundcru} Fix $r \in\N$ and $\alpha \in (0,\frac{1}{2})$.
Assume that $I_{j}(\delta ;x)$ is a.n.\ for $j=1,\ldots, r$. Then
for any $0 < \kappa <1$ 
\begin{equation}  \label{eq:crubo}
\int_{y> -\kappa x}F(x+\mathrm{d} y) \, g_{r}(y)
\,\lesssim_{\kappa,\gamma,r}\, b_{1}(x) \qquad \forall x \ge 0\,.
\end{equation}
\end{proposition}

\begin{proof}
The case $r=1$ is easy: since $b_{2}\in RV(2\alpha -1)$ and $%
2\alpha - 1 < 0$, for any fixed $\delta_0 > 0$
\begin{equation}
\int_{|y| > \delta_0 x, \, y>-\kappa x}F(x+\mathrm{d} y) \, b_{2}(y) \le %
\bigg(\sup_{|y| > \delta_0 x}b_{2}(y) \bigg) \, \overline F((1-\kappa)x)
\lesssim_{\delta_0, \kappa} \frac{b_{2}(x)}{A(x)} = b_{1}(x) \,.
\end{equation}%
On the other hand, the contribution to the integral of $|y| \le \delta_0 x$
gives $I_1(\delta_0; x)$ which is $\lesssim b_1(x)$ for $\delta_0 > 0$ small
enough, as we already observed, because $I_1(\delta; x)$ is a.n..

\smallskip

Next we fix $r \ge 2$. By induction, we can assume that \eqref{eq:crubo} holds with
$r$ replaced by $1, 2, \ldots, r-1$ and
our goal is to prove it for $r$.

Assume first that $y \le 0$, say $y = -t$ with $t \ge 0$.
By \eqref{eq:firstobserve} and
the inductive hypothesis \eqref{eq:crubo} for $r-1$ (since $t \ge 0$), we get
\begin{equation*}
	g_r(-t) \le A(t) \int_{|z| \le \overline{\gamma} t}
	F(t + \dd z) \, g_{r-1}(z) \lesssim_{\gamma,r} A(t) \, b_1(t)
	= b_2(t) = g_1(t) \,.
\end{equation*}
When we plug this bound into \eqref{eq:crubo} restricted to $y \le 0$, we get
\begin{equation*}
\int_{0 \le t < \kappa x} F(x - \mathrm{d} t) \, g_{r}(-t) \lesssim_{\gamma,r}
\int_{0 \le t < \kappa x} F(x - \mathrm{d} t) \, g_1(t) \lesssim_{\kappa} b_1(x) \,,
\end{equation*}
where the last inequality holds by the inductive hypothesis \eqref{eq:crubo} for $r=1$.

It remains to look at the contribution of $y > 0$ in \eqref{eq:crubo}. 
By \eqref{eq:Ikmod}, the contribution of $\{0 < y \le \delta_1 x\}$ to \eqref{eq:crubo}
is bounded by $I_{r}(\delta_1, x)$ which is a.n.\ by assumption, hence it
is $\lesssim b_1(x)$ provided $\delta_1 > 0$ is small enough. It remains to
focus on $\{y > \delta_1 x\}$.

We need a simple observation; let $I=(a_{1},a_{2})$ where $0\leq
a_{1}<a_{2}\leq \infty $ and, for $\gamma \in (0,1)$, 
put $I^{\prime }=(\gamma a_{1},(2-\gamma
)a_{2}).$ Then for all non-negative functions $f,g: \R \to [0,\infty)$%
\begin{equation}  \label{z}
\begin{split}
&\int_{y\in xI}F(x+\mathrm{d} y)  \, f(y) \int_{|z|\leq \overline{\gamma }y}F(-y+%
\mathrm{d} z)\, g(z) \\
&=\int_{y\in xI}F(x+\mathrm{d} y)  \, f(y)
\int_{\gamma y \le w \le (2-\gamma )y} F(-%
\mathrm{d} w) \, g(y-w) \\
&\le \int_{w\in xI^{\prime}}F(-\mathrm{d} w)\int_{(2-\gamma )^{-1}w \le y
\le \gamma ^{-1}w} F(x+\mathrm{d} y) \, f(y) \, g(y-w) \\
&=\int_{w\in xI^{\prime }} F(-\mathrm{d} w) \int_{-\gamma_{1}w \le v \le
\gamma _{2}w} F(x+w+\mathrm{d} v)\, f(w+v) \, g(v),
\end{split}%
\end{equation}%
where $\gamma _{1}=1-(2-\gamma )^{-1}$ and $\gamma _{2}=\gamma ^{-1}-1.$

Applying \eqref{eq:firstobserve}
together with \eqref{z},
the contribution of $\{y > \delta_1 x\}$ to %
\eqref{eq:crubo} is
\begin{equation*}
\begin{split}
\int_{\delta _{1} x}^{\infty } F(x + \mathrm{d} y) \, g_{r}(y) & \leq
\int_{\delta _{1}x}^{\infty }F(x+\mathrm{d} y) \, A(y)
\int_{-\overline{\gamma }y}^{%
\overline{\gamma }y}F(-y+\mathrm{d} z) \, g_{r-1}(z) \\
&=\int_{\gamma\delta _{1}x}^{\infty } F(-\mathrm{d}
w)\int_{-\gamma _{1}w}^{\gamma _{2}w} F(x+w + \mathrm{d} v)
\, A(v+w) \, g_{r-1}(v) \\
&\lesssim_{\gamma} \int_{\gamma\delta _{1}x}^{\infty }F(-\mathrm{%
d} w) \, A(w)\int_{-\gamma _{1}(x+w)}^{\infty} 
F(x+w + \mathrm{d} v) \, g_{r-1}(v) \,.
\end{split}%
\end{equation*}
Applying \eqref{eq:crubo} for $r-1$, and the fact that $b_2(\cdot)$ is
asymptotically decreasing, we obtain
\begin{equation*}
\begin{split}
\int_{\delta _{1} x}^{\infty } F(x + \mathrm{d} y) \, g_{r}(y) & \lesssim
\int_{\gamma\delta _{1}x}^{\infty} F(-\mathrm{d} w)\, A(w) \,
b_1(x+w) \le \int_{\gamma\delta _{1}x}^{\infty }F(-\mathrm{d} w)\,
b_2(x+w) \\
& \lesssim b_{2}(x)\int_{\gamma\delta _{1}x}^{\infty} F(-\mathrm{%
d} w)\lesssim_{\delta_1,\gamma} b_1(x) \,.\qedhere
\end{split}%
\end{equation*}
\end{proof}

We now introduce a generalization $g_k(y,f)$ of $g_k(y)$
(in the same way as $h_k(y,f)$ generalizes $h_k(y)$, see
\eqref{L}-\eqref{Lf}). For any non-negative, even function $f:%
\mathbb{R}\rightarrow [0,\infty)$ we denote by $g_{k}(y_{1},f)$ what we get
by replacing $b_{k+1}(y_{k})$ by $b_{k}(y_{k})f(y_{k})$ in \eqref{eq:gak},
that is
\begin{equation}
g_{k}(y_{1},f):=%
\begin{cases}
b_{1}(y_{1})\,f(y_{1}) & \text{ if }k=1 \\
\rule{0pt}{1.9em}\displaystyle\int_{\Omega _{k}(y_{1})}P_{y_{1}}(\mathrm{d}%
y_{2},\ldots ,\mathrm{d}y_{k})\,b_{k}(y_{k})\,f(y_{k}) & \text{ if }k\geq 2%
\end{cases}%
\,.  \label{eq:gkf}
\end{equation}%
In particular, $g_{k}(y)$ is $g_{k}(y,A)$. 

We are going to assume that $f(|\cdot|) \in RV(\beta)$ for some $\beta > 0$,
so $f$ is asymptotically increasing
and $f(w) \lesssim f(y)$ for $|w| \le |y|$.
Then, in analogy with \eqref{eq:firstobserve}, we claim that
\begin{equation}\label{eq:firstobserve2}
	\forall r \ge 2 \,, \ \forall y \in \R : \qquad
	g_r(y,f) \le f(y) \, \int_{|z| \le \overline{\gamma} |y|}
	F(-y + \dd z) \, g_{r-1}(z) \,.
\end{equation}
The case $r = 2$ follows immediately by \eqref{eq:gkf}, while
for $r \ge 3$ we can argue as in \eqref{Z16a},
replacing $b_{r+1}(y_r)$ by $b_r(y_r, f)$ and bounding
$f(y_r) \lesssim f(y)$, since $|y_r| \le |y|$ on $\Omega_r(y)$.

We now state our second preliminary result, which is in the
same spirit as Proposition~\ref{th:boundcru}.

\begin{proposition}
\label{F} Fix $r \in\N$ and $\alpha \in (0,\frac{1}{2})$.
If $r \ge 2$, assume that $I_{j}(\delta ;x)$ is a.n.\ 
for $j=1,\ldots, r-1$.
Fix any $f \in RV(\beta)$
with $0 < \beta < 1-\alpha$. Then for all $0 < \delta_0 < \kappa <1$
\begin{equation}
\int_{|y|>\delta _{0}x,\, y>-\kappa x} F(x+\mathrm{d} y) \, g_{r}(y,f)
\lesssim_{\delta_0, \kappa,\gamma,r} \frac{f(x)}{x \vee 1} \qquad \forall x
\ge 0 \,.  \label{z2}
\end{equation}
\end{proposition}

\begin{proof}
Since $g_{1}(\cdot,f) = b_{1}(\cdot )f(\cdot ) \in RV(\alpha +
\beta - 1)$ is asymptotically decreasing, we have
\begin{align*}
\int_{|y|>\delta _{0}x, \, y>-\kappa x} F(x+\mathrm{d} y) \, g_{1}(y,f) &
\lesssim_{\delta_0} b_{1}(x)f(x) \int_{y>-\kappa x} F(x+\mathrm{d} y) \\
& \lesssim_{\kappa} \frac{b_{1}(x)f(x)}{A(x)} =\frac{f(x)}{x \vee 1} \, ,
\end{align*}%
which proves \eqref{z2} if $r=1$. Henceforth we assume that $r \ge 2$
and proceed by induction.
Note that we can apply Proposition~\ref{th:boundcru} with $r$ replaced by $r-1$
(since here we assume that $I_{j}(\delta ;x)$ is a.n.\ 
for $j=1,\ldots, r-1$).

Assume first that $y \le 0$, say $y = -t$ with $t \ge 0$.
Then by \eqref{eq:firstobserve2} we can bound
\begin{equation}  \label{Z16}
\begin{split}
	g_{r}(-t,f) \le f(t) \int_{|z|\leq \overline{\gamma }t} F(t+\mathrm{d} z) \,
	g_{r-1}(z)\lesssim_{\gamma} f(t) \, b_1(t) \,,
\end{split}%
\end{equation}%
where for the last inequality we apply Proposition~\ref{th:boundcru} for $r-1$ 
(since $t \ge 0$).
Since $f(\cdot) b_1(\cdot)$ is asymptotically decreasing,
the contribution of $y \le 0$ to \eqref{z2} is then estimated by
\begin{equation*}
\int_{\delta _{0}x < t < \kappa x} F(x-\mathrm{d} t)\,
g_{r}(-t,f)\lesssim_{\delta_0} f(x) \, b_1(x) \int_{\delta _{0}x < t
< \kappa x} F(x-\mathrm{d} t) \lesssim_{\kappa} \frac{f(x) \, b_1(x)}{A(x)}
= \frac{f(x)}{x \vee 1} \,.
\end{equation*}

It remains to control the contribution to \eqref{z2} of $y>0$. 
By \eqref{eq:firstobserve2} and \eqref{z}
\begin{equation}\label{eq:enpr1}
\begin{split}
\int_{\delta _{0}x}^{\infty }F(x+\mathrm{d}y)\,g_{r }(y,f)& \leq
\int_{\delta _{0}x}^{\infty }F(x+\mathrm{d}y) \, f(y)\,
\int_{-\overline{\gamma }y}^{%
\overline{\gamma }y}F(-y+\mathrm{d}z)\,g_{r -1}(z) \\
& \lesssim \int_{\gamma\delta _{0}x}^{\infty }F(-\mathrm{d}%
w)\int_{-\gamma _{1}w}^{\gamma _{2}w}F(x+w + \mathrm{d}v)\,f(w+v)\,g_{r -1}(v) \\
& \lesssim _{\gamma }\int_{\gamma\delta _{0}x}^{\infty }F(-%
\mathrm{d}w)\,f(w)\int_{-\gamma _{1}(x+w)}^{\infty }F(x+w +%
\mathrm{d}v)\,g_{r -1}(v)\,.
\end{split}%
\end{equation}%
Applying again Proposition~\ref{th:boundcru} for $r -1$ we get,
since $f(\cdot) b_1(\cdot)$ is asymptotically decreasing,
\begin{equation*}
\begin{split}
\int_{\delta _{0}x}^{\infty }F(x+\mathrm{d}y)\,g_{\ell }(y,f)& \lesssim_{\kappa,\gamma,r}
\int_{\gamma\delta _{0}x}^{\infty }F(-\mathrm{d}%
w)\,f(x+w)\,b_{1}(x+w) \\
& \lesssim f(x)\,b_{1}(x)\int_{\gamma\delta _{0}x}^{\infty }F(-%
\mathrm{d}w)\lesssim _{\gamma ,\delta _{0}}\frac{f(x)}{x\vee 1}\,.\qedhere
\end{split}%
\end{equation*}
\end{proof}

\smallskip

We are finally ready to prove Propositions~\ref{G} and~\ref{H}.

\begin{proof}[Proof of Proposition~\ref{G}]

We write $r$ in place of $\ell$.
We assume that $I_j(\delta; x)$ is a.n.\ for $j=1,\ldots, r-1$
(if $r \ge 2$).
Moreover, for point \eqref{it:2} we also assume that $I_r(\delta; x)$ is a.n..

Recall the definitions of $h_k(y,f)$, $g_k(y,f)$, see  \eqref{Lf}, \eqref{eq:gkf}.
We claim that 
\begin{equation}\label{eq:claim}
	\forall \text{ even } f \in RV(\beta) \text{ with } 0 < \beta < 1- r\alpha: \qquad
	h_r(y,f) \lesssim_{\gamma,r} \sum_{j=1}^r g_j(y,f) \,.
\end{equation}
Then relation \eqref{m} follows immediately by Proposition~\ref{F}.
This proves point~\eqref{it:1}.

For point~\eqref{it:2},
we note that for $\alpha < \frac{1}{r+1}$
we can plug $f=A$ in \eqref{eq:claim},
because $\beta =\alpha $ satisfies $\beta <1-r\alpha $.
This gives
$h_r(y) \lesssim_{\gamma} \sum_{j=1}^r g_j(y)$, which plugged into \eqref{K}
shows that $J_r(\delta; x) \lesssim_{\gamma} \sum_{j=1}^r I_r(\delta; x)$.
Since in point~\eqref{it:2} we assume that $I_j(\delta;x)$ for $j=1,\ldots, r$
(including $j=r$), relation \eqref{n} follows and completes the proof.

\smallskip

It remains to prove \eqref{eq:claim}.
This holds for $r=1$, since $h_1(y,f) = g_1(y,f)$.
Henceforth we fix $r \ge 2$ and we proceed by induction.

Let us first show that
\begin{equation}\label{eq:ccllaa}
\begin{split}
	& \forall r' = 1,\ldots, r-1 \,, \quad \
	\forall \text{ even } f' \in RV(\beta') \text{ with } 0 < \beta' < 1 - r'\alpha \,, \\
	& \forall y < 0: \qquad \int_{z \ge -\overline{\gamma} |y|} F(-y + \dd z) \,
	h_{r'}(z, f') \lesssim_{\gamma}
	\sum_{i=1}^{r'+1} g_i(y, \tfrac{f'}{A}) \,,
\end{split}
\end{equation}
where we stress that $y < 0$.
By the inductive assumption,
we can apply \eqref{eq:claim} with $r$ replaced by $r'$ (since $r' \le r-1$)
and $f$ replaced by $f'$, hence 
\begin{equation*}
	\int_{z \ge -\overline{\gamma} |y|} F(-y + \dd z) \,
	h_{r'}(z, f') \lesssim_{\gamma}
	\sum_{j=1}^{r'}
	\int_{z \ge -\overline{\gamma} |y|} F(-y + \dd z) \,
	g_{j}(z, f') \,.
\end{equation*}
We now split the domain of integration in the two
subsets $[-\overline{\gamma}|y|, \overline{\gamma}|y|]$
and $(\overline{\gamma}|y|, \infty)$. The first subset gives
$\int_{|z| \le \overline{\gamma} |y|} F(-y + \dd z)
\, g_{j}(z, f') = g_{j+1}(y, \tfrac{f'}{A})$, by \eqref{eq:gkf}.
For the second subset we can apply Proposition~\ref{F}
(since $-y \ge 0$), getting
\begin{equation*}
	\int_{z > \overline{\gamma}|y|} F(-y+\dd z)
	\, g_{j}(z, f') \lesssim_{\gamma} \frac{f'(|y|)}{|y| \vee 1}
	= b_1(y) \, \frac{f'(y)}{A(y)} = g_1(y,\tfrac{f'}{A}) \,,
\end{equation*}
where we recall that $A(\cdot)$ and $f'(\cdot)$ are even functions.
This completes the proof of \eqref{eq:ccllaa}.

\smallskip

We are ready to prove \eqref{eq:claim}.
Let us first consider the case $y < 0$. By \eqref{Lf} we can write
\begin{equation*}
	h_r(y,f) = \int_{y_2 \ge -\overline{\gamma}|y|} F(-y+\dd y_2)
	\, h_{r-1}(y_2, A f) \,.
\end{equation*}
We can now apply \eqref{eq:ccllaa} with $r' := r-1$
and $f' := Af$ (because $f' \in RV(\beta')$ with $\beta' = \alpha + \beta$ which
satisfies $0 < \beta' < 1-r'\alpha$). This proves \eqref{eq:claim} when $y < 0$.

\smallskip

Next we consider the case $y \ge 0$.
If we restrict the domain of integration $\Theta_r(y)$ in \eqref{Lf}
to $y_2 \ge 0, y_3 \ge 0, \ldots, y_r \ge 0$, then the domain becomes
$\{0 \le y_j \le \overline{\gamma} y_{j-1} \text{ for } 2 \le j \le r\}$
which is included in $\Omega_r(y)$, see \eqref{eq:Omegak}.
The corresponding contribution to $h_r(y, f)$ is then bounded from above by $g_r(y, f)$,
see \eqref{eq:gkf}. This proves \eqref{eq:claim} 
when $y_2 \ge 0, y_3 \ge 0, \ldots, y_r \ge 0$.

\smallskip

It remains to estimate $h_r(y, f)$ for $y \ge 0$, when some of the coordinates
$y_2, y_3, \ldots, y_r$
in the integral in \eqref{Lf} are negative. Let us define
$H := \min \{j\in \{2, \ldots, r\}:\, y_{j} < 0\}$. 

In the extreme case $H=r$, 
the corresponding contribution to $h_{r}(y,f)$ is, for $r \ge 3$,
\begin{equation} \label{eq:duli}
\begin{split}
&\int_{y_{2}=0}^{\overline{\gamma }y}\cdots \int_{y_{r-1}=0}^{\overline{%
\gamma }y_{r-2}}\int_{y_{r}=-\overline{\gamma }y_{r-1}}^{0}P_{y}(\dd y_{2},%
\cdots , \dd y_{r}) \, b_{r}(y_{r})f(y_{r}) \\
+& \int_{y_{2}=0}^{\overline{\gamma }y}\cdots \int_{y_{r-1}=0}^{\overline{%
\gamma }y_{r-2}}\int_{y_{r}=-\infty }^{-\overline{\gamma }%
y_{r-1}} P_{y}(\dd y_{2},\cdots ,\dd y_{r}) \, b_{r}(y_{r})f(y_{r}) \,.
\end{split}
\end{equation}%
If $r = 2$, one should ignore the first integrals, that is we have
\begin{equation} \label{eq:duli2}
	\int_{y_2 = - \overline{\gamma}y}^0 F(-y+\dd y_2) \, b_2(y_2) \, f(y) \,+\,
	\int_{y_2 = - \infty}^{\overline{\gamma}y} F(-y+\dd y_2) \, b_2(y_2) \, f(y) \,.
\end{equation}
The first integral in \eqref{eq:duli}-\eqref{eq:duli2}
is bounded by $g_r(y,f)$, because the domain of integration
for $(y_2, \ldots, y_r)$ is included in $\Omega_r(y)$ (recall \eqref{eq:gkf}
and \eqref{eq:Omegak}).
For the second integral, we note that $b_{r}(\cdot )f(\cdot ) \in 
RV(r\alpha - 1 + \beta)$ 
is asymptotically
decreasing, since $r\alpha - 1 + \beta < 0$ by assumption, hence we can bound
$b_{r}(y_{r})f(y_{r}) \lesssim b_{r}(y_{r-1})f(y_{r-1})$.
Since $P_{y}(\dd y_{2},\cdots ,\dd y_{r}) = P_{y}(\dd y_{2},\cdots ,\dd y_{r-1})
F(-y_{r-1}+\dd y_{r})$, when we integrate over $y_{r} \in (-\infty,
-\overline{\gamma} y_{r-1}]$ we get a factor $\lesssim_{\gamma} 1/A(y_{r-1})$.
Overall, for $r \ge 3$ we can bound \eqref{eq:duli} by
\begin{equation*}
\begin{split}
	\leq & \, g_{r}(y,f)+\int_{y_{2}=0}^{\overline{\gamma }y}\cdots
	\int_{y_{r-1}=0}^{\overline{\gamma }y_{r-2}}P_{y} (\dd y_{2},\cdots
	,\dd y_{r-1}) \, \frac{1}{A(y_{r-1})} \, b_{r}(y_{r-1})f(y_{r-1}) \\
	\leq & \, g_{r}(y,f)+g_{r-1}(y,f) \,,
\end{split}
\end{equation*}
and the same bound holds also for $r=2$.
This proves \eqref{eq:claim} when $H=r$.

Finally, if $H=j \in \{2, \ldots, r-1\}$, 
the contribution to $h_{r}(y,f)$ is (recall again \eqref{Lf})
\begin{align}
&\int_{y_{2}=0}^{\overline{\gamma }y}\cdots \int_{y_{j-1}=0}^{\overline{%
\gamma }y_{j-2}}\int_{y_{j}=-\infty }^{0}P_{y}(\dd y_{2},\cdots
,\dd y_{j})\int_{y_{j+1}>-\overline{\gamma }%
|y_{j}|} F(-y_{j}+ \dd y_{j+1})\, h_{r-j}(y_{j+1},A^{j}f) \nonumber \\
&\lesssim_{\gamma} \sum_{i=1}^{r-j+1} 
\int_{y_{2}=0}^{\overline{\gamma }y}\cdots \int_{y_{j-1}=0}^{%
\overline{\gamma }y_{j-2}}\int_{y_{j}=-\infty }^{0}
P_{y}(\dd y_{2},\cdots, \dd y_{j}) \, 
g_{i}(y_{j}, A^{j-1}f) \,, \label{eq:deco}
\end{align}%
where we have applied \eqref{eq:ccllaa} with $r' = r-j$
and $f' = A^j f$ (note that $f' \in RV(\beta')$ with
$\beta' = j\alpha + \beta$, which satisfies $0 < \beta' < 1 - r' \alpha$).
We split the integral over $y_j$ in the two subsets
$[-\overline{\gamma} y_{j-1}, 0]$ and $(-\infty,-\overline{\gamma} y_{j-1}]$.
\begin{itemize}
\item On the first subset, we can enlarge the domain of integration
to $(y_2, \ldots, y_j) \in \Omega_j(y)$, see \eqref{eq:Omegak},
hence the corresponding contribution to \eqref{eq:deco} is
\begin{equation*}
\begin{split}
	& \sum_{i=1}^{r-j+1} \int_{(y_2, \ldots, y_j) \in \Omega_j(y)}
	P_{y}(\dd y_{2},\cdots, \dd y_{j}) \, 
	g_{i}(y_{j}, A^{j-1}f) \\
	& = \sum_{i=1}^{r-j+1} \int_{(y_2, \ldots, y_{j+i-1}) \in \Omega_{j+i-1}(y)}
	P_{y}(\dd y_{2},\cdots, \dd y_{j+i-1}) \, 
	b_{j+i-1}(y_{j+i-1}) \, f(y_{j+i-1}) \\
	& = \sum_{i=1}^{r-j+1} g_{j+i-1}(y, f)  \,,
\end{split}
\end{equation*}
by \eqref{eq:gkf}. This proves \eqref{eq:ccllaa} for the first subset.

\item On the second subset, we first consider a fixed $i\geq 2$:
renaming $y_j = -z$, we can write
\begin{align}
	&\int_{y_{j}=-\infty }^{-\overline{\gamma }y_{j-1}}F(-y_{j-1}+ \dd
	y_{j})\,g_{i}(y_{j},A^{j-1}f) \\
	&=\int_{z=\overline{\gamma }y_{j-1}}^{\infty }F(-y_{j-1}- \dd
	z)\int_{\Omega _{i}(z)}\,P_{-z}(\dd y_{j+1,}\cdots
	\dd y_{j+i})b_{i}(y_{j+i},A^{j-1}f) \nonumber \\
	&=\int_{z=\overline{\gamma }y_{j-1}}^{\infty }F(-y_{j-1}-\dd
	z)\int_{|y_{j+1}|\leq \overline{\gamma }z}F(z+\dd y_{j+1}) \nonumber \\
	& \qquad\qquad\qquad\qquad\qquad \int_{\Omega
	_{i-1}(y_{j+1})}\,P_{y_{j+1,}}(\dd y_{j+2,}\cdots
	\dd y_{j+i})b_{i+j-1}(y_{j+i})f(y_{j+i}) \,. \label{eq:wecana}
\end{align}
We next write $b_{i+j-1}(\cdot)f(\cdot) = \{A^{j-1}(\cdot) f(\cdot)\} b_i(\cdot)$ and then bound
$A^{j-1}(y_{j+i}) f(y_{j+i}) \lesssim A^{j-1}(z) f(z)$,
because $A^{j-1}(\cdot) f(\cdot)$ is asymptotically increasing, to get
\begin{equation*}
\begin{split}
	&\lesssim_{\gamma } \int_{z=\overline{\gamma }y_{j-1}}^{\infty }F(-y_{j-1}-	 \dd z)
	A(z)^{j-1}f(z)\int_{|y_{j+1}|\leq \overline{\gamma }%
	z}F(z+\dd y_{j+1}) \, g_{i-1}(y_{j+1}) \\
	&\lesssim_{\gamma } \int_{z=\overline{\gamma }y_{j-1}}^{\infty }F(-y_{j-1}-%
	\dd z) \, A(z)^{j-1} \, f(z) \, b_{1}(z) \lesssim_{\gamma}b_{j-1}(y_{j-1})f(y_{j-1}) \,,
\end{split}
\end{equation*}%
where we used
\eqref{eq:crubo} and the fact that $b_{1}(z)f(z)A^{j-1}(z)$
is regularly varying with index $j\alpha +\beta -1 < 0$ and so is asymptotically decreasing.
This shows that
\begin{equation*}
	\int_{y_{j}=-\infty }^{-\overline{\gamma }y_{j-1}}F(-y_{j-1}+ \dd
	y_{j})\,g_{i}(y_{j},A^{j-1}f) \lesssim_{\gamma}b_{j-1}(y_{j-1})f(y_{j-1}) \,,
\end{equation*}
and the same bound holds also for $i=1$
(since $g_{1}(z,A^{j-1}f) = b_{1}(z) f(z) A^{j-1}(z)$,
we can directly apply \eqref{eq:wecana}).
Thus the contribution of $y_j \le -\overline{\gamma} y_{j-1}$ to \eqref{eq:deco} is
\begin{equation*}
\begin{split}
	& \lesssim_{\gamma} \sum_{i=1}^{r-j+1}\int_{y_{2}=0}^{\overline{\gamma }y}\cdots
	\int_{y_{j-1}=0}^{\overline{\gamma }y_{j-2}}P_{y}(\mathrm{d}y_{2},\cdots ,%
	\mathrm{d}y_{j-1})\,b_{j-1}(y_{j-1})f(y_{j-1}) \\
	&\leq \sum_{i=1}^{r-j+1}\int_{\Omega _{j-1}(y)}P_{y}(\mathrm{d}y_{2},\cdots
	,\mathrm{d}y_{j-1})\,b_{j-1}(y_{j-1})f(y_{j-1})=(r-j+1)g_{j-1}(y).
\end{split}
\end{equation*}%
This completes the proof of \eqref{eq:ccllaa}.
\end{itemize}
\end{proof}

\begin{proof}[Proof of Proposition~\ref{H}]
We write $r$ in place of $\ell$.
We assume that $\tilde I_j(\delta; x)$
and $I_j(\delta;x)$ are a.n.\ for $j=1,\ldots, r$,
with $r \ge 2$ and $\alpha \le \frac{1}{r+1}$,
and we need to show that $\tilde J_r(\delta; x)$ is a.n..

We first give a basic estimate:
from \eqref{eq:btildealt}, \eqref{eq:Kar2} and \eqref{eq:b}, for any $\lambda \in
(0,r)$ we have 
\begin{equation}
\tilde{b}_{r+1}(y,z)\leq A(y)^{\lambda} \, \tilde{b}_{r+1-\lambda }(y,z) \le
A(y)^{\lambda} \sum_{m=A(z)}^{\infty} \frac{m^{(r-\lambda)}}{a_m} \lesssim
A(y)^{\lambda} \, b_{r+1-\lambda} (z).  \label{bo}
\end{equation}

Let us prove that $\tilde J_r(\delta; x)$ is a.n..
For $\alpha < \frac{1}{r+1}$ we can simply apply Proposition~\ref{G} with $\ell = r$,
because $\tilde J_r(\delta; x) \lesssim J_r(\delta; x)$.
Indeed, by \eqref{eq:btildealt},
\begin{equation*}
	\tilde b_{r+1}(y_{r-1},y_r) \lesssim
	\sum_{n \ge A(|y_r|)}\frac{n^{r}}{a_{n}}
	\lesssim \frac{A(|y_r|)^{r+1}}{|y_r| \vee 1} = b_{r+1}(y_r) \,,
\end{equation*}
because $n^{r}/a_{n}$ is regularly varying with index
$r - 1/\alpha < -1$.

Henceforth we fix $\alpha = \frac{1}{r+1}$.
For $r=2$ there is nothing to prove, since $\tilde{J}%
_{2}(\delta ;x) = \tilde{I}_{2}(\delta ;x)$. 

We now fix $r\geq 3$.
If we consider the contribution to the integrals in \eqref{eq:tildeJ}-\eqref{h} of 
$y_{1} \geq 0, y_{2}\geq
0, \ldots ,y_{r-1}\geq 0$, the domain of integration, see \eqref{y}, reduces to 
\begin{equation*}
	\{0 \leq y_{1} \leq \delta x\} \cap
	\{0\leq y_{i}\leq \overline{\gamma }y_{i-1}\text{ for }2\leq i\leq r-1\}
	\cap \{|y_{r}| \leq \overline{\gamma }y_{r-1}\} \,.
\end{equation*}
This contribution is bounded by $\tilde{I%
}_{r}(\delta ;x)$, see \eqref{eq:tildeIk}, which is a.n.\ by assumption.

Next we consider the contribution to \eqref{eq:tildeJ}-\eqref{h} coming from $y_{1},\ldots
,y_{r-1}$ such that $y_{i}<0$ for some $1\leq i\leq r-1$. Let us define $%
\tilde{H}=\max \{j\in \{1,\ldots ,r-1\}:\ y_{j}<0\}$. If $\tilde{H}=r-1$,
the bound (\ref{bo}) with $\lambda =r-1$ and the fact that $%
y_{r-1}<0$ show that%
\begin{align*}
& \int_{|y_{r}|\leq \overline{\gamma }|y_{r-1}|}F(-y_{r-1}+\mathrm{d}y_{r})\,%
\tilde{b}_{r+1}(y_{r-1,}y_{r}) \\
& \lesssim A(y_{r-1})^{r-1}\int_{|y_{r}|\leq \overline{\gamma }%
|y_{r-1}|}F(-y_{r-1}+\mathrm{d}y_{r})\,b_{2}(y_{r})\lesssim
A(y_{r-1})^{r-1}b_{1}(y_{r-1})=b_{r}(y_{r-1}),
\end{align*}%
where the second inequality comes from Proposition~\ref{th:boundcru}. Plugging this
bound into \eqref{h}, we see that the contribution to $\tilde{h}_{r}(y)$
is at most $h_{r-1}(y)$ (recall \eqref{L}), hence the
contribution to $\tilde{J}_{r}(\delta ;x)$ is at most 
$J_{r-1}(\delta ;x)$ (recall \eqref{K}), which is a.n.\ by Proposition~\ref{G}
with $\ell=r-1$.

We finally consider the contribution of $\tilde{H}=r-j$ with $j\geq 2$. This
means that $y_{r-j}<0$, while $y_{r-j+1}\geq 0,\ldots ,y_{r-1}\geq 0$, and
the range of integration in \eqref{h} is a subset of%
\begin{equation*}
\Theta_{r-j}(y_1) \cap \{y_{r-j}<0\}\cap \{y_{r-j+1}\geq 0\}\cap 
\{|y_{r-j+1+\ell}|\leq \overline{\gamma }|y_{r-j+\ell}|,\ell =1,\cdots j-1\}.
\end{equation*}
We split this into the two subsets $\{0 \leq y_{r-j+1}\leq
\overline{\gamma }|y_{r-j}|\}$ and $\{y_{r-j+1} > \overline{\gamma }|y_{r-j}|\}$.

\smallskip

On the first subset $\{0 \leq y_{r-j+1}\leq
\overline{\gamma }|y_{r-j}|\}$, we
bound $\tilde{b}_{r+1}(y_{r-1},y_{r})\lesssim A(y_{r-1})^{r-j}b_{j+1}(y_{r})$,
by \eqref{bo} with $\lambda =r-j$, and then $A(y_{r-1}) \lesssim A(y_{r-j})$.
Recalling the definition \eqref{eq:gak} of $g_{j}(\cdot )$, we see that this 
part of the integral with respect to $y_{r-j+1},\cdots, y_{r}$ is
\begin{equation*}
\begin{split}
& \lesssim A(y_{r-j})^{r-j}\,\int\limits_{0 \le y_{r-j+1}\leq \overline{\gamma
}|y_{r-j}|}\!\!\!\!\!\!\!\!F(|y_{r-j}|+\mathrm{d}y_{r-j+1})%
\,g_{j}(y_{r-j+1}) \\
& \leq A(y_{r-j})^{r-j}\,\int\limits_{z\geq -\overline{\gamma }%
|y_{r-j}|}\!\!\!\!\!\!\!\!F(|y_{r-j}|+\mathrm{d}z)\,g_{j}(z) \\
& \lesssim A(y_{r-j})^{r-j}\,b_{1}(y_{r-j})=b_{r-j+1}(y_{r-j})\,,
\end{split}%
\end{equation*}%
where the last inequality follows by Proposition~\ref{th:boundcru}.
The contribution to \eqref{h} is
\begin{equation} \label{eq:arggu}
\lesssim \int_{\Theta _{r-j}(y_1)\cap \{y_{r-j}<0\}}P_{y}(\mathrm{d}%
y_{2},\ldots ,\mathrm{d}y_{r-j})\,b_{r-j+1}(y_{r-j})\leq h_{r-j}(y_1)\,,
\end{equation}%
hence the contribution to $\tilde{J}_{r}(\delta ;x)$ is 
$\lesssim J_{r-j}(\delta ;x)$, which is a.n.\ by Proposition~\ref{G}.

\smallskip

On the second subset $\{y_{r-j+1} > \overline{\gamma }|y_{r-j}|\}$, we
bound $\tilde{b}_{r+1}(y_{r-1},y_{r})\lesssim
A(y_{r-1})^{r-j+1}b_{j}(y_{r})$, by \eqref{bo} with $\lambda =r-j+1$,
and then $A(y_{r-1}) \lesssim A(y_{r-j+1})$, getting
\begin{align*}
&\lesssim \!\!\! \int\limits_{y_{r-j+1} > \overline{\gamma }|y_{r-j}|}
\!\!\!\!\!\! F(|y_{r-j}| + \dd y_{r-j+1}) \, A(y_{r-j+1})^{r-j+1} \, 
\!\!\!\!\!\!\!\!\! \int\limits_{|y_{r-j+2}| \le \overline{\gamma } y_{r-j+1}}
\!\!\!\!\!\! F(-y_{r-j+1} + \dd y_{r-j+2}) \, g_{j-1}(y_{r-j+2}) \\
& = \int\limits_{y > \overline{\gamma }|y_{r-j}|}
F(|y_{r-j}| + \dd y) \, A(y)^{r-j+1} \, 
 \int\limits_{|z| \le \overline{\gamma } y}
F(-y + \dd z) \, g_{j-1}(z) \,,
\end{align*}
where we have set $y = y_{r-j+1}$ and $z = y_{r-j+2}$ for short.
Applying \eqref{z}, where we recall that $\gamma _{1}=1-(2-\gamma )^{-1}$ and 
$\gamma _{2}=\gamma ^{-1}-1$, we get
\begin{align*}
	&\lesssim \int\limits_{w \ge \gamma\overline{\gamma }|y_{r-j}|} F(-\dd w)
	\int\limits_{-\gamma_1 w \le v \le \gamma_2 w} F(|y_{r-j}|+w+\dd v) 
	\, A(w+v)^{r-j+1} \, g_{j-1}(v) \\
	&\lesssim_{\gamma} \int\limits_{w \ge \gamma\overline{\gamma }|y_{r-j}|} F(-\dd w) \,
	A(w)^{r-j+1} \int\limits_{v \ge -\gamma_1 (|y_{r-j}|+w)} F(|y_{r-j}|+w+\dd v) 
	\, g_{j-1}(v) \\
	&\lesssim_{\gamma} \int\limits_{w \ge \gamma\overline{\gamma }|y_{r-j}|} F(-\dd w) \,
	A(w)^{r-j+1} \, b_1(|y_{r-j}|+w)  \,,
\end{align*}
by Proposition~\ref{th:boundcru} with $r=j-1$. Finally, this is easily bounded by
\begin{align*}
	& \int\limits_{w \ge \gamma\overline{\gamma }|y_{r-j}|} F(-\dd w) \,
	b_{r-j+2}(|y_{r-j}|+w) \lesssim F(-\gamma\overline{\gamma }|y_{r-j}|) \,
	b_{r-j+2}(|y_{r-j}|) \lesssim b_{r-j+1}(|y_{r-j}|) \,,
\end{align*}
because $b_{r-j+2}(\cdot) \in RV(\alpha(r-j+2)-1)$ is asymptotically
decreasing, since $j \ge 2$ and $\alpha = \frac{1}{r+1}$.
Arguing as in \eqref{eq:arggu},
we see that the contribution to $\tilde{J}_{r}(\delta ;x)$ is 
$\lesssim J_{r-j+1}(\delta ;x)$, which is a.n.\ by Proposition~\ref{G}.
This completes the proof.
\end{proof}

\smallskip

\section{Soft results}
\label{sec:soft}

In this section we prove 
Theorem~\ref{th:1/2},
Propositions~\ref{pr:main1}, \ref{th:integr}, \ref{th:suffrw}
and Theorem~\ref{th:Levy},
which are corollaries of our main results.

\subsection{Proof of Theorem~\ref{th:1/2}}

Assume that condition \eqref{eq:cond} holds. By \eqref{eq:b} we can write
\begin{equation*}
	\sup_{1 \le z \le x} b_2(z) = \sup_{1 \le z \le x} \frac{A(z)^2}{z}
	\lesssim \frac{A(x)^2}{x} \,.
\end{equation*}
For $0 \le z \le 1$ we can also write $b_2(z) \le A(1)^2 = b_2(1) 
\lesssim \frac{A(x)^2}{x}$ hence by \eqref{eq:I1}
\begin{equation*}
\begin{split}
	I_1^+(\delta; x) & \lesssim \frac{A(x)^2}{x} \, F([x-\delta x, x]) 
	\underset{x\to\infty}{\sim}
	\frac{A(x)^2}{x} \, \left( \frac{1}{A((1-\delta)x)} - \frac{1}{A(x)}\right) \\
	& \underset{x\to\infty}{\sim}
	\frac{A(x)}{x} \, \left( \frac{1}{(1-\delta)^{\alpha}} - 1\right) 
	\underset{\delta\to 0}{=} \frac{A(x)}{x} \, O(\delta) \,.
\end{split}
\end{equation*}
This shows that $I_1^+(\delta; x)$ is a.n.,
hence the SRT holds by Theorem~\ref{th:main}.

\smallskip

Next we prove the second part of Theorem~\ref{th:1/2}:
we assume that condition \eqref{eq:cond} is not satisfied,
and we build a probability $F$ for which the SRT fails.
Since $A \in RV(\frac{1}{2})$,
we can write $A(x) = \ell(x) \sqrt{x}$ where $\ell$ is slowly varying.
By assumption, see \eqref{eq:cond}, there is a subsequence $x_n \to \infty$
such that $\sup_{1 \le s \le x_n} \ell(s) \gg \ell(x_n)$, hence
we can find $1 \le s_n \le x_n$ for which $\ell(s_n) \gg \ell(x_n)$.
We have necessarily $s_n = o(x_n)$, because
$\ell(s) / \ell(x_n) \to 1$ uniformly for $s \in [\epsilon x_n, x_n]$,
for any fixed $\epsilon > 0$,
by the uniform convergence theorem of slowly varying functions
\cite[Theorem~1.2.1]{cf:BinGolTeu}.
Summarizing:
\begin{equation} \label{eq:suma}
	x_n \to \infty \,, \qquad s_n = o(x_n) \,, \qquad
	\epsilon_n := \frac{\ell(x_n)}{\ell(s_n)} \to 0  \,.
\end{equation}

By Lemma~\ref{lem:uao} below, there is a probability $F$ on $(0,\infty)$,
which satisfies \eqref{eq:tail1}, such that
\begin{equation} \label{eq:uao0}
	F(\{x_n\}) \ge \, \frac{\epsilon_{n}}{A(x_{n})} \qquad 
	\text{for infinitely many } n \in \N \,.
\end{equation}
Since $A(x) = \ell(x) \sqrt{x}$, recalling \eqref{eq:I1},
for infinitely many $n\in\N$ we can write
\begin{equation*}
\begin{split}
	I_1^+(\delta; x_n + s_n)
	& \ge \frac{A(s_{n})^2}{s_{n}} \, F(\{ x_{n}\})
	\ge \ell(s_{n})^2 \, \frac{\epsilon_n}{A(x_{n})}
	= \frac{\ell(s_{n})}{\sqrt{x_{n}}}
	= \frac{1}{\epsilon_{n}} \, \frac{A(x_n)}{x_{n}}
	\gg \frac{A(x_n + s_n)}{x_{n} + s_n} \,,
\end{split}
\end{equation*}
where the last inequality holds because $\epsilon_n \to 0$ and
$x_n + s_n \sim x_n$, see \eqref{eq:suma}.
This shows that $I_1^+(\delta; x)$ is not a.n., hence the SRT
fails, by Theorem~\ref{th:main}. \qed

\subsection{Proof of Proposition~\ref{pr:main1}}
\label{sec:prmain1}
We claim that \eqref{eq:suff0} is equivalent to the following relation:
\begin{equation} \label{eq:suff0eq}
	F((x-y, x])
	\,\underset{x\to\infty}{=} \, O
	\bigg( \frac{1}{A(x)}\Big(\frac{y}{x}\Big)^{\gamma} \bigg) 
	\qquad \text{for any $y = y_x \ge 1$ with } y = o(x) \,.
\end{equation}
It is clear that \eqref{eq:suff0} implies \eqref{eq:suff0eq}.
On the other hand, if \eqref{eq:suff0} fails, there are sequences
$x_n \to \infty$, $C_n \to \infty$ and $y_n \in [1, \frac{1}{2} x_n]$ such that
\begin{equation}\label{eq:contradic}
	F((x_n-y_n, x_n]) > C_n \frac{1}{A(x_n)} \Big(\frac{y_n}{x_n} \Big)^{\gamma} \,.
\end{equation}
By extracting subsequences, we may assume that $\frac{y_n}{x_n} \to \rho \in [0, \frac{1}{2}]$.
If $\rho = 0$, then $y_n = o(x_n)$ and \eqref{eq:contradic} contradicts \eqref{eq:suff0eq}.
If $\rho > 0$, then \eqref{eq:contradic} contradicts \eqref{eq:tail1}, because it yields
\begin{equation*}
	\overline{F}(\tfrac{1}{2}x_n) \ge F((x_n-y_n, x_n]) > C_n \frac{1}{A(x_n)}
	\big( \rho^{\gamma} + o(1) \big) \gg \frac{1}{A(\frac{1}{2}x_n)} \,.
\end{equation*}

We first prove that relation \eqref{eq:suff0eq} for every $\gamma < 1-2\alpha$ is 
a necessary condition
for the SRT. We can assume that $\alpha < \frac{1}{2}$, because
for $\alpha = \frac{1}{2}$ we have $\gamma < 0$ and \eqref{eq:suff0eq} follows 
by \eqref{eq:tail1}.
If we restrict the integral \eqref{eq:I1} to $z \in [0,y)$, where $y = y_x = o(x)$,
for large $x$
we can bound
 $I_1^+(\delta; x) \gtrsim b_2(y) \, F((x-y,x])$
because $b_2(z) \in RV(2\alpha - 1)$ is asymptotically
decreasing.
 If the SRT holds, $I_1^+(\delta;x)$ is a.n.\ by Theorem~\ref{th:main},
hence $b_2(y) \, F((x-y,x]) = o(b_1(x))$, i.e.\
\begin{equation*}
	F((x-y,x]) \underset{x\to\infty}{=} o\bigg(\frac{1}{A(x)} \, \frac{b_2(x)}{b_2(y)}\bigg) \,,
	\qquad \text{for any $y = y_x \ge 1$ with } y_x = o(x) \,.
\end{equation*}
Since $b_2 \in RV(2\alpha - 1)$, it follows by Potter's bounds \eqref{eq:Potter}
that, for any given $\gamma < 1-2\alpha$, we have 
$\frac{b_2(x)}{b_2(y)} \lesssim_{\gamma} (\frac{y}{x})^{\gamma}$,
hence \eqref{eq:suff0eq} holds as claimed
(even with $o(\cdot)$ instead of $O(\cdot)$).

\smallskip

We now turn to the sufficiency part.
Let $F$ be a probability on $[0,\infty)$ 
which satisfies \eqref{eq:tail1}, with $\alpha \in (0,\frac{1}{2}]$,
such that
relation \eqref{eq:suff0} holds for some $\gamma > 1-2\alpha$
and $x_0, C < \infty$.
We prove that the SRT holds by showing that $\tilde I_1^+(\delta; x)$
defined in \eqref{eq:I1hat} is a.n., by Proposition~\ref{th:integr}.
Applying \eqref{eq:suff0} and recalling \eqref{eq:I1hat}, for
fixed $\delta \in (0, \frac{1}{2})$ and large $x$ we get
\begin{equation*}
	\tilde I_1^+(\delta; x)
	\le \frac{C}{A(x) \, x^{\gamma}} \int_1^{\delta x} 
	\frac{A(z)^2}{z^{2-\gamma}} \, \dd z
	\underset{\,x\to\infty\,}{\sim}
	\frac{C'}{A(x) \, x^{\gamma}} \, \frac{A(\delta x)^2}{(\delta x)^{1-\gamma}}
	\underset{\,x\to\infty\,}{\sim} C' \, \delta^{2\alpha - 1 + \gamma}
	\, \frac{A(x)}{x} \,,
\end{equation*}
where $C' := C / (2\alpha - 1 - \gamma)$ and
the first asymptotic equivalence holds by \eqref{eq:Kar1},
because $z \mapsto A(z)^2 / z^{2-\gamma}$ is regularly varying
with index $2\alpha - (2 - \gamma) > -1$ (since $\gamma > 1-2\alpha$).
This shows that $\tilde I_1^+(\delta; x)$ is a.n.\ and completes the proof
of Proposition~\ref{pr:main1}.\qed

\subsection{Proof of Proposition~\ref{th:integr}}

By \eqref{eq:I1hat} we can write
\begin{equation} \label{eq:I+alt}
\begin{split}
	\tilde I_1^+(\delta; x) & = \int_1^{\delta x} 
	\bigg( \int_{\R} \ind_{\{y \in [0, z)\}} \, F(x-\dd y) \bigg)
	\, \frac{b_2(z)}{z} \, \dd z \\
	& = \int_{y \in [0, \delta x)} F(x-\dd y)
	\bigg( \int_{1 \vee y}^{\delta x} \frac{b_2(z)}{z} \, \dd z \bigg) \,.
\end{split}
\end{equation}
We recall that $b_k$ is defined in \eqref{eq:b}.
Assume that $\alpha < \frac{1}{2}$.
Then the function $z \mapsto b_2(z)/z$ is regularly varying
with index $2\alpha-2 < -1$, hence by \eqref{eq:Kar2}, for $y \ge 0$ we can write
\begin{equation*}
	\int_{1 \vee y}^{\delta x} \frac{b_2(z)}{z} \, \dd z
	\le \int_{1 \vee y}^{\infty} \frac{b_2(z)}{z} \, \dd z
	\lesssim b_2(1 \vee y) \lesssim b_2(y) \,,
\end{equation*}
because for $0 \le y < 1$ we have $b_2(y)
\ge A(0)^2 > 0$, see \S\ref{eq:regvar}.
Recalling \eqref{eq:I1},
we have shown that $\tilde I_1^+(\delta; x) \lesssim I_1^+(\delta; x)$
when $\alpha < \frac{1}{2}$. Then, if $I_1^+(\delta; x)$ is a.n.,
also $\tilde I_1^+(\delta; x)$ is a.n..

We now work for $\alpha \le \frac{1}{2}$.
Let us restrict the outer integral in \eqref{eq:I+alt} to $y \in [0, \frac{\delta}{2}x)$, 
and the inner integral to $z \in [1 \vee y, 2 \vee 2y)$. For $y \ge 1$ we have
\begin{equation*}
	\int_{1 \vee y}^{2 \vee 2 y} \frac{b_2(z)}{z} \, \dd z =
	\int_{y}^{2 y} \frac{b_2(z)}{z} \, \dd z  \ge \frac{b_2(y)}{2y} (2y-y) 
	= \frac{1}{2} b_2(y) \,,
\end{equation*}
while for $0 \le y < 1$ we can write 
$\int_{1 \vee y}^{2 \vee 2 y} \frac{b_2(z)}{z} \, \dd z = \int_{1}^{2} \frac{b_2(z)}{z} \, \dd z
= C \gtrsim b_2(y)$. Overall, it follows from \eqref{eq:I+alt} that
$\tilde I_1^+(\delta; x) \gtrsim I_1^+(\frac{\delta}{2}; x)$. This completes the proof.\qed

\subsection{Proof of Proposition~\ref{th:suffrw}}

Assume that both $\tilde I_1(\delta;x)$ and $\tilde I_1^{*}(\delta; x)$ are a.n.,
see \eqref{eq:tildeI10} and \eqref{eq:I1rw*}. We first show that, for any $\eta \in (0,1)$,
\begin{equation}\label{eq:counterh2}
	\forall z\in\R, \ \forall \ell \in \N: \qquad
	\int_{|y| \le \eta |z|}
	F(-z+\dd y) \, \tilde b_{\ell+1}(z,y) \lesssim_{\eta} \, b_{\ell}(z) \,.
\end{equation}
Since $\tilde b_{\ell + 1}(z,y) \le A(z)^{\ell - 1} \tilde b_{2}(z,y)$
and $b_{\ell}(z) = A(z)^{\ell - 1} b_1(z)$, see \eqref{eq:btilde} and \eqref{eq:b},
it is enough to prove \eqref{eq:counterh2} for $\ell = 1$. 
Let us fix $0 < \delta_0 < \eta$.
For $|y| > \delta_0 |z|$ we can bound $\tilde b_{2}(z,y)
\lesssim \tilde b_{2}(z, \delta_0 z) \lesssim_{\delta_0} b_2(z)$
and $\int_{\delta_0 |z| < |y| \le \eta |z|} F(-z+\dd y) \lesssim_{\delta_0, \eta}
1/A(z)$. It remains to prove \eqref{eq:counterh2} for $\ell = 1$ and with
$\eta$ replaced by an arbitrary $\delta_0 > 0$. The left hand side
of \eqref{eq:counterh2} equals $\tilde I_1(\eta; z)$
for $z \ge 0$ and $\tilde I_1^{*}(\eta; -z)$ for $z \le 0$
(recall \eqref{eq:tildeI1}), which are a.n.\ by assumption, hence we can fix
$\eta = \delta_0 > 0$ small enough so that the inequality \eqref{eq:counterh2} holds
for $|z| > x_0$, for a suitable $x_0 \in (0,\infty)$.
Finally, for $|z| \le x_0$ both sides of \eqref{eq:counterh2} are uniformly
bounded away from $0$ and $\infty$, hence the inequality \eqref{eq:counterh2} still holds.

Observe that, for $|z| \le \eta |w|$, we can bound
$b_{\ell}(z) \lesssim_{\eta} \tilde b_{\ell}(\frac{1}{\eta} z, z)
\le \tilde b_{\ell}(w, z)$, so \eqref{eq:counterh2} yields
\begin{equation}\label{eq:counterh3}
	\forall z, w\in\R \text{ with } |z| \le \eta |w|, \ \forall \ell \in \N: \qquad
	\int_{|y| \le \eta |z|}
	F(-z+\dd y) \, \tilde b_{\ell+1}(z,y) \lesssim_{\eta} \, \tilde b_{\ell}(w,z) \,.
\end{equation}
If we plug this inequality into \eqref{eq:tildeIk}, we see that
$\tilde I_2(\delta, \eta; x) \lesssim_{\eta} \tilde I_{1}(\delta; x)$
and, similarly,
$\tilde I_k(\delta, \eta; x) \lesssim_{\eta} \tilde I_{k-1}(\delta, \eta; x)$
for any $k \ge 3$. Since $\tilde I_1(\delta; x)$ is a.n.\ by assumption,
it follows that $\tilde I_k(\delta, \eta; x)$ is a.n.\ for any $k \ge 2$,
hence the SRT holds by Theorem~\ref{th:mainrw}.

\smallskip

Finally, if relation \eqref{eq:suff0} holds both for $F$ and for $F^*$,
the same arguments as in the proof of Proposition~\ref{pr:main1},
see \S\ref{sec:prmain1}, show that both $\tilde I_1(\delta; x)$
and $\tilde I_1^*(\delta; x)$ are a.n..\qed

\subsection{Proof of Theorem~\ref{th:Levy}}
Since Stone's local limit theorem applies equally to L\'{e}vy processes, 
see \cite[Proposition~2]{BD97},
an argument similar to the random walk case (see Subsection~\ref{sec:refor})
shows that the SRT \eqref{x1} holds if and only if $\widehat T(\delta; x)$ is a.n., where
\begin{equation}\label{x4}
\widehat T(\delta; x) := \int_{0}^{\delta A(x)}
\P(X_{t}\in (x-h,x]) \,\dd t \,.  
\end{equation}%

Let $J_s := X_s - X_{s-}$ be the jump of $X$ at time $s > 0$.
If we write
\begin{equation*}
	X_{t}=X_{t}^{(1)}+X_{t}^{(2)} \,, \qquad
	\text{where} \qquad 
	X_{t}^{(1)} := \sum_{s\leq t}J_{s}
	\, \ind_{\{|J_{s}| \ge 1\}\text{ }} \,,
\end{equation*}
then $X^{(1)}$ and $X^{(2)}$ are independent L\'{e}vy processes.
\begin{itemize}
\item The process $X^{(1)}$ is compound Poisson: we can write $X^{(1)}_t
= S_{N_{\lambda t}}$, where $N = (N_t)_{t\ge 0}$ is a standard Poisson process, 
$S = (S_n)_{n\in\N_0}$ is a random walk with step distribution
$\P(S_1 \in \dd x) = F(\dd x)$ given in \eqref{eq:F0}, and $\lambda = \Pi(\R \setminus
(-1,1)) \in (0,\infty)$.

\item The process $X_{t}^{(2)}$ can be written as
$X_{t}^{(2)}=\sigma B_{t}+\mu t+M_{t},$ where $M$ is the martingale
formed from the compensated sum of jumps with modulus less than $1$. 
\end{itemize}
To complete the proof, we show that the SRT
holds for the random walk $S$ if and only if it holds for $X^{(1)}$ 
(step 1)
if and only if it holds for $X$ (step 2).

\smallskip

\noindent
\emph{Step 1.}
Since $X^{(1)}_t = S_{N_{\lambda t}}$, we have
\begin{equation*}
	\P(X^{(1)}_{t}\in (x-h,x]) = \sum_{n\in\N_0} e^{-\lambda t} \frac{(\lambda t)^n}{n!}
	\P(S_n \in (x-h,x]) \,.
\end{equation*}
Note that $\int_0^z e^{-\lambda t} \frac{(\lambda t)^n}{n!} \, \dd t
= \frac{1}{\lambda} \P(Z_{n,\lambda} \le z)$, where $Z_{n,\lambda}$ denotes a
random variable with a $\mathrm{Gamma}(n,\lambda)$ distribution.
Then the quantity $\widehat T(\delta; x) = \widehat T_{X^{(1)}}(\delta; x)$ for $X^{(1)}$ equals
\begin{equation} \label{eq:hatT1}
	\widehat 
	T_{X^{(1)}}(\delta; x) = \frac{1}{\lambda} \sum_{n\in\N_0} \P(Z_{n,\lambda} \le \delta A(x)) \,
	\P(S_n \in (x-h,x]) \,.
\end{equation}
For $n \le \lambda \delta A(x)$ we have $\P(Z_{n,\lambda} \le \delta A(x))
\ge \P(Z_{n,\lambda} \le  \frac{n}{\lambda}) \to \frac{1}{2}$ as $n\to\infty$,
by the central limit theorem (recall that
$Z_{n,\lambda} \sim \frac{1}{\lambda}(Y_1 + \ldots + Y_n)$, where $Y_i$ are i.i.d.\
$\mathrm{Exp}(1)$ random variables).
Denoting by $T_S(\delta; x)$ the quantity in 
\eqref{eq:SRTeq} for the random walk $S$, and restricting the sum in \eqref{eq:hatT1}
to $n \le \lambda \delta A(x)$, we get
\begin{equation*}
	\widehat T_{X^{(1)}}(\delta; x) \gtrsim T_S(\lambda \delta; x) \,.
\end{equation*}

To prove a reverse inequality,
we observe that for all $z \le \frac{1}{2} \frac{n}{\lambda}$
we can write, for $\epsilon > 0$,
\begin{equation*}
	\P(Z_{n,\lambda} \le z) \le e^{\epsilon \lambda z} \E[e^{-\epsilon \lambda Z_{n,\lambda}}]
	= \frac{e^{\epsilon \lambda z}}{(1+\epsilon)^n}
	\le \bigg( \frac{e^{\frac{1}{2}\epsilon}}{1+\epsilon} \bigg)^n
	\le e^{-c n} \,,
\end{equation*}
where the last inequality holds with $c = c_{\epsilon} > 0$, provided we fix $\epsilon > 0$
small. Then, splitting the sum in \eqref{eq:hatT1} according to $n \le 2 \lambda \delta A(x)$
and $n > 2 \lambda \delta A(x)$, we get
\begin{equation*}
	\widehat T_{X^{(1)}}(\delta; x) \le 
	\frac{1}{\lambda} \sum_{n\le 2 \lambda \delta A(x)} 	\P(S_n \in (x-h,x])
	+ \sum_{n > 2 \lambda \delta A(x)} e^{-cn}
	\lesssim T_S(2 \lambda \delta; x) + e^{-2 c \lambda \delta A(x)} \,.
\end{equation*}
These inequalities show that $\widehat T_{X^{(1)}}(\delta; x)$ is a.n.\ if and only if
$T_S(\delta;x)$ is a.n., that is, the SRT holds for $X^{(1)}$ if and only if it
holds for $S$.

\smallskip

\noindent
\emph{Step 2.}
Assume that $X = X^{(1)} + X^{(2)}$ and
the SRT holds for $X^{(1)}$, that is $\widehat T_{X^{(1)}}(\delta; x)$
is a.n.. Then, given $\varepsilon >0$, there are $\delta_{0},x_{0}$
such that, for all $0<\delta <\delta _{0}$,
\begin{equation}\label{eq:innint}
\forall y > x_0: \qquad 
\widehat T_{X^{(1)}}(\delta; y) =
\int_{0}^{\delta A(y)} \P(X_{t}^{(1)}\in (y-h,y]) \, \dd t \leq \varepsilon \, \frac{A(y)}{y}.
\end{equation}%
Let us now write
\begin{align*}
\int_{0}^{\delta A(x)}
\P(X_{t} &\in (x-h,x],\, X_{t}^{(2)}\leq x/2) \, \dd t \\
&=\int_{-\infty }^{x/2}
\P(X_{t}^{(2)}\in \dd z) \int_{0}^{\delta A(x)} \P(X_{t}^{(1)}\in
(x-z-h,x-z]) \, \dd t \,.
\end{align*}%
For $z \le x/2$ we can write $A(x) \le c A(x/2)$, for any $c > 2^{\alpha}$ and for large $x$.
Then the inner integral is bounded by 
$\widehat T_{X^{(1)}}(c\delta; x-z) \le \varepsilon \frac{A(x-z)}{x-z} \lesssim
\varepsilon \frac{A(x)}{x}$, by \eqref{eq:innint}. This shows that
\begin{equation} \label{eq:toghe}
	\int_{0}^{\delta A(x)}
	\P(X_{t} \in (x-h,x],\, X_{t}^{(2)}\leq x/2) \, \dd t  \,\lesssim \,
	\varepsilon \, \frac{A(x)}{x} \,.
\end{equation}

Note that $X^{(2)}$ has finite exponential moments, because its L\'evy measure 
$\Pi(\,\cdot\, \cap (-1,1))$ is compactly supported, hence
$\E[e^{|X_t^{(2)}|}] \le \E[e^{X_t^{(2)}}] + \E[e^{-X_t^{(2)}}] \le e^{ct}$
for a suitable $c \in (0,\infty)$.
This yields the exponential bound $\P(|X_{t}^{(2)}|>a)\leq e^{-a} \, e^{ct}$, for all
$a \ge 0$, hence
\begin{align*}
\int_{0}^{\delta A(x)}
\P(X_{t} &\in (x-h,x], \, X_{t}^{(2)}>x/2) \, \dd t \\
&\leq \int_{0}^{\delta A(x)} \P(X_{t}^{(2)}>x/2) \, \dd t
\lesssim e^{-x/2} \, e^{c\delta A(x)} \underset{\,x\to\infty\,}{=}
o\left( \frac{A(x)}{x}\right) \,.
\end{align*}%
Together with \eqref{eq:toghe}, this shows that $\widehat T_{X}(\delta;x)$ is a.n.,
that is the SRT holds for $X$.

If the SRT holds for $X$, to show that it holds for $X^{(1)}$
we can repeat the previous arguments switching $X$ and $X^{(1)}$ (no
special feature of $X^{(1)}$ was used in this step).
\qed

\smallskip

\section{Counterexamples}
\label{sec:examples}

In this section we prove Propositions~\ref{th:counter}
and~\ref{th:counterex}. We first develop some useful tools.

\subsection{Preliminary tools}

Let us describe a practical way to build 
counter-examples.

\begin{remark}\label{rem:counter}
Let us fix $A \in RV(\alpha)$.
Let $F_1$ be a probability on $(0,\infty)$ which satisfies
\begin{equation}\label{eq:slig}
	\overline{F_1}(x) 
	\underset{x\to\infty}{\sim} \frac{2}{A(x)} \,, \qquad \ \
	F_1((x-h, x]) \underset{x\to\infty}{=} O\bigg(\frac{1}{x A(x)}\bigg) \,,
	\quad \forall h > 0 \,.
\end{equation}
(For instance, fix $n_0 \in \N$ 
such that $c_1 := \sum_{n > n_0} \frac{2 \, \alpha}{n \, A(n)} < 1$
and define
$F_1(\{n_0\}) := 1-c_1$, $F_1(\{n\}) := \frac{2 \, \alpha}{n \, A(n)}$
for $n\in\N$ with $n > n_0$.)
Let $F_2$ be a probability on $(0,\infty)$ such that
\begin{equation}\label{eq:sat0}
	\overline{F_2}(x) 
	\underset{x\to\infty}{=} o\bigg(\frac{1}{A(x)}\bigg)  \,.
\end{equation}
If we define $F := \frac{1}{2}(F_1 + F_2)$, we obtain a new probability on $(0,\infty)$
which satisfies
\begin{equation}\label{eq:sat1}
	\overline{F}(x) 	\underset{x\to\infty}{\sim} \frac{1}{A(x)} \,, \qquad \ \
	F(x+I) \ge \frac{1}{2} F_2(x+I) \,.
\end{equation}
\end{remark}

\smallskip

Next we state a useful result.
To provide motivation, note that if $F$ satisfies \eqref{eq:tail1}, then necessarily
$F(x+I) = o(\frac{1}{A(x)})$ as $x\to\infty$ (because $\overline{F}(x-h)
\sim \overline{F}(x) \sim \frac{1}{A(x)}$).
Interestingly, this bound can
be approached as close as one wishes, in the following sense.

\begin{lemma}\label{lem:uao}
Fix two arbitrary
positive sequences $x_n \to \infty$ and $\epsilon_n \to 0$.
For any $A \in RV(\alpha)$, with $\alpha \in (0,1)$, there is a probability $F$ on $(0,\infty)$
satisfying \eqref{eq:tail1} such that
\begin{equation} \label{eq:uao}
	F(\{x_n\}) \ge \, \frac{\epsilon_{n}}{A(x_{n})} \qquad 
	\text{for infinitely many } n \in \N \,.
\end{equation}
\end{lemma}

\begin{proof}
Let us fix $A \in RV(\alpha)$.
By Remark~\ref{rem:counter}, it is enough
to build a probability $F_2$ on $(0,\infty)$,
supported on the sequence $\{x_n\}_{n\in\N}$, which satisfies
\eqref{eq:sat0} and
\begin{equation}\label{eq:tosat}
	F_2(\{x_n\}) \ge 2 \frac{\epsilon_n}{A(x_n)} \quad
	\text{for infinitely many } n \in \N \,.
\end{equation}
Then, if we define $F := \frac{1}{2}(F_1 + F_2)$, the proof is completed
(recall \eqref{eq:sat1}).

By assumption $x_n \to \infty$ and $\epsilon_n \to 0$, hence
we can fix a subsequence $(n_k)_{k\in\N}$ such that
\begin{equation} \label{eq:geome}
	\frac{\epsilon_{n_{k+1}}}{A(x_{n_{k+1}})} \le \frac{1}{2}
	\, \frac{\epsilon_{n_{k}}}{A(x_{n_{k}})} \,, \qquad \forall k \in \N \,.
\end{equation}
This ensures that $\sum_{k\in\N} \frac{\epsilon_{n_k}}{A(x_{n_k})} < \infty$ (the series
converges geometrically) and we fix $k_0 \in \N$ so that
$\sum_{k \ge k_0} \frac{\epsilon_{n_k}}{A(x_{n_k})} \le \frac{1}{2}$. We now define
$F_2$, supported on the set $\{x_{n_k}: \ k \ge k_0\}$, by
\begin{equation*}
	F_2(\{x_{n_k}\}) := c_2 \, \frac{\epsilon_{n_k}}{A(x_{n_k})}
	\quad \text{for } k \ge k_0 \,, \qquad
	\text{where} \quad
	c_2 := \Bigg( \sum_{k\ge k_0} \frac{\epsilon_{n_k}}{A(x_{n_k})} \Bigg)^{-1} \ge 2 \,.
\end{equation*}
In this way, \eqref{eq:tosat} is satisfied. It remains
to check that \eqref{eq:sat0} holds.
Given $x \in (0,\infty)$, if we set $\bar k := \min\{k \ge k_0: \ x_{n_k} > x\}$,
we can write
\begin{equation*}
	F_2((x,\infty)) =
	\sum_{k \ge \bar k} c_2 \, \frac{\epsilon_{n_k}}{A(x_{n_k})}
	\le c_2 \, \frac{\epsilon_{n_{\bar k}}}{A(x_{n_{\bar k}})}
	\sum_{k \ge \bar k} \frac{1}{2^{k-\bar k(x)}}
	\le  2 \, c_2 \,
	\frac{\epsilon_{n_{\bar k}}}{A(x)} \,,
\end{equation*}
where we used \eqref{eq:geome}, and
the last inequality holds because $x_{n_{\bar k}} > x$,
by definition of $\bar k$. Since $\epsilon_n \to 0$ by assumption, 
and $\bar k \to \infty$ as $x \to \infty$,
the proof is completed.
\end{proof}

\subsection{Proof of Proposition~\ref{th:counter}}
Let us fix $A \in RV(\alpha)$ with $\alpha \in (0,\frac{1}{2})$.
By Remark~\ref{rem:counter}, it is enough to build
a probability $F_2$ on $(0,\infty)$ which satisfies \eqref{eq:sat0} and moreover
\begin{equation} \label{eq:tosat2}
	F_2(x+I) = O\bigg(\frac{\zeta(x)}{x A(x)} \bigg) \,, \qquad
	I_1^+(\delta; x; F_2) \text{ is not a.n.} \,,
\end{equation}
where $I_1^+(\delta; x; F_2)$ denotes the quantity $I_1^+(\delta; x)$ in \eqref{eq:I1}
with $F$ replaced by $F_2$.
Once this is done,
we can set $F := \frac{1}{2}(F_1 + F_2)$ and the proof is completed
(recall \eqref{eq:sat1}).

\smallskip

By assumption $\zeta(\cdot)$ is non-decreasing with $\lim_{x\to\infty} \zeta(x) = \infty$.
Let us define $x_n := 2^n$, and fix $n_0 \in \N$ large enough so that
$\zeta(x_{n_0-1}) \ge 1$. Let us define
\begin{equation} \label{eq:zetan}
	z_n := \frac{1}{2} \frac{x_n}{\zeta(x_{n-1})^{1+\theta}} \,, \qquad
	\text{where $\theta > 0$ will be fixed later}\,.
\end{equation}
Note that $z_n \le \frac{1}{2} x_n$ for $n \ge n_0$ (because $\zeta(x_{n-1}) \ge 1$),
hence the intervals $(x_n - z_n, x_n]$ are disjoint.
We may also assume that $z_n \ge 1$, possibly
enlarging $n_0$ (if we decrease $\zeta(\cdot)$ we get a stronger statement,
so we can replace $\zeta(x)$ by $\min\{\zeta(x), \log x\}$, so that $z_n \to \infty$).

We define a probability $F_2$
supported on the set
$\bigcup_{n \ge n_0} (x_n - z_n, x_n]$,
with a constant density on each interval, as follows:
\begin{equation} \label{eq:F2}
	F_2(x_n - \dd s) := c \, \frac{\zeta(x_{n-1})}{x_n \, A(x_n)} \, 
	\ind_{[0, z_n]}(s) \, \dd s \,, 
	\qquad \forall n \ge n_0 \,,
\end{equation}
for a suitable $c \in (0,\infty)$.
We are going to show that $F_2$ is a finite measure,
so we can fix the constant $c$ to make it a probability.
Note that
\begin{equation*}
	F_2((x_n-z_n, x_n]) = c\, \frac{\zeta(x_{n-1})}{x_n \, A(x_n)} \, z_n
	= \frac{c}{2 \, A(x_n) \, \zeta(x_{n-1})^{\theta}} \,.
\end{equation*}
Since $A(x_n) = A(2 x_{n-1}) \sim 2^{\alpha} A(x_{n-1})$ as $n\to\infty$,
we may assume that $A(x_n) \ge 2^{\alpha/2} A(x_{n-1})$ for all $n \ge n_0+1$
(possibly enlarging $n_0$). Since $\zeta(x_{n-1})^{\theta} \ge
\zeta(x_{n-2})^{\theta}$, we obtain
\begin{equation*}
	F_2((x_n-z_n, x_n]) \le 2^{-\alpha/2} F_2((x_{n-1}-z_{n-1}, x_{n-1}]) \,,
	\qquad \forall n \ge n_0 + 1 \,.
\end{equation*}
It follows that, for every $n\ge n_0$,
\begin{equation*}
	\sum_{m \ge n} F_2((x_m-z_m, x_m]) \le
	F_2((x_n-z_n, x_n]) \sum_{m \ge n} (2^{-\alpha/2})^{m-n}
	= C \, F_2((x_n-z_n, x_n]) \,,
\end{equation*}
where $C := (1-2^{-\alpha/2})^{-1} < \infty$.
This shows that $F_2$ is indeed a finite measure.

For all large $x \in (0,\infty)$, we have
$x_{n-1} < x \le x_n$ for a unique $n \ge n_0$, hence
\begin{equation*}
	\overline{F_2}(x) \le \sum_{m\ge n} F_2((x_m-z_m, x_m])
	\le C \, F_2((x_n-z_n, x_n]) =
	\frac{c \, C}{2 \, A(x_n) \, \zeta(x_{n-1})^{\theta}}
	\underset{x\to\infty}{=} o\bigg(\frac{1}{A(x)}\bigg) \,,
\end{equation*}
so that \eqref{eq:sat0} holds. Similarly,
for $x_{n-1} < x \le x_n$ we can write, by \eqref{eq:F2},
\begin{equation*}
	F_2(x+I) = F_2((x-h,x]) \le c \, h \, \frac{\zeta(x_{n-1})}{x_n \, A(x_n)}
	\le c \, h \, \frac{\zeta(x)}{x \, A(x)} \,,
\end{equation*}
because both $\zeta(\cdot)$ and $A(\cdot)$ are non-decreasing, hence 
the first relation in \eqref{eq:tosat2} holds.
Finally, for fixed $\delta \in (0, \frac{1}{2})$, since $z_n \le \delta x_n$
for $n$ large enough, we have by \eqref{eq:Kar1}
\begin{equation*}
	I_1^+(\delta; x_n; F_2) = c \, \frac{\zeta(x_{n-1})}{x_n \, A(x_n)} 
	\int_{0 \le z \le z_n} \frac{A(z)^2}{z \vee 1} \, \dd z
	\underset{\,n\to\infty\,}{\sim} c \, \frac{\zeta(x_{n-1})}{x_n \, A(x_n)} \, A(z_n)^2 \,.
\end{equation*}
Recalling \eqref{eq:zetan}, we can apply Potter's bounds \eqref{eq:Potter}, since
$z_n \ge 1$, to get, for any $\epsilon > 0$,
\begin{equation*}
	I_1^+(\delta; x_n; F_2) \gtrsim_{\epsilon}
	\frac{\zeta(x_{n-1})}{x_n \, A(x_n)} \, 
	\frac{A(x_n)^2}{\zeta(x_{n-1})^{2(1+\theta)(\alpha+\epsilon)}}
	= \zeta(x_{n-1})^{1-2(1+\theta)(\alpha+\epsilon)} \, \frac{A(x_n)}{x_n}
	\gg \frac{A(x_n)}{x_n} \,,
\end{equation*}
where the last inequality holds provided we choose $\theta > 0$ and $\epsilon > 0$
small enough, depending only on $\alpha$, so that
$1-2(1+\theta)(\alpha+\epsilon) > 0$ (we recall that $\alpha < \frac{1}{2}$).
This shows that $I_1^+(\delta; x_n; F_2)$ is not a.n.\ and completes the proof.\qed

\subsection{Proof of Proposition~\ref{th:counterex}}

We fix $\alpha \in (0,\frac{1}{3})$
and choose for simplicity $A(x) := x^{\alpha}$.
We are going to build a probability $F$ on $\R$ which satisfies
\eqref{eq:tail2} with $p=q=1$, such that $\tilde I_1(\delta; x)$
is a.n.\ but $\tilde I_2(\delta,\eta; x)$ is not a.n., for any $\eta \in (0,1)$.
It suffices to show that $I_1(\delta; x)$
is a.n.\ but $I_2(\delta,\eta; x)$ is not a.n., thanks to
\eqref{eq:IItildele2} and \eqref{eq:IItildele}.

In analogy with Remark~\ref{rem:counter}, we fix a probability
$F_1$, \emph{this time on the whole real line $\R$}, which satisfies \eqref{eq:tail2} with $p=q=3$
and such that $F_1((x-h,x]) = O(\frac{1}{|x|A(x)})$ as $x \to \pm \infty$.
Then we define two probabilities $F_2, F_3$ on $(0,\infty)$
which both satisfy \eqref{eq:sat0}, and we set
\begin{equation} \label{eq:F123}
	F := \frac{1}{3}(F_1 + F_2 + F_3^*) \,,
\end{equation}
where $F_3^*(A) := F_3(-A)$ is the reflection of $F_3$
(so that it is 
a probability on $(-\infty, 0)$).
Clearly, \eqref{eq:tail2} holds for $F$ with $p=q=1$. 
It remains to build $F_2$ and $F_3$.

\smallskip

We are going to define
$F_2$ so that
\begin{equation}\label{eq:tossa}
	I_1(\delta; x; F_2) \text{ is a.n.} \,,
\end{equation}
(where $I_1(\delta; x; F_2)$ denotes the quantity in \eqref{eq:I1rw}
with $F$ replaced by $F_2$).
This implies that $I_1(\delta; x) = I_1(\delta; x; F)$ is a.n.,
because $I_1(\delta; x; F_1)$ is clearly a.n., while $F_3^*$ is supported on $(-\infty, 0)$
and gives no contribution.

We fix a parameter $p \in (1,\frac{1}{3\alpha})$.
We set $E_{n,k} := [2^n + 2^k, 2^n + 2^k +\frac{2^k}{2k^p})$ for $n \in \N$ with $n\ge 2$
and for $1 \le k \le n-1$. Note that
$E_{n,k} \subseteq [2^n + 2^k, 2^n + 2^{k+1})$ are disjoint intervals, and moreover
$\bigcup_{k=1}^{n-1} E_{n,k} \subseteq [2^n, 2^{n+1})$.
We define $F_2$ with a density, which is
constant in each interval $E_{n,k}$ (for $n \ge 2$ and $1 \le k \le n-1$)
and zero otherwise, given by
\begin{equation} \label{eq:F2last}
\begin{split}
	& F_2(2^n + 2^k + \dd w) := \frac{c}{\ell(n) \, (2^n)^{1-\alpha}} \, 
	\frac{1}{(2^k)^{2\alpha}} \, \ind_{[0, \frac{2^k}{2k^p})}(w) \, \dd w \,,
\end{split}
\end{equation}
where $c\in (0,\infty)$ is a suitable normalizing constant and we set for short
\begin{equation} \label{eq:ellen}
	\ell(n) := \log (1 + n ) \,.
\end{equation}
Note that
\begin{equation} \label{eq:F2Enk}
	F_2(E_{n,k}) = \frac{c}{\ell(n) \, (2^n)^{1-\alpha}} \, 
	\frac{1}{(2^k)^{2\alpha}} \, \frac{2^k}{2k^p}
	= \frac{c}{\ell(n) \, (2^n)^{1-\alpha}} \, \frac{(2^k)^{1-2\alpha}}{2 k^{p}} \,,
\end{equation}
hence
\begin{equation*}
\begin{split}
	F_2([2^n, 2^{n+1})) & = \sum_{k=1}^{n-1} F_2(E_{n,k}) 
	\le \frac{c}{\ell(n) \, (2^n)^{1-\alpha}}
	\sum_{k=1}^{n-1} (2^k)^{1-2\alpha}
	\lesssim \frac{c\, (2^n)^{1-2\alpha}}{\ell(n) \, (2^n)^{1-\alpha}} \\
	& = \frac{c }{\ell(n) \, (2^n)^{\alpha}}
	\underset{n\to\infty}{=} o\bigg(\frac{1}{(2^n)^{\alpha}}\bigg) \,.
\end{split}
\end{equation*}
Note that $F_2([2^n, 2^{n+1}))$ decreases exponentially fast in $n$, hence
for $x \in [2^n, 2^{n+1})$ we have 
$\overline{F_2}(x) \le \overline{F_2}(2^n)  \lesssim F_2([2^n, 2^{n+1})) = o(1/A(x))$,
which shows that \eqref{eq:sat0} is fulfilled. 
It remains to check \eqref{eq:tossa}. We do this by showing that,
for any $\delta < \frac{1}{4}$,
\begin{equation} \label{eq:goalI}
	I_1(\delta; x; F_2) 
	\underset{x\to\infty}{=} o\big( b_1(x) \big)
	= o\bigg(\frac{1}{x^{1-\alpha}} \bigg) \,.
\end{equation}
This is elementary but slightly technical, and it is shown below.

\smallskip

Finally, we define $F_3$,
We introduce the disjoint intervals $G_k := [2^k, 2^k + \frac{2^k}{k^p})$
for $k \ge 2$. We
let $F_3$ have a density, constant on every $G_k$ (for $k \ge 2$)
and zero otherwise, given by
\begin{equation} \label{eq:F3last}
	F_3(2^k + \dd z) := \frac{c'}{\ell(k)} \, \frac{k^p}{(2^k)^{1+\alpha}}
	\, \ind_{[0, \frac{2^k}{k^p})}(z) \, \dd z \,,
\end{equation}
where $c'\in (0,\infty)$ is a normalizing constant. Then
\begin{equation*}
	F_3([2^k, 2^{k+1})) =
	F_3(G_k) = \frac{c'}{\ell(k)} \, \frac{k^p}{(2^k)^{1+\alpha}}
	\, \frac{2^k}{k^p} = \frac{c'}{\ell(k)} \, \frac{1}{(2^k)^{\alpha}}
	= o\bigg(\frac{1}{(2^k)^{\alpha}}\bigg) \,.
\end{equation*}
Then
for $x \in [2^k, 2^{k+1})$ we have $	\overline{F_3}(x) \le \overline{F_3}(2^k) \lesssim  
\overline{F_3}([2^k, 2^{k+1})) = o(1/A(x))$ as $x \to \infty$,
hence \eqref{eq:sat0} holds. 
Given $\eta \in (0,1)$, fix $k_0 = k_0(\eta)$ large enough so that $\frac{1}{2k^p} < \eta$
for $k \ge k_0$. Then, recalling \eqref{eq:Ik}
and \eqref{eq:F123}, we can write
\begin{equation*}
\begin{split}
	I_2(\delta, & \eta; 2^n; F) 
	\ge \int_{0 \le y \le \delta 2^n}
	F(2^n +\dd y) \int_{|z| \le \eta y} F(-y+\dd z) \, b_3(z) \\
	& \ge \frac{1}{9} \sum_{k=k_0}^{\lfloor \log_2 (\delta 2^n) \rfloor} 
	\int_{0 \le w \le \frac{2^k}{2 k^p}}
	F_2(2^n + 2^k + \dd w) 
	\int_{0 \le z \le \frac{2^k}{2 k^p}}
	\, F_3\big( 2^k + w + \dd z \big) \, 
	(1 \vee z)^{3\alpha-1} \,.
\end{split}
\end{equation*}
Note that $(1 \vee z)^{3\alpha-1} \ge (\frac{2^k}{2 k^p})^{3\alpha-1}$ 
(we recall that $\alpha < \frac{1}{3}$)
and by \eqref{eq:F3last}
\begin{equation*}
	\forall 0 \le w \le \tfrac{2^k}{2 k^p}: \qquad
	F_3\big( 2^k + w + [0, \tfrac{2^k}{2 k^p}) \big) =
	\frac{c'}{\ell(k)} \, \frac{k^p}{(2^k)^{1+\alpha}}
	\, \frac{2^k}{2k^p} \gtrsim \frac{1}{\ell(k)} \, \frac{1}{(2^k)^{\alpha}} \,.
\end{equation*}
Since $\log_2 (\delta 2^n) = n + \log_2 \delta$, recalling
\eqref{eq:F2Enk}, we can write for large $n$
\begin{equation*}
\begin{split}
	I_2(\delta, \eta; 2^n; F) 
	& \gtrsim \sum_{k=k_0}^{n/2} F_2(E_{n,k}) \, \frac{1}{\ell(k)} \, \frac{1}{(2^k)^{\alpha}}
	\, \bigg( \frac{2^k}{2 k^p} \bigg)^{3\alpha-1} 
	\gtrsim \frac{1}{\ell(n) \, (2^n)^{1-\alpha}} \sum_{k=k_0}^{n/2}
	\frac{1}{\ell(k) \, k^{3\alpha p}} \,.
\end{split}
\end{equation*}
Since we have fixed $p < \frac{1}{3\alpha}$,
applying \eqref{eq:Kar1} and recalling \eqref{eq:ellen} we finally obtain
\begin{equation*}
	I_2(\delta, \eta; 2^n; F)  \gtrsim \frac{n^{1-3\alpha p}}{\ell(n)^2} \, 
	\frac{1}{(2^n)^{1-\alpha}} \gg \frac{1}{(2^n)^{1-\alpha}} = b_1(2^n) \,.
\end{equation*}
This shows that $I_2(\delta, \eta; x; F)$ is not a.n..\qed

\medskip
\noindent
\emph{Proof of \eqref{eq:goalI}.}
We recall that $F_2$ is supported
on the intervals $E_{n,k} := [2^n + 2^k, 2^n + 2^k +\frac{2^k}{2k^p})$ with $n\ge 2$
and $1 \le k \le n-1$. Let us set $E_n := \bigcup_{k=1}^{n-1} E_{n,k} \subseteq [2^n, 2^{n+1})$.

For large $x \ge 0$, we define $n\ge 2$ such that
$2^n \le x < 2^{n+1}$.
For $\delta \in (0,\frac{1}{4})$ 
and large $x$, the interval $(x-\delta x, x + \delta x)$ can intersect at most
$E_n$ and $E_{n+1}$ (because the rightmost point in $E_{n-1}$
is $2^{n-1} + 2^{n-2} + \frac{2^{n-2}}{2(n-2)^p} \sim \frac{3}{4} \, 2^{n}$ as $n\to\infty$).
Consequently we can write
\begin{equation} \label{eq:contrii}
	I_1(\delta; x; F_2) = \int_{|y| \le \delta x} F_2(x + \dd y) \, b_2(y)
	\le \int_{z \in E_n \cup E_{n+1}} F_2(\dd z) \, b_2(z-x) \,.
\end{equation}
For $z \in E_{n+1}$ we have
$z \in E_{n+1, k}$ for some $1 \le k \le n$, in which case
$z \ge 2^{n+1} + 2^k$. Since $x < 2^{n+1}$, we have
$|z-x| = z-x > 2^k$ which yields $b_2(z-x) \le b_2(2^k) = (2^k)^{2\alpha - 1}$.
Recalling \eqref{eq:F2Enk}, we see that
the contribution of $E_{n+1}$ to \eqref{eq:contrii} is at most
\begin{equation*}
	\sum_{k=1}^{n} \frac{c}{\ell(n) \, (2^{n+1})^{1-\alpha}} 
	\, \frac{(2^k)^{1-2\alpha}}{2 k^{p}} \, (2^k)^{2\alpha - 1}
	\lesssim \frac{1}{\ell(n) \, (2^{n+1})^{1-\alpha}} \, \sum_{k=1}^\infty
	\frac{1}{k^p} =
	o\bigg(\frac{1}{(2^{n+1})^{1-\alpha}}\bigg) \,,
\end{equation*}
since we chose $p > 1$.
This is $o\big(\frac{1}{x^{1-\alpha}}\big)$,
so it is negligible for \eqref{eq:goalI}.

Then we look at the contribution of $E_n$ to \eqref{eq:contrii}.
Assume first that $2^n + 2 \le x < 2^{n+1}$.
Then we can write $2^n + 2^{\tilde k} \le x < 2^n + 2^{\tilde k+1}$
for a unique $\tilde k \in \{1,\ldots, n-1\}$.
For $z \in E_{n}$ we have
$z \in E_{n, k}$ for some $1 \le k \le n-1$. We distinguish three cases.
\begin{itemize}
\item If $k \le \tilde k - 1$ (in particular, $\tilde k \ge 2$), then 
\begin{equation*}
	|z-x| = x-z \gtrsim (2^n + 2^{\tilde k}) -
	(2^n + 2^{\tilde k-1} + \tfrac{2^{\tilde k-1}}{2(\tilde k-1)^p}) \gtrsim 2^{\tilde k} \,,
\end{equation*}
hence $b_2(z - x) \lesssim b_2(2^{\tilde k}) = (2^{\tilde k})^{2\alpha-1}$.
By \eqref{eq:F2Enk},
the contribution to \eqref{eq:contrii} is at most
\begin{equation*}
	\sum_{k=1}^{\tilde k - 1} F_2(E_{n,k}) \, (2^{\tilde k})^{2\alpha-1}
	\le \frac{c}{\ell(n) \, (2^n)^{1-\alpha}} \, (2^{\tilde k})^{2\alpha-1}
	\sum_{k=1}^{\tilde k - 1} (2^k)^{1-2\alpha}
	\lesssim \frac{c}{\ell(n) \, (2^n)^{1-\alpha}} \,,
\end{equation*}
which is $o\big(\frac{1}{x^{1-\alpha}}\big)$, so it is negligible for \eqref{eq:goalI}.

\item If $k \ge \tilde k + 2$, then 
$|z-x| = z - x \gtrsim (2^n + 2^k) - (2^n + 2^{\tilde k + 1}) \ge 2^k - 2^{k-1} \gtrsim 2^k$,
hence $b_2(z - x) \lesssim b_2(2^{k}) = (2^{k})^{2\alpha-1}$ and we get
\begin{equation*}
	\sum_{k=\tilde k + 1}^{n - 1} F_2(E_{n,k}) \, (2^{k})^{2\alpha-1}
	\le \frac{c}{\ell(n) \, (2^n)^{1-\alpha}} \,
	\sum_{k= \tilde k + 1}^\infty \frac{1}{2k^p}
	\lesssim \frac{c}{\ell(n) \, (2^n)^{1-\alpha}} \,,
\end{equation*}
because $p > 1$, hence this contribution is also negligible for \eqref{eq:goalI}.

\item If $k \in \{\tilde k, \tilde k + 1\}$, then $|z-x| \le
2^{\tilde k + 2} - 2^{\tilde k} = 3 \cdot 2^{\tilde k}$.
By \eqref{eq:F2last}, since the density of $F_2$ is larger in
$E_{n, \tilde k}$ than in $E_{n, \tilde k + 1}$,
we see the contribution to \eqref{eq:contrii} is at most
\begin{equation*}
	\frac{c}{\ell(n) \, (2^n)^{1-\alpha}}
	\, \frac{1}{(2^{\tilde k})^{2\alpha}}
	\int_{|w| \le 3 \cdot 2^{\tilde k}} (|w| \vee 1)^{2\alpha-1} \, \dd w
	\lesssim \frac{1}{\ell(n) \, (2^n)^{1-\alpha}} \,,
\end{equation*}
which is negligible for \eqref{eq:goalI}.
\end{itemize}

Finally, the regime $2^n \le x < 2^{n} + 2$ is treated similarly.
For $z \in E_{n,k}$, we distinguish the cases $k \ge 2 $ and $k = 1$.
If we set $\tilde k := 0$, the estimates in the two cases 
$k \ge \tilde k + 2$ and $k \in \{\tilde k, \tilde k + 1\}$ treated above
apply with no change.\qed

\smallskip

\appendix

\section{Some technical results}

\label{sec:app}

Let us fix a probability $F$ on $\mathbb{R}$ which satisfies %
\eqref{eq:tail1} with $\alpha \in (0, \frac{1}{2}]$ and with $p, q > 0$.
The next Lemmas show some relations
between the quantities $I_k$ and $\tilde I_k$
defined in \eqref{eq:I1rw}, \eqref{eq:Ik} and in
\eqref{eq:tildeI1}, \eqref{eq:tildeIk}, respectively.
We recall that $\kappa_{\alpha} \in \N$ is defined in \eqref{eq:kappaalpha}.

\begin{lemma}\label{th:corequiv1}
Fix $\eta \in (0,1)$.
If $\tilde I_{k}(\delta,\eta; x)$ is a.n.\ for $k=\kappa_{\alpha}$,
then it is a.n.\ for all $k\in\N$.
\end{lemma}

\begin{lemma}\label{th:corcascade}
Assume $\frac{1}{\alpha} \not\in \N$
and fix $\eta \in (0,1)$.
If $I_{k}(\delta ,\eta ;x)$ is a.n.\ for $k=\kappa_{\alpha}$,
then it is a.n. for all $k \in \N$.
\end{lemma}

\begin{lemma}\label{th:corequiv2}
With no restriction on $\alpha$,
if $\tilde I_{\kappa_{\alpha}}(\delta,\eta; x)$
is a.n., then also $I_{\kappa_{\alpha}}(\delta,\eta; x)$ is a.n..
The reverse implication holds if $\frac{1}{\alpha} \not\in \N$
(but not necessarily if $\frac{1}{\alpha} \in \N$).
\end{lemma}

\begin{proof}[Proof of Lemma~\protect\ref{th:corcascade}]
Fix $k\in \mathbb{N}$ with $k\ge 2$ and $\eta \in (0,1)$.
We are going to prove the following relations:
\begin{align}
	& \text{if } k < \tfrac{1}{\alpha}-1: \quad \ 
	I_{k-1}(\delta ,\eta ;x)\lesssim _{\eta }I_{k}(\delta ,\eta ;x)\,, 
	\label{eq:Ikk-1}\\
	& \text{if } k > \tfrac{1}{\alpha}-1: \quad \ 
	I_{k}(\delta ,\eta ;x)\lesssim _{\eta }I_{k-1}(\delta ,\eta ;x)\,.
	\label{eq:Ik-1k}
\end{align}%
Since we assume that
$\frac{1}{\alpha} \not\in \N$, we have $\frac{1}{\alpha}
- 2 < \kappa_{\alpha} < \frac{1}{\alpha} - 1$.
If $I_{\kappa_{\alpha}}$ is a.n., it follows that also
$I_{\kappa_{\alpha}-1}$, $I_{\kappa_{\alpha}-2}$, $\ldots$ are a.n., by \eqref{eq:Ikk-1},
while $I_{\kappa_{\alpha}+1}, I_{\kappa_{\alpha}+2}, \ldots$ are a.n., by \eqref{eq:Ik-1k}.

\smallskip

It remains to prove \eqref{eq:Ikk-1}-\eqref{eq:Ik-1k}.
For $k < \frac{1}{\alpha} - 1$, the function $b_{k+1}(y)$, see \eqref{eq:b}, is regularly varying
with index $(k+1)\alpha -1<0$, hence it is asymptotically decreasing:
bounding $b_{k+1}(y_{k})\gtrsim b_{k+1}(\eta y_{k-1})$ for $|y_{k}|\leq \eta
|y_{k-1}|$ gives
\begin{equation*}
\int_{|y_{k}|\leq \eta |y_{k-1}|}F(-y_{k-1}+\mathrm{d}y_{k})\,b_{k+1}(y_{k})%
\gtrsim F(-(1-\eta )|y_{k-1}|)\,b_{k+1}(\eta y_{k-1})\gtrsim _{\eta
}b_{k}(y_{k-1})\,,
\end{equation*}%
which plugged into \eqref{eq:Ik} shows that $I_{k}(\delta ,\eta ;x)\gtrsim
_{\eta }I_{k-1}(\delta ,\eta ;x)$, proving \eqref{eq:Ikk-1}. 

To prove \eqref{eq:Ik-1k}, note that
for $\alpha >%
\frac{1}{k+1}$ the function $b_{k+1}(y)$ is asymptotically increasing:
by a similar argument, we get $I_{k}(\delta ,\eta ;x) \lesssim_{\eta }
I_{k-1}(\delta ,\eta ;x)$, that is \eqref{eq:Ik-1k}.
\end{proof}

\begin{proof}[Proof of Lemma~\ref{th:corequiv2}]
We are going to prove the following inequalities between
$I_1$ and $\tilde I_1$:
\begin{gather}
	\label{eq:IItildele1}
	I_1(\tfrac{\delta}{2}; x) \lesssim \tilde I_1(\delta; x) \,,
	\\
	\label{eq:IItildele2}
	\text{if } \alpha < \tfrac{1}{2}: \qquad
	\tilde I_1(\delta; x) \lesssim I_1(\delta; x)  \,.
\end{gather}
For $k\in\N$ with $k\ge 2$,
we have the following relations between $\tilde I_k$
and $I_k, I_{k-1}$:
\begin{gather} 
\label{eq:IItildele}
\max\big\{ I_{k-1}(\delta,\eta; x), \,
I_{k}(\delta, \eta; x) \big\} 
\lesssim_{\eta} \tilde I_k(\delta, \eta; x) \,, \\
\label{eq:IItildeeq}
\text{if } k \ne \tfrac{1}{\alpha}-1: \qquad
\tilde I_{k}(\delta, \eta; x) 
\lesssim_{\eta} 
\max\big\{ I_{k-1}(\delta,\eta; x), \, I_{k}(\delta, \eta; x) \big\} \,.
\end{gather}
Given these relations,
if $\tilde I_{\kappa_{\alpha}}$ is a.n.,
then also $I_{\kappa_{\alpha}}$ is a.n.: it suffices to apply \eqref{eq:IItildele1}
and \eqref{eq:IItildele} with $k=\kappa_{\alpha}$.
When $\frac{1}{\alpha} \not\in \N$,
the reverse implication also holds,
because we can apply \eqref{eq:IItildele2} if $\kappa_{\alpha} = 1$
(note that $\alpha < \frac{1}{2}$, since $\frac{1}{\alpha} \not\in \N$)
or \eqref{eq:IItildeeq} if $\kappa_{\alpha} > 1$.

\smallskip

It remains to prove \eqref{eq:IItildele1}-\eqref{eq:IItildeeq}.
By \eqref{eq:btildealt}, for $|y_{k}|\leq \eta |y_{k-1}|$
with $\eta \in (0,1)$ we can write
\begin{equation} \label{eq:ssttaa}
\tilde{b}_{k+1}(y_{k-1},y_{k})\geq \sum_{m=A(|y_{k}|)}^{A(|y_{k}|/\eta )}%
\frac{m^{k}}{a_{m}}\gtrsim \frac{A(|y_{k}|)^{k}}{(|y_{k}|/\eta)\vee 1}\,\big(%
A(|y_{k}|/\eta )-A(|y_{k}|)\big)\gtrsim_{\eta} b_{k+1}(y_{k})\,,
\end{equation}%
The same arguments show that, for $|y| \le \frac{\delta}{2} x$, we 
have $\tilde b_2(\delta x,y) \ge \tilde b_2(2y,y) \gtrsim b_2(y)$.
Plugging these bounds into \eqref{eq:tildeI1}
and \eqref{eq:tildeIk} proves \eqref{eq:IItildele1} and also
$\tilde I_k(\delta,\eta; x) \gtrsim_{\eta} 
I_k(\delta,\eta; x)$, which is half of \eqref{eq:IItildele}. 
For the other half, note that for $|y_{k}|\leq \eta |y_{k-1}|$,
always by \eqref{eq:btildealt},
\begin{equation*}
\tilde{b}_{k+1}(y_{k-1},y_{k})
\gtrsim \sum_{m=A(\eta|y_{k-1}|)}^{A(|y_{k-1}| )}%
\frac{m^{k}}{a_{m}}\gtrsim \frac{A(\eta|y_{k-1}|)^{k}}{|y_{k-1}|\vee 1}\,\big(%
A(|y_{k-1}|)-A(\eta|y_{k-1}|)\big)\gtrsim_{\eta} b_{k+1}(y_{k-1})\,,
\end{equation*}%
hence
\begin{equation} \label{eq:inan}
	\int\limits_{|y_{k}|\leq \eta |y_{k-1}|}
	\!\!\!\!\!\!\!\! F(-y_{k-1}+\mathrm{d}y_{k})\,
	\tilde b_{k+1}(y_{k-1},y_{k})%
	\gtrsim_{\eta} b_{k+1}(y_{k-1})
	F(-(1-\eta )|y_{k-1}|) \gtrsim_{\eta}b_{k}(y_{k-1})\,.
\end{equation}
From \eqref{eq:tildeI1} we get $\tilde I_k(\delta,\eta; x) 
\gtrsim_{\eta} I_{k-1}(\delta,\eta; x)$,
which completes the proof of \eqref{eq:IItildele}.

Next we prove \eqref{eq:IItildele2} and \eqref{eq:IItildeeq}.
We distinguish two cases.
\begin{itemize}
\item If $k < \frac{1}{\alpha} - 1$, the sequence $m^{k}/a_{m}$ is regularly varying with
index $k-\frac{1}{\alpha }<-1$. By \eqref{eq:Kar2}, we can write
\begin{equation*}
\tilde{b}_{k+1}(y_{k-1},y_{k})\leq \sum_{m=A(|y_{k}|)}^{\infty }\frac{m^{k}}{%
a_{m}}\lesssim \frac{A(|y_{k}|)^{k+1}}{|y_{k}|\vee 1}=b_{k+1}(y_{k})\,,
\end{equation*}%
which yields $\tilde I_k(\delta,\eta; x) \lesssim_{\eta} I_{k}(\delta,\eta; x)$.
For $k=1$, we have proved \eqref{eq:IItildele2}, since
$k < \frac{1}{\alpha} - 1$ means precisely $\alpha < \frac{1}{2}$, while
for $k \ge 2$ we have proved half of \eqref{eq:IItildeeq}.

\item If $k > \frac{1}{\alpha} - 1$, with $k \ge 2$,
the sequence $m^{k}/a_{m}$ is regularly varying with
index $k-\frac{1}{\alpha } > -1$ and by \eqref{eq:Kar1} we get
\begin{equation*}
\tilde{b}_{k+1}(y_{k-1},y_{k})\leq \sum_{m=1}^{A(|y_{k-1}|)}\frac{m^{k}}{%
a_{m}}\lesssim \frac{A(|y_{k-1}|)^{k+1}}{|y_{k-1}|\vee 1}=b_{k+1}(y_{k-1})\,,
\end{equation*}%
and in analogy with \eqref{eq:inan} we get 
$\tilde I_k(\delta,\eta; x) \lesssim_{\eta} I_{k-1}(\delta,\eta; x)$.
Relation \eqref{eq:IItildeeq} is proved.
\end{itemize}
The proof is completed.
\end{proof}

\begin{proof}[Proof of Lemma~\ref{th:corequiv1}]
Fix $k\in \mathbb{N}$ with $k\ge 2$ and $\eta \in (0,1)$.
In analogy with \eqref{eq:Ikk-1}-\eqref{eq:Ik-1k}, we are going to prove that
\begin{align}
	& \text{if } k \le \tfrac{1}{\alpha}-1: \quad \ 
	\tilde I_{k-1}(\delta ,\eta ;x) \lesssim_{\eta }
	\tilde I_{k}(\delta ,\eta ;x)\,, 
	\label{eq:Ikk-1tilde}\\
	& \text{if } k > \tfrac{1}{\alpha}-1: \quad \ 
	\tilde I_{k}(\tfrac{\delta}{2} ,\eta ;x)
	\lesssim_{\eta } \tilde  I_{k-1}(\delta ,\eta ;x)\,,
	\label{eq:Ik-1ktilde}
\end{align}%
where $k=\frac{1}{\alpha}-1$ is included in \eqref{eq:Ikk-1tilde}
(unlike \eqref{eq:Ikk-1}).
Since we assume that $\tilde I_{\kappa_{\alpha}}$ is a.n.,
and since $\kappa_{\alpha} \le \frac{1}{\alpha}-1$, we can apply 
\eqref{eq:Ikk-1tilde} iteratively to see that 
$\tilde I_{\kappa_{\alpha}-1}$, $\tilde I_{\kappa_{\alpha}-2}$, \ldots\ are a.n..
Similarly, since $\kappa_{\alpha} + 1 > \frac{1}{\alpha}-1$, relation
\eqref{eq:Ik-1ktilde} shows that
$\tilde I_{\kappa_{\alpha}+1}$, $\tilde I_{\kappa_{\alpha}+2}$, \ldots\ are a.n..

\smallskip

It remains to prove \eqref{eq:Ikk-1tilde}-\eqref{eq:Ik-1ktilde}.
By \eqref{eq:Ikk-1} and \eqref{eq:Ik-1k} we have
\begin{equation}\label{eq:maxI}
	\max\big\{ I_{k-1}(\delta,\eta; x), \, I_{k}(\delta, \eta; x) \big\} \\
	\approx_{\eta} \begin{cases}
	I_{k}(\delta, \eta; x) & \text{if } k < \tfrac{1}{\alpha}-1 \\
	\rule{0pt}{1.2em}I_{k-1}(\delta, \eta; x) & \text{if } k > \tfrac{1}{\alpha}-1 
\end{cases} \,.
\end{equation}
Let us fix $k \le \frac{1}{\alpha} - 1$ and assume first that $k \ge 3$.
By \eqref{eq:IItildeeq} and \eqref{eq:maxI}
(with $k-1$ in place of $k$; note that $k-1 < \frac{1}{\alpha}-1$), and then
by \eqref{eq:IItildele},
we have
\begin{equation*}
	\tilde I_{k-1} \lesssim_{\eta} \max\{I_{k-2}, I_{k-1}\}
	\approx_{\eta} I_{k-1} \le \max\{I_{k-1}, I_{k}\} \lesssim_{\eta} \tilde I_k \,.
\end{equation*}
If $k=2$, 
the assumption $k \le \frac{1}{\alpha} - 1$ means $\alpha \le \frac{1}{3}$,
hence we can apply \eqref{eq:IItildele2} followed by \eqref{eq:IItildele} to get
$\tilde I_1 \lesssim I_1 \le \max\{I_1, I_2\} \lesssim_{\eta} \tilde I_2$.
This completes the proof of \eqref{eq:Ikk-1tilde}.

Fix now $k > \frac{1}{\alpha}-1$ (note that $k \ge 2$, since $\alpha \le \frac{1}{2}$).
By \eqref{eq:IItildeeq} and \eqref{eq:maxI}, we can write
\begin{equation*}
	\tilde I_k(\tfrac{\delta}{2}) \lesssim_{\eta} 
	\max\{I_{k-1}(\tfrac{\delta}{2}), I_k(\tfrac{\delta}{2})\} \lesssim_{\eta}
	I_{k-1}(\tfrac{\delta}{2}) \,.
\end{equation*}
If $k \ge 3$, we apply \eqref{eq:IItildele} with $k-1$ in place of $k$, to get
\begin{equation*}
	I_{k-1}(\tfrac{\delta}{2}) \le \max\{I_{k-2}(\tfrac{\delta}{2}),
	I_{k-1}(\tfrac{\delta}{2})\} \lesssim_{\eta} \tilde I_{k-1}(\tfrac{\delta}{2})
	\le \tilde I_{k-1}(\delta) \,.
\end{equation*}
This yields $\tilde I_k(\tfrac{\delta}{2}) \lesssim_{\eta} \tilde I_{k-1}(\delta)$,
which is precisely \eqref{eq:Ik-1ktilde}.
If $k = 2$, we apply \eqref{eq:IItildele1} to see that
$I_{k-1}(\tfrac{\delta}{2}) = I_{1}(\tfrac{\delta}{2}) \le \tilde I_{1}(\delta)$. 
This completes the proof of \eqref{eq:Ik-1ktilde}.
\end{proof}


\end{document}